\documentclass{amsart}
\usepackage{amscd}
\usepackage{amssymb}
\input xypic
\xyoption{all}
\newdir{ >}{{}*!/-5pt/@{>}}

\newcommand{\ul}{\underline}

\newcommand{\idem}{\mathrm{Idem}}
\newcommand{\assem}{\mathrm{Assembly}}

\newcommand{\con}{\mathrm{con}}
\newcommand{\proj}{\mathrm{proj}}

\def\cA{\mathcal A}
\def\cR{\mathcal R}
\def\cC{\mathcal C}

\def\xc{\mathcal{HC}}

\def\ob{\mathrm{ob}}
\def\arr{\mathrm{ar}}
\def\inc{\mathrm{inc}}
\def\obc{\ob\cC}
\def\cD{\mathcal D}

\def\cE{\mathcal E}

\def\fF{\mathfrak{F}}

\def\cF{\mathcal F}
\def\fA{\mathfrak A}

\def\fall{\mathcal All}
\def\cG{\mathcal G}

\def\ring{\mathrm{Rings}}
\def\rings{\mathrm{Rings}}
\def\cat{\mathrm{Cat}}
\def\inj{\mathrm{inj}}
\def\spt{\mathrm{Spt}}

\def\sets{\mathfrak Sets}
\def\gpd{\mathfrak Gpd}
\def\fS{\mathbb{S}}
\def\bH{\mathbb{H}}
\def\ab{\mathfrak Ab}
\def\ind{\mathrm{Ind}}
\def\inc{\mathrm{inc}}
\def\comp{\mathrm{Comp}}
\def\bind{\mathrm{Big}\ind}
\def\Ar{\mathrm{Ar}}
\def\Ob{\mathrm{Ob}}

\def\res{\mathrm{Res}}
\def\St{\mathrm{St}}
\def\li{\mathrm{Link}}
\def\cSt{\overline{\St}}

\def\supp{\mathrm{supp}}

\def\equ{\mathrm{eq}}
\def\coequ{\mathrm{coeq}}
\def\Top{\mathrm{Top}}
\def\Sing{\mathrm{Sing}}
\def\dom{\mathrm{dom}}
\def\cod{\mathrm{cod}}
\def\inf{\mathrm{inf}}
\def\ninf{\mathrm{ninf}}
\def\nil{\mathrm{nil}}
\def\nat{\natural}
\def\ev{\mathrm{ev}}

\newcommand{\C}{\mathbb{C}}

\newcommand{\Q}{\mathbb{Q}}
\newcommand{\R}{\mathbb{R}}
\newcommand{\F}{\mathbb{F}}

\renewcommand{\cR}{\mathcal{R}}
\newcommand{\Z}{\mathbb{Z}}
\newcommand{\N}{\mathbb{N}}

\newcommand{\ho}{\mathrm{Ho}}
\newcommand{\hofi}{\mathrm{hofiber}}

\newcommand{\fC}{\mathfrak{A}}

\newcommand{\sS}{\mathbb{S}}
\newcommand{\org}{\mathrm{Or}G}
\newcommand{\orfg}{\mathrm{Or}_{\cF}G}
\def\bu{\bullet}
\def\diag{\operatorname{diag}}

\def\colim{\operatornamewithlimits{colim}}

\def\End{\operatorname{End}}

\def\map{\operatorname{map}}
\def\lra{\longrightarrow}
\def\onto{\twoheadrightarrow}
\def\iso{\stackrel{\cong}\lra}
\def\triqui{\vartriangleleft}
\def\weq{\overset\sim\lra}
\def\lweq{\overset\sim\longleftarrow}
\def\fibeq{\overset\sim\onto}
\def\tor{\mathrm{Tor}}
\def\tot{\mathrm{Tot}}

\numberwithin{equation}{section}
\theoremstyle{plain}
\newtheorem{thm}[equation]{Theorem}
\newtheorem{cor}[equation]{Corollary}
\newtheorem{lem}[equation]{Lemma}
\newtheorem{prop}[equation]{Proposition}

\newcommand{\comment}[1]{}  %to comment out chunks of text

\theoremstyle{definition}

\theoremstyle{remark}
\newtheorem{rem}[equation]{Remark}
\newtheorem{ex}[equation]{Example}
\newtheorem{nota}[equation]{Notation}
\newtheorem{para}[equation]{Equivariant homology}
\newtheorem{stan}[equation]{Standing Assumptions}
\newtheorem{sect}[equation]{Sectional Assumptions}

\begin{document}

\bibliographystyle{plain}

\title{Isomorphism conjectures with proper coefficients}
\author{Guillermo Corti\~nas}
\email{gcorti@dm.uba.ar}\urladdr{http://mate.dm.uba.ar/\~{}gcorti}
\address{Dep. Matem\'atica-IMAS, FCEyN-UBA\\ Ciudad Universitaria Pab 1\\
1428 Buenos Aires\\ Argentina}
\author{Eugenia Ellis}
\email{eugenia@cmat.edu.uy}
\address{CMAT, Facultad de Ciencias-UDELAR\\
Igu\'a $4225$, 11400 Montevideo, Uruguay}
\thanks{Corti\~nas was supported by CONICET; both authors were partially supported by
the MathAmSud network U11MATH-05 and by grants UBACyT 20020100100386, and MTM2007-64704 (FEDER
funds).}
\begin{abstract}
Let $G$ be a group and let $E$ be a functor from small $\Z$-linear categories to spectra. Also let $A$ be a ring with a $G$-action. Under mild conditions on $E$ and $A$ one can define an equivariant homology theory of $G$-simplicial sets
$H^G(-,E(A))$ with the property that if $H\subset G$ is a subgroup, then
\[
H^G_*(G/H,E(A))=E_*(A\rtimes H)
\] 
If now $\cF$ is a nonempty family of subgroups of $G$, closed under conjugation and under subgroups, then there
is a model category structure on $G$-simplicial sets such that a map $X\to Y$ is a weak equivalence (resp. a fibration) if and only 
if $X^H\to Y^H$ is an equivalence (resp. a fibration) for all $H\in\cF$. The strong isomorphism conjecture for the
quadruple $(G,\cF,E,A)$ asserts that if $cX\to X$ is the $(G,\cF)$-cofibrant replacement then
\[
H^G(cX,E(A))\to H^G(X,E(A))
\]
is an equivalence. The isomorphism conjecture says that this holds when $X$ is the one point space, in which case
$cX$ is the classifying space $\cE(G,\cF)$. In this paper we introduce an algebraic notion of $(G,\cF)$-properness for $G$-rings, modelled on the analogous notion for $G$-$C^*$-algebras, and show that the strong $(G,\cF,E,P)$ isomorphism conjecture for $(G,\cF)$-proper $P$ is true in
several cases of interest in the algebraic $K$-theory context. Thus we give a purely algebraic, discrete counterpart to a result of Guentner, Higson and Trout in the $C^*$-algebraic case. We apply this to show that under rather general hypothesis, the assembly map $H_*^G(\cE(G,\cF),E(A))\to E_*(A\rtimes G)$ can be identified with the boundary map in the long exact sequence of $E$-groups associated to certain exact sequence of rings. Along the way we prove several results on excision in algebraic $K$-theory and cyclic homology which are of independent interest.
\end{abstract}
\maketitle
\section{Introduction}
Let $G$ be a group; a \emph{family} of subgroups of $G$ is a nonempty family $\cF$ closed under conjugation and under
taking subgroups. If $\cF$ is a family of subgroups of $G$,  then a $G$-simplicial set $X$ is called a {\em $(G,\cF)$-complex} if the stabilizer of every simplex of $X$ is in $\cF$. The category of $G$-simplicial sets can be equipped with a closed model structure
where an equivariant map $X\to Y$ is a weak equivalence (resp. a fibration) if $X^H\to Y^H$ is a weak equivalence
(resp. a fibration) for every $H\in\cF$ (see Section \ref{sec:model}); $(G,\cF)$-complexes are the cofibrant objects in this model
structure (Remark \ref{rem:cofichar}). By a general construction of Davis and L\"uck (see 
\cite{dl}) any functor $E$ from the category $\Z-\cat$ of small $\Z$-linear categories to the category $\spt$ of spectra which sends category equivalences to equivalences of spectra gives rise to an equivariant homology theory of $G$-spaces
$X\mapsto H^G(X,E(R))$ for each unital ring $R$ with a $G$-action (unital $G$-ring, for short), such that if 
$H\subset G$ is a subgroup, then
\begin{equation}\label{intro:cross}
H_*^G(G/H,E(H))=E_*(R\rtimes H)
\end{equation}
is just $E_*$ evaluated at the crossed product. The \emph{strong isomorphism conjecture} for the quadruple $(G,\cF,E,R)$
asserts that $H^G(-,E(R))$ sends $(G,\cF)$-equivalences to weak equivalences of spectra. The strong isomorphism conjecture is equivalent to the assertion that for every $G$-simplicial set $X$ the map
\begin{equation}\label{intro:preass}
H^G(cX,E(R))\to H^G(X,E(R))
\end{equation}
induced by the $(G,\cF)$-cofibrant replacement $cX\to X$ is a weak equivalence. The weaker \emph{isomorphism conjecture} is the particular case when $X$ is a point; it asserts that if 
$\cE(G,\cF)\fibeq pt$ is the cofibrant replacement then the map
\begin{equation}\label{intro:ass}
H^G(\cE(G,\cF),E(R))\to H^G(pt,E(R))
\end{equation}
called the \emph{assembly map}, is an equivalence of spectra. This formulation of the conjecture is equivalent to that of Davis-L\"uck, (\cite{dl}) which is given in terms of topological spaces (see Proposition \ref{prop:equimodel} and paragraph \ref{para:equihom}). 

In this paper we are primarily concerned with the strong isomorphism conjecture for nonconnective algebraic $K$-theory --denoted $K$ in this paper-- homotopy algebraic $K$-theory $KH$, and Hochschild and cyclic homology $HH$ and $HC$.
Our main results are outlined in Theorem \ref{intro:main} below. First we need to explain the terms ``excisive" and ``proper" appearing in the theorem. Let $E:\ring\to \spt$ be a functor; we say that a not necessarily
unital ring $A$ is $E$-\emph{excisive} if whenever $A\to R$ is an embedding
of $A$ as a two sided ideal in a unital ring $R$, the sequence
\[
E(A)\to E(R)\to E(R/A)
\]
is a homotopy fibration. Unital rings are $E$-excisive for all functors $E$ considered in Theorem \ref{intro:main};
thus the theorem remains true if ``unital" is substituted for ``excisive". By a result of Weibel \cite{kh}, Homotopy algebraic $K$-theory satisfies excision; this means that every ring is $KH$-excisive. The rings which are excisive with respect to cyclic and Hochschild homology are the same; they were characterized by Wodzicki in \cite{wodex}, where he coined the term $H$-unital for such rings. By results of Suslin and Wodzicki, a ring is excisive for rational $K$-theory if and only if it is $H$-unital (see \cite{qs} for the if part and \cite{wodex} for the only if part); $K$-excisive rings were characterized by Suslin in \cite{sus}. Under mild assumptions on $E$ (the Standing Assumptions \ref{stan}), which are satisfied by all the examples considered in Theorem \ref{intro:main}, one can make sense of $H^G(-,E(A))$ for not necessarily unital, $E$-excisive $A$ (see Section \ref{sec:r&c}). The ring $\Z^{(X)}$ of polynomial functions on a locally finite simplicial set $X$ which are supported on a finite simplicial subset, and the ring $C_{\rm comp}(|X|,\F)$ of compactly supported continuous functions with
values in $\F=\R,\C$ are unital if and only if $X$ is finite, and are $E$-excisive for all $X$ and all the functors $E$ of Theorem \ref{intro:main}; they are $(G,\cF)$-proper whenever $X$ is a $(G,\cF)$-complex. In general if $X$ is a locally finite simplicial set with a $G$-action and $A$ is
a $G$-ring, then $A$ is called \emph{proper} over $X$ if it carries a $\Z^{(X)}$-algebra structure which is compatible with the action of $G$ and satisfies $\Z^{(X)}\cdot A=A$. We say that $A$ is $(G,\cF)$-\emph{proper} if it is proper over a $(G,\cF)$-complex. 

\begin{thm}\label{intro:main}
Let $G$ be a group, $\cF$ a family of subgroups, $E:\Z-\cat\to\spt$ a functor, and $P$ an $E$-excisive $G$-ring. The strong isomorphism conjecture for the quadruple $(G,\cF,E,P)$ is satisfied in each of the following cases.

\item[i)] $E=HH$ or $HC$ and $\cF$ contains all the cyclic subgroups of $G$.

\item[ii)] $E=KH$ and $P$ is $(G,\cF)$-proper.

\item[iii)] $E=K$ and $P$ is proper over a $0$-dimensional $(G,\cF)$-space.

\item[iv)] $E=K$, $\cF$ contains all the cyclic subgroups of $G$ and $P$ is a $(G,\cF)$-proper
$\Q$-algebra.

\item[v)] $E=K\otimes\Q$, $\cF$ contains all the cyclic subgroups of $G$ and $P$ is $(G,\cF)$-proper.

\end{thm}

Part i) of the theorem for unital rings is Proposition \ref{prop:asshh}; that it holds for all $HC$-excisive rings follows from this by Corollary \ref{cor:assnuni} and Proposition \ref{prop:hhstands}. Even for unital rings, part i) generalizes a result of L\"uck and Reich \cite{LR1}, who proved it under the additional assumption that $G$ acts trivially on $A$. Theorem \ref{thm:propern} proves that part ii) holds for any functor $E:\Z-\cat\to\spt$ satisfying
certain properties, including excision; the fact that $KH$ satisfies them is the subject of Section \ref{sec:kh}. We prove in Theorem \ref{thm:proper0} that 
part iii) of the theorem holds for any $E$ satisfying the standing assumptions; that they hold for $K$-theory
is established in Proposition \ref{prop:kstands}. Parts iv) and v) are the content of Theorem \ref{thm:main}.\\

The concept of properness used in this article is a discrete, algebraic translation of the analogous concept of proper
$G$-$C^*$-algebra. By a result of Guentner, Higson and Trout, the full $C^*$-crossed product version of the Baum-Connes conjecture with coefficients holds whenever the coefficient algebra is a proper $G$-$C^*$-algebra \cite{ght}. This result is a basic
fact behind the Dirac-dual Dirac method that was used, for example, in the proof of the Baum-Connes conjecture for
a-$T$-menable groups \cite{hk}. It is also at the basis of recent work of Meyer and Nest (\cite{ralf},\cite{mn1},\cite{mn2}) in which the conjecture and the Dirac method are recast in terms of triangulated categories. We expect that Theorem \ref{intro:main} can similarly be used as a tool in proving instances of the isomorphism conjecture for (homotopy) algebraic $K$-theory.   As a first application of Theorem \ref{intro:main} we prove the following theorem, which identifies the assembly map \eqref{intro:ass} as the connecting map in an excision
sequence.

\begin{thm}\label{intro:dirac}
Let $G$ be a group and $\cF$ a family of subgroups. Then there is a functor which assigns to each $G$-ring $A$
a $G$-ring $\fF^\infty A=\fF^\infty(\cF,A)$ equipped with an exhaustive filtration by $G$-ideals $\{\fF^nA:n\ge 0\}$, and a natural
transformation $A\to \fF^0A$, which, if $E$ is as in Theorem \ref{intro:main} and $A$ is $E$-excisive, have the following properties.

\item[i)] The map $E(A\rtimes G)\to E(\fF^0A\rtimes G)$ is an equivalence.

\item[ii)] The following sequence is a homotopy fibration
\[
E(\fF^0A\rtimes G)\to E(\fF^\infty A\rtimes G)\to E((\fF^\infty A/\fF^0 A)\rtimes G)
\]
In particular there is a map
\[
\partial:\Omega E((\fF^\infty A/\fF^0 A)\rtimes G)\to E(\fF^0A\rtimes G)
\]
\item[iii)] There is an equivalence
\[
H^G(\cE(G,\cF),E(A))\weq \Omega E((\fF^\infty A/\fF^0 A)\rtimes G)
\]
which makes the following diagram commute up to homotopy
\[
\xymatrix{
H^G(\cE(G,\cF),E(A))\ar[d]^\wr\ar[r]^(.6){\rm Assembly}&E(A\rtimes G)\ar[d]^\wr\\
\Omega E((\fF^\infty A/\fF^0 A)\rtimes G)\ar[r]^(.6)\partial& E(\fF^0A\rtimes G)}
\]
\end{thm}

The theorem above holds more generally for any functor satisfying certain hypothesis, listed in \ref{stan} and \ref{secstan}; see Proposition \ref{prop:dirac} and Theorem \ref{thm:assbound}. 

We also prove a number of results about $K$-excisive and $H$-unital rings which are needed for the proof of the theorems above; they are summarized in the following theorem.

\begin{thm}\label{intro:exci}
\item[i)] If $A$ is a $K$-excisive (resp. $H$-unital) $G$-ring, then $A\rtimes G$ is $K$-excisive (resp. $H$-unital).
\item[ii)] Let $\{A_i\}$ be a family of rings and let $A=\bigoplus_iA_i$ their direct sum, with coordinate-wise product.
Then $A$ is $K$-excisive (resp. $H$-unital) if and only if each $A_i$ is.
\item[iii)] If $A$ and $B$ are $K$-excisive rings, and at least one of them is flat as a $\Z$-module, then $A\otimes B$
is $K$-excisive.
\end{thm}

Part i) of Theorem \ref{intro:exci} results by combining Propositions \ref{prop:crossbar} and \ref{prop:crossh}. Part ii) follows from Propositions \ref{prop:barsum} and \ref{prop:barsumh}. Part iii) is Proposition \ref{prop:tenso}. The analogue of part iii) for $H$-unital rings is true without flatness assumptions, and was proved by Suslin and Wodzicki in \cite[Theorem 7.10]{qs}. \\

The rest of this paper is organized as follows. In Section \ref{sec:model} we formulate the isomorphism conjectures in terms of closed model categories. If $G$ is a group, $\cF$ a family of subgroups and $\cC$ is either the category $\Top$ of topological spaces or the category $\fS$ of simplicial sets, we introduce closed model structures on the equivariant category $\cC^G$ in which an equivariant map $X\to Y$ is a weak equivalence (resp. a fibration) if $X^H\to Y^H$ is one for every $H\in\cF$. We show in Proposition \ref{prop:equimodel} that the realization and singular functors give a Quillen equivalence between $\fS^G$ and $\Top^G$. In Section \ref{sec:r&c} we give a list of five basic conditions for a functor $E:\Z-\cat\to\spt$, the Standing Assumptions \ref{stan}; all functors $E$ considered
in the paper satisfy them. All but one of these conditions refer to needed permanence properties of $E$-excisive rings; thus they concern only the restriction of $E$ to $\ring$. The remaining condition is that for all $\cC\in\Z-\cat$ there must
be an equivalence
\begin{equation}\label{intro:weqac}
E(\cA(\cC))\weq E(\cC)
\end{equation}
Here 
\[
\cA(\cC)=\bigoplus_{x,y\in \cC}\hom_\cC(x,y)
\]
is the arrow ring. The assignment $\cC\to \cA(\cC)$ is functorial only for functors which are injective on objects;
likewise the equivalence \eqref{intro:weqac} is only required to be natural with respect to such functors. These conditions imply, for example, that $E$ sends naturally equivalent functors to homotopy equivalent maps of spectra (Lemma \ref{lem:enat}), and that $H^G(X,E(-))$ maps extensions of $E$-excisive rings to homotopy fibrations (Proposition \ref{prop:exciequi1}). We also discuss a fully functorial construction $\Z-\cat\to \ring$, $\cC\to\cR(\cC)$ which comes with a map $p:\cR(\cC)\to \cA(\cC)$ and give condtions on $E$ under which $E(p)$ is an equivalence for all 
$\cC$ (Lemma \ref{lem:eac=ebc}); they apply, for example, when $E=KH$, but fail for $E=K$(see Example \ref{ex:acoprdb}).
In Section \ref{sec:kth} we present the model for the (nonconnective) $K$-theory spectrum that we use in this article --essentially borrowed from Pedersen-Weibel's paper \cite{pw2}-- and prove (Proposition \ref{prop:kstands}) that it satisfies the standing assumptions. For this we need several properties of $K$-excisive rings which are proved in the appendix (including those listed as parts i) and ii) of Theorem \ref{intro:exci}). Section \ref{sec:kh} concerns Weibel's homotopy $K$-theory; the fact that it satisfies the standing assumptions is proved in Proposition \ref{prop:homostand}. We also show (Proposition \ref{prop:homob}) that there is a natural equivalence $KH(\cC)\to KH(\cR(\cC))$ $(\cC\in\Z-\cat)$. The basic definitions of Hochschild and cyclic homology for rings and $\Z$-linear categories
are reviewed in Section \ref{sec:hc}, where it is shown (Proposition \ref{prop:hhstands}) that they satisfy the standing assumptions. Part i) of Theorem \ref{intro:main} is proved in Section \ref{sec:asshh} (Proposition \ref{prop:asshh}). In the next section we discuss various Chern characters connecting $K$-theory with cyclic homology. Of these, the relative character 
\[
\nu:K^\nil(\cC)\otimes\Q=\hofi(K(\cC)\to KH(\cC))\to \Omega^{-1}|HC(\cC)|\otimes\Q
\]
(defined in \eqref{map:nu})) plays a prominent role in the article. Here $|-|$ is the spectrum associated by the Dold-Kan correspondence. We show in Proposition \ref{prop:ninfexci} that its fiber 
\begin{equation}\label{intro:kninf}
K^\ninf(\cC)=\hofi(\nu)
\end{equation}
satisfies the standing assumptions, that in addition it is excisive and that $K_*^\ninf$ commutes with filtering colimits. Section \ref{sec:zx} reviews some of the properties of the ring $\Z^{(X)}$ of finitely supported, integral polynomial functions on a simplicial set $X$. For example, $\Z^{(-)}$ is functorial for proper maps, and sends disjoint unions to direct sums (see Subsection \ref{subsec:rings}). Moreover, if $X$ is locally finite, and $Y\subset X$ is a subobject, then the the restriction map $\Z^{(X)}\to\Z^{(Y)}$ is onto (Corollary \ref{cor:sursupp}).
 We also show that if $X$ is locally finite, then $\Z^{(X)}$ is free as an abelian group (see Lemma \ref{lem:zxfree}) and that if $E$ satisfies the standing assumptions then the ring $\Z^{(X)}$ is $E$-excisive (Proposition \ref{prop:tfp}). Thus by Theorem \ref{intro:exci} iii), the class of $K$-excisive rings is closed under tensoring with $\Z^{(X)}$ (Proposition \ref{prop:exiax}).
In Section \ref{sec:propring} we consider $G$-rings which are proper over a $(G,\cF)$-complex $X$. We establish discrete analogues of several of the properties of proper $C^*$-algebras discussed in \cite{ght}. For a subgroup $H\subset G$ we introduce the induction functor $\ind_H^G:H-\rings\to G-\rings$ (Subsection \ref{subsec:ind}) and show that it is an equivalence between
$H-\rings$ and the full subcategory of those $G$-rings which are proper over the $0$-dimensional simplicial set $G/H$ (Proposition \ref{prop:indcomp}). Next we give a discrete variant of Green's imprimitivity theorem; we show in Theorem \ref{thm:git} that there
is an isomorphism
\begin{equation}\label{intro:git}
\ind_H^G(A)\rtimes G\cong M_{G/H}(A\rtimes H)
\end{equation}
Here $M_{G/H}$ denotes matrices indexed by $G/H\times G/H$ with finitely many nonzero coefficients. Also in this section we consider
the restriction functor $\res_G^H$ going from $G$-rings to $H$-rings and study the composites $\ind_H^G\res_G^H$ and
$\res_G^H\ind_K^G$ for subgroups $K,H\subset G$ (Lemmas \ref{lem:indtriv} and \ref{lem:indxtheta}).
The material in Section \ref{sec:propring} is used in the next section to define, for a group $G$, a subgroup $K\subset G$, a $G$-simplicial set $X$,
a functor $E:\Z-\cat\to\spt$ satisfying the standing assumptions, and a $K$-ring $A$, an induction map
\[
\ind:H^K(X,E(A))\to H^G(X,E(\ind_K^G(A)))
\]
We show in Proposition \ref{prop:equindiso} that the map above is an equivalence. Then we use this result to prove
part iii) of Theorem \ref{intro:main} for any functor satisfying assumptions \ref{stan}; see Theorem \ref{thm:proper0}. The latter theorem is applied in Section \ref{sec:dirac}, where
Theorem \ref{intro:dirac} is proved for any $E$ satisfying assumptions \ref{stan} and \ref{secstan} (see Proposition 
\ref{prop:dirac} and Theorem \ref{thm:assbound}). In Section \ref{sec:main} we begin by proving part ii) of Theorem
\ref{intro:main} for any functor $E$ satsifying excision in addition to the hypothesis of \ref{stan} and \ref{secstan} (see Theorem \ref{thm:propern}). In particular, it holds when $E$ is the functor $K^{\ninf}$ of \eqref{intro:kninf}. Parts iv) and v) of Theorem \ref{intro:main} are the content of Theorem \ref{thm:main}). The proof uses part i) of Theorem \ref{intro:main}, and Theorem \ref{thm:propern} applied to $K^\ninf$.
In the Appendix we recall the results of Suslin and Wodzicki on $K$-excisive and $H$-unital rings, and establish Theorem \ref{intro:exci} (see Propositions \ref{prop:barsum}, \ref{prop:barsumh}, \ref{prop:tenso}, \ref{prop:crossbar} and \ref{prop:crossh}).

\smallskip

\begin{nota}
If $\cC$ is a (small) category, we write $\obc$ for the (small) set
of objects and $\arr \cC$ for that of arrows. We often consider a
set $X$ as a discrete category, whose only arrows are the identity
maps. In particular, we do this when $X=\obc$; note that there is a
faithful functor $\obc\to \cC$.\goodbreak

We write $\sS$ for the category of simplicial sets
and $\Top$ for
that of topological spaces. A \emph{family} $\cF$ of subgroups of a
group $G$ is a nonempty family closed under conjugation and under
taking subgroups. We write $\orfg$ for the orbit category relative
to the family $\cF$; its objects are the $G$-sets $G/H$, $H\in\cF$;
its homomorphisms are the $G$-equivariant maps. If $C$ and $D$ are
categories, we write $C^D$ for the category of functors $D\to C$,
where the homomorphisms are the natural transformations. In
particular $\Top^G$ and $\sS^G$ are the categories of $G$-spaces and
$G$-simplicial sets, and $\Top^{{\orfg}^{op}}$ and
$\sS^{{\orfg}^{op}}$ those of contravariant $\orfg$-spaces and
$\orfg$-simplicial sets. If $f:C\to C'$ is a functor, we write
$f_*:C^D\to C'^D$ for the functor $g\mapsto f\circ g$. Thus for
example $|\ \ |_*:\sS^G\to\Top^G$ is the equivariant geometric
realization functor; this notation is used in Section
\ref{sec:model}. In the rest of the paper, if $C$ is a chain complex
of abelian groups,  $|C|$ is the spectrum the Dold-Kan
correspondence associates to it. Topological spaces are considered
briefly in Section \ref{sec:model} where it is explained that we can
equivalently work with simplicial sets, which is what we do in the
rest of the paper. In particular --except briefly in Section
\ref{sec:model}-- a spectrum is a sequence $\{{}_nE\}$ of pointed simplicial
sets and bonding maps $\Sigma {}_nE\to {}_{n+1}E$. If $E,F:C\to
\spt$  are functorial spectra, then by a (natural) \emph{map}
$f:E\weq F$ we mean a zig-zag of natural maps
\[
E=Z_0\overset{f_1}\longrightarrow Z_1\overset{f_2}\longleftarrow
Z_2\overset{f_3}\longrightarrow\dots  Z_n=F
\]
such that each right to left arrow $f_i$ is an object-wise weak
equivalence. If also the left to right arrows are object-wise weak
equivalences, then we say that $f$ is a \emph{weak equivalence} or
simply an \emph{equivalence}. If $\{E_i\}$ is a family of spectra,
we write $\bigoplus_iE_i$ for their wedge or coproduct.
\goodbreak

Rings in this paper are not assumed unital, unless explicitly
stated. We write $\ring$ for the category of rings and ring
homomorphisms, and $\ring_1$ for the subcategory of unital rings and
unit preserving homomorphisms. We use the letters $A,B$ for rings,
and $R,S$ for unital rings. If $V$ is an abelian group, then the
tensor algebra of $V$ is $TV=\bigoplus_{n\ge 1}V^{\otimes n}$; thus
for us $TV$ is nonunital. If $V$ is free, then $TV$ is a free
nonunital ring. If $X$ is a set, then $M_X$ is the ring of all
matrices $(z_{x,y})_{x,y\in X\times X}$ with integer coefficients,
only finitely many of which are nonzero. If $A$ is a ring, then
$M_XA=M_X\otimes A$; in particular $M_X\Z=M_X$. If $\{A_i\}$ is a family
of rings, then $\bigoplus_iA_i$ is their direct sum as abelian groups,
equipped with coordinate-wise multiplication.
\end{nota}

\section{Model category structures and assembly
maps}\label{sec:model}
\numberwithin{equation}{section} We begin
with some general considerations on model category structures for
diagrams of spaces.

We consider $\Top$ and $\sS$ with their usual, cofibrantly generated closed model structures.
If $C=\Top$, $\sS$, and $I$ is any small category, then, by \cite[Thm. 11.6.1]{hirsch}, $C^I$
is again a cofibrantly generated closed model category, with object-wise fibrations and weak equivalences, and where generating (trivial) cofibrations are
of the form
\[
\coprod_{\hom_I(\alpha,-)}f:\coprod_{\hom_I(\alpha,-)}\dom f\to\coprod_{\hom_I(\alpha,-) }\cod f
\]
with $\alpha\in I$ and $f:\dom f\to \cod f$ a generating (trivial)
cofibration in $C$. Recall that the geometric realization functor
$|\ \ |:\sS\to\Top$ and its right adjoint $\Sing:\Top\to\sS$ form a
Quillen equivalence. Hence by \cite[Thm. 11.6.5]{hirsch}, the
induced functors $|-|_*:\sS^I\rightleftarrows \Top^I:\Sing_*$ are
Quillen equivalences too.

Next fix a group $G$ and a family $\cF$ of subgroups of $G$. By the previous discussion applied to the orbit
category $\orfg^{op}$, we have a Quillen equivalence

\begin{equation}\label{map:quilleneqorg}
\xymatrix{\Top^{{\orfg}^{op}}\ar@/^/[rr]^{\Sing_*}&&\sS^{{\orfg}^{op}}\ar@/^/[ll]^{|\ \ |_*}}
\end{equation}

For $C=\Top,\sS$, consider the functor
\[
R:C^G\to C^{{\orfg}^{op}},\quad R(X)(G/H)=\map_G(G/H,X)=X^H
\]
and its left adjoint, the coend
\[
L:C^{{\orfg}^{op}}\to C^{G},\quad L(Y)=\int^{\org} Y(G/H)\times G/H
\]
The Quillen equivalence \eqref{map:quilleneqorg} fits into a diagram
\begin{equation}\label{diag:quilleneqs}
\xymatrix{\Top^{{\orfg}^{op}}\ar@/_/[dd]_{L}\ar@/^/[rr]^{\Sing_*}&&\sS^{{\orfg}^{op}}\ar@/_/[dd]_{L}\ar@/^/[ll]^{|\ \ |_*}\\
\\
\Top^G\ar@/^/[rr]^{\Sing_*}\ar@/_/[uu]_{R}&&\sS^G\ar@/^/[ll]^{|\ \ |_*}\ar@/_/[uu]_{R}}
\end{equation}

\begin{lem}\label{lem:fixed}
Let
\[
\xymatrix{B\ar[r]& Y\\
A\ar@{ >-}[u]^i\ar[r]& X\ar[u]}
\]
be a cocartesian diagram of $G$-sets. Assume that $i$ is injective. Then
\[
\xymatrix{B^G\ar[r]& Y^G\\
A^G\ar@{ >-}[u]^i\ar[r]& X^G\ar[u]
}
\]
is again cocartesian.
\end{lem}
\begin{proof} Straightforward.
\end{proof}
\begin{prop}\label{prop:equimodel} Let $\cC=\Top$, $\sS$.
\item[i)] $\cC^G$ is a closed model category where a map $f$ is a fibration (resp. a weak equivalence) if and only if $R(f)$ is. Moreover $\cC^G$ is cofibrantly generated, where the generating (trivial) cofibrations
are the maps $f\times id:\dom f\times G/H\to \cod f\times G/H$, with $f$ a generating (trivial) cofibration
and $H\in\cF$.
\item[ii)] Each of the pairs of functors of diagram \eqref{diag:quilleneqs} is a Quillen equivalence.
\end{prop}
\begin{proof}
One can give conditions on two sets of maps and a subcategory of a category $\cD$ to be respectively the generating cofibrations, generating trivial cofibrations and weak equivalences in a closed model structure of $\cD$; see M. Hovey's book \cite[Thm. 2.1.19]{hov}. It is straightforward that those conditions are satisfied in our case, for $\cD=\cC^G$. This proves i). The top pair of functors in diagram
\eqref{diag:quilleneqs} is a Quillen equivalence by the discussion above the proposition. By definition of fibrations and weak equivalences in $\cC^G$, these are both preserved and reflected by $R$. In particular
$(L,R)$ is a Quillen pair. To show that it is an equivalence, it suffices, by \cite[Cor. 1.3.16]{hov}, to show that if $X\in \cC^{\orfg^{op}}$ is cofibrant, then the unit map
\begin{equation}\label{unit}
X\to RLX
\end{equation}
is a weak equivalence; in fact we shall see that it is an
isomorphism. Because every cofibrant object is a retract of a
cofibrant cell complex, it suffices to check that \eqref{unit} is an
isomorphism on cell complexes. Because the unit map preserves the
skeletal filtration, it suffices to check that $X^n\to RLX^n$ is an
isomorphism for all $n$. By definition, the generating cofibrant
cells in $\cC^{\orfg^{op}}$ are of the form
$\coprod_{\map_G(-,G/H)}\Delta^n$. But for every $T\in\fS$, we have:
\begin{align*}
RL(\coprod_{\map_G(-,G/H)}T)(G/K)=&R(G/H\times T)(G/K)\\
=&(G/H\times T)^K\\
=&\map_{\org}(G/K,G/H)\times T=\coprod_{\map_G(-,G/H)}T
\end{align*}
Thus the unit map is an isomorphism on cells, and therefore on
coproducts of cells, since taking fixed points under a subgroup
preserves coproducts of $G$-simplicial sets. In particular
\eqref{unit} is an isomorphism on the zero skeleton of $X$. Assume
by induction that \eqref{unit} is an isomorphism on the
$n$-skeleton. The $n+1$-skeleton is a pushout

\[
\xymatrix{\coprod_{H\in I_n}\coprod_{\map_G(-,G/H)}\Delta^{n}\ar[r]& X^{n+1}(-)\\
          \coprod_{H\in I_n}\coprod_{\map_G(-,G/H)}\partial\Delta^{n}\ar[u]\ar[r]&
          X^n(-)\ar[u]}
\]
Applying $L$ to this diagram yields a cocartesian diagram with
injective vertical maps. Hence by Lemma \ref{lem:fixed} and the
inductive hypothesis, the diagram
\[
\xymatrix{\coprod_{H\in I_n}(\map_G(-,G/H))\times \Delta^{n}\ar[r]& RLX^{n+1}(-)\\
          \coprod_{H\in I_n}(\map_G(-,G/H))\times \partial\Delta^{n}\ar[r]\ar[u]& RLX^n(-)\ar[u]}
\]
is again a pushout. It follows that $RLX^{n+1}\cong X^{n+1}$ and
thus \eqref{unit} is an isomorphism on all cell complexes, as we had
to prove. We have shown that the top horizontal and both vertical
pairs of functors are Quillen equivalences; by \cite[Cor.
1.3.15]{hov}, this implies that also the bottom pair is a Quillen
equivalence.
\end{proof}

\begin{rem}\label{rem:cofichar}
An object of a cofibrantly generated category is cofibrant if
and only if it is a retract of a cellular complex built from generating cofibrant cells.
In the case of $\fS^G$, every object is built from cells of the form $\Delta^n\times G/H$
for $H\subset G$ a subgroup; it is cofibrant for the model structure of Proposition
\ref{prop:equimodel} if and only if all such cells have $H\in\cF$.
Thus the cofibrant cell complexes exhaust the class of cofibrant objects.
Observe also that they can be characterized as those objects $X\in\fS^G$
such that $X^H=\emptyset$ for $H\notin\cF$.
\end{rem}

\begin{para}\label{para:equihom}
For the model structures of Proposition \ref{prop:equimodel}, the functorial
cofibrant replacement in $\Top^G$ of the point space $*$ is a model
for the classifying space of $G$ with respect to $\cF$ and the
cofibrant replacement of $*$ in $\sS^G$ is a simplicial version.
Moreover because $|-|_*:\sS^G\to \Top^G$ is a Quillen equivalence,
it takes the simplicial version to the topological one. In
particular if $E$ is a functor from $\Top^G$ to spectra and
$\pi:\cE(G,\cF)\to *$ is the cofibrant replacement in $\sS^G$, then
we have a map
\begin{equation}\label{easse}
 E(\pi):E(|\cE(G,\cF)|)\to E(*)
\end{equation}
If
\[
E(X)=F_\%(X)=R(X)\otimes_{\mathrm{Or}G}F:=\int^{\org}X^H_+\wedge F(G/H)
\]
for some functor $F:\mathrm{Or}G\to Spt$, \eqref{easse} is the
Davis-L\"uck assembly map of \cite[\S5.1]{dl}. In case $F=|F'|$ is
the geometric realization of a functorial spectrum in the simplicial
set sense, we have further
\[
|F'|_\%(|X|)=|\int^{\org} X^H_+\wedge
F'(G/H)|=|F'_\%(X)|
\]
and the assembly map for $F$ is the geometric realization of that of
$F'$. Hence we can equivalently work with assembly maps in the
topological or the simplicial setting; we choose to do the latter.
In particular all spectra considered henceforth are simplicial.
If $C$ is a chain complex, we will write $|C|$ for the spectrum
associated to it by the Dold-Kan correspondence; since topological spaces
will occur only rarely from now on, and since we will not use $|\ \ |$
to indicate realization, this should cause no confusion.

\end{para}

\section{Rings and categories}\label{sec:r&c}
\numberwithin{equation}{subsection}
\subsection{Crossed products and equivariant homology}

A \emph{groupoid} is a small category where all arrows are
isomorphisms. Let $\cG$ be a groupoid, and let $R$ be a unital ring.
An \emph{action} of $\cG$ on $R$ is a functor $\rho:\cG\to\ring_1$
such that $\rho(x)=R$ for all $x\in \ob \cG$. For example we may
take $\rho(g)=id_R$ for all arrows $g\in\arr\cG$; this is called the
\emph{trivial} action. Whenever $\rho$ is fixed, we omit it from our
notation, and write
\[
g(r)=\rho(g)(r) 
\]
for $g\in\arr\cG$ and $r\in R$. Given a triple $(\cG,\rho,R)$, we consider a small $\Z$-linear category $R\rtimes\cG$. The objects of $R\rtimes\cG$ are those of $\cG$, and
\[
\hom_{R\rtimes\cG}(x,y)=R\otimes\Z[\hom_\cG(x,y)]
\]
If $s\in R$ and $g\in\hom_\cG(x,y)$, we write $s\rtimes g$ for $s\otimes g$.
Composition is defined by the rule
\begin{equation}\label{rule:cross}
(r\rtimes f)\cdot (s\rtimes g)=r f(s)\rtimes fg
\end{equation}
here $r,s\in R$, and $f$ and $g$ are composable arrows in $\cG$. In case the action of $\cG$ on $R$ is trivial,
we also write $R[\cG]$ for $R\rtimes\cG$.

\smallskip
Let $G$ be a group; consider the functor $\cG^G:G-\sets\to\gpd$ which sends a $G$-set $S$ to its \emph{transport groupoid}. By definition $\ob\cG^G(S)=S$, and $\hom_{\cG^G(S)}(s,t)=\{g\in G:g\cdot s=t\}$.

\begin{nota} If $E$ is a functor from $\Z$-linear categories to spectra, $R$ a unital $G$-ring, and $X$ a $G$-space, we put
\[
H^G(X,E(R)):=E(R\rtimes\cG^G(?))_\%(X)
\]
\end{nota}

\subsection{The ring $\cA(\cC)$}

Let $\cC$ be a small $\Z$-linear category. Put
\begin{equation}\label{ac}
\cA(\cC)=\bigoplus_{a,b\in\obc}\hom_\cC(a,b)
\end{equation}
If $f\in \cA(\cC)$ write $f_{a,b}$ for the component in
$\hom_\cC(b,a)$. The following multiplication law
\begin{equation}\label{rule:matrix}
(fg)_{a,b}=\sum_{c\in\obc}f_{a,c}g_{c,b}
\end{equation}
makes $\cA(\cC)$ into an associative ring, which is unital if and only if $\obc$ is finite. Whatever the cardinal
of $\obc$ is, $\cA(\cC)$ is always a ring with {\it local units}, i.e. a filtering colimit of unital rings.

\goodbreak

\bigskip

\paragraph{\em $\cA(?)$ and tensor products}

The \emph{tensor product} of two $\Z$-linear categories $\cC$ and
$\cD$ is the $\Z$-linear category $\cC\otimes\cD$ with
$\ob(\cC\otimes\cD)=\ob(\cC)\times \ob(\cD)$ and
\[
\hom_{\cC\otimes\cD}((c_1,d_1),(c_2,d_2))=\hom_\cC(c_1,c_2)\otimes\hom_\cD(d_1,d_2)
\]
We have
\[
\cA(\cC\otimes\cD)=\cA(\cC)\otimes\cA(\cD)
\]

\begin{ex} If $\cG$ is a groupoid acting trivially on a unital ring $R$, then
\[
\cA(R[\cG])=\cA(R\otimes\Z[\cG])=R\otimes\cA(\Z[\cG])
\]
\end{ex}

\goodbreak

\bigskip

\paragraph{\em $\cA(?)$ and crossed products}

If $A$ is any, not necessarily unital ring, and $\cG$ is a groupoid acting on $A$, we put
\[
\cA(A\rtimes\cG)=\bigoplus_{x,y\in\ob\cG}A\otimes\Z[\hom_\cG(x,y)]
\]
The rules \eqref{rule:cross} and \eqref{rule:matrix} make
$\cA(A\rtimes\cG)$ into a ring, which in general is nonunital and
does not have local units. The ring $\cA(A\rtimes\cG)$ may also be
described in terms of the \emph{unitalization} $\tilde{A}$ of $A$.
By definition, $\tilde{A}=A\oplus\Z$ equipped with the trivial
$\cG$-action on the $\Z$-summand and the following multiplication
\begin{equation}\label{rule:unitalization}
(a,\lambda)(b,\mu)=(ab+\lambda b+a\mu,\lambda\mu)
\end{equation}
We have
\begin{equation}\label{acrossnuni}
\cA(A\rtimes\cG)=\ker(\cA(\tilde{A}\rtimes\cG)\to \cA(\Z[\cG]))
\end{equation}
Note that $\cA(A\rtimes\cG)$ is defined, even though $A\rtimes\cG$
is not. One can actually define $A\rtimes\cG$ as a nonunital
category, i.e. a category without identity morphisms, but we do not
go into that in this paper.

Next we fix a group $G$ and a subgroup $H\subset G$ and consider the
ring $\cA(A\rtimes\cG^G(G/H))$ associated to the crossed product by
the transport groupoid. Note that
\[
\hom_{\cG^G(G/H)}(H,H)=H=\hom_{\cG^H(H/H)}(H,H)
\]
thus there is a fully faithful functor $\cG^H(H/H)\to \cG^G(G/H)$. This functor induces a ring homomorphism
\[
\j:A\rtimes H=\cA(A\rtimes \cG^H(H/H))\subset \cA(A\rtimes\cG^G(G/H))
\]
The next lemma compares the map $\j$ with the canonical inclusion
\[
\iota:A\rtimes H\to M_{G/H}(A\rtimes H),\quad x\mapsto
e_{H,H}\otimes x
\]
In the following lemma and elsewhere, we make use of a section
$s:G/H\to G$ of the canonical projection onto the quotient by a
sugroup $H\subset G$. We say that the section $s$ is \emph{pointed}
if it is a map of pointed sets, that is, if it maps the class of $H$
to the element $1\in G$.

\begin{lem}\label{lem:across=cross} Let $A$ be a ring, $G$ a group acting on $A$, and $H\subset G$ a subgroup.
Then there is an isomorphism $\alpha:\cA(A\rtimes \cG^G(G/H))\iso
M_{G/H}(A\rtimes H)$ making the following diagram commute:
\[
\xymatrix{A\rtimes H\ar[r]^(.4)\j\ar[dr]_\iota & \cA(A\rtimes
\cG^G(G/H))\ar[d]^{\wr^{\alpha}}\\ & M_{G/H}(A\rtimes H)}
\]
The isomorphism $\alpha$ is natural in $A$ but not in the pair
$(G,H)$, as it depends on a choice of pointed section $s:G/H\to G$
of the projection $\pi:G\to G/H$.
\end{lem}
\begin{proof}
Let $s$ be as in the lemma; put $\hat{g}=s(\pi(g))$ ($g\in G$). The
isomorphism $\alpha:\cA(A\rtimes \cG^G(G/H))\iso M_{G/H}(A\rtimes
H)$ is defined as follows. For $b\in A$, $s,t\in G$, and $g\in
\hom_{\cG^G(G/H)}(sH,tH)$, put
\[
\alpha(b\rtimes g)=e_{tH,sH}\otimes \hat{t}^{-1}(b)\rtimes
(\hat{t}^{-1}g\hat{s})
\]
It is straightforward to check that $\alpha$ is an isomorphism and
that $\alpha\j=\iota$.
\end{proof}

\smallskip
\paragraph{\em Functoriality of $\cA(?)$}

If  $F:\cC\mapsto \cD$ is a $\Z$-linear functor which is injective on objects, then it defines a homomorphism
$\cA(F):\cA(\cC)\to \cA(\cD)$ by the rule $\alpha\mapsto F(\alpha)$.  Hence we may
regard $\cA$ as  a functor
\begin{equation}\label{fun:ac}
\cA:\inj-\Z-\cat\to\ring
\end{equation}
from the category of $\Z$-linear categories and functors which are injective on objects, to the category of rings. However $\cA(F)$ is not defined for general
$\Z$-linear $F$.

\begin{rem}\label{rem:inj}
 The use of the prefix $\inj$ here differs from that in \cite{dl}. Indeed, here $\inj$ indicates that
functors are injective on objects, whereas in \cite{dl}, it refers to functors which are injective on arrows.
\end{rem}

\subsection{The nonunital case}

A {\it Milnor square}  is a pullback square of rings
\begin{equation}\label{diag:mil}
\xymatrix{R'\ar[d]\ar[r]& R\ar[d]^f\\
               S'\ar[r]_g& S}
\end{equation}
such that either $f$ or $g$ is surjective. Below we shall assume $f$
is surjective. Let $E:\Z-\cat\to\spt$ be a functor. If $A$ is a not
necessarily unital ring, embedded as an ideal in a unital ring $R$,
we write $E(R:A)=\hofi(E(R)\to E(R/A))$. The functor $E$ is said to
satisfy \emph{excision} for the Milnor square  \eqref{diag:mil} if
\[
\xymatrix{E(R')\ar[d]\ar[r]& E(R)\ar[d]^{E(f)}\\
               E(S')\ar[r]& E(S)}
\]
is homotopy cartesian. If $\ker f\cong A$, then $E$ satisfies
excision on \eqref{diag:mil} if and only if
\[
E(R',R:A)=\hofi (E(R':A)\to E(R:A))
\]
is weakly contractible. We say that the ring $A$ is {\em
$E$-excisive} if $E$ satisfies excision on every Milnor square
\eqref{diag:mil} with $\ker f\cong A$. Assume unital rings are
$E$-excisive; if $A$ is any, not necessarily $E$-excisive ring, we
consider its unitalization $\tilde{A}$, defined in
\eqref{rule:unitalization} above. Put
\[
E(A)=\hofi(E(\tilde{A})\to E(\Z))
\]
Because of our assumption that unital rings are $E$-excisive, if $A$ happens to be unital, the two definitions of $E(A)$ are naturally homotopy equivalent. Note that
if
\[
0\to A'\to A\to A"\to 0
\]
is an exact sequence of rings and $A'$ is $E$-excisive, then
\[
E(A')\to E(A)\to E(A")
\]
is a homotopy fibration. We say that $E$ is \emph{excisive} or that it \emph{satisfies excision}, if every ring is
$E$-excisive.

\begin{stan}\label{stan}
From now on, we shall be primarily concerned with functors $E:\Z-\cat\to\spt$ that satisfy the following:

\begin{itemize}

\item[i)] Every ring with local units is $E$-excisive.

\item[ii)] If $H$ is a group and $A$ an $E$-excisive $H$-ring, then $A\rtimes H$ is
$E$-excisive.

\item[iii)] If $A$ is $E$-excisive, $X$ a set and $x\in X$, then $M_{X}A$ is $E$-excisive, and $E$ sends the map
$A\to M_{X}A$, $a\mapsto e_{x,x}a$ to a weak equivalence.

\item[iv)] There is a natural weak equivalence $E(\cA(\cC))\weq E(\cC)$ of functors $\inj-\Z-\cat\to\spt$.

\item[v)] Let $\{A_i:i\in I\}$ be a family of rings, and let $A=\bigoplus_{i\in I}
A_i$ be their direct sum, with coordinate-wise multiplication. Then
$A$ is $E$-excisive if and only if each $A_i$ is. Moreover if these
equivalent conditions are satisfied, then the map $\bigoplus_i
E(A_i)\to E(A)$ is an equivalence.
\end{itemize}
\end{stan}

\begin{rem}\label{rem:prerc}
Observe that standing assumptions i)-iii) and v) are only concerned
with the restriction of $E$ to the full subcategory
$\ring\subset\Z-\cat$, and that assumption iv) says that
$E_{|\ring}$ determines the whole functor up to weak equivalence.
However the assumptions are enough to prove for instance that $E$
maps category equivalences to equivalences of spectra; see
\ref{rem:equiv}. Note also that the equivalence of iv) is natural
only with respect to functors which are injective on objects,
because $\cA(-)$ is only functorial on $\inj-\Z-\cat$. One could ask
whether it is possible to extend a functor $E:\ring\to\spt$
satisfying i)-iii) and v) to all of $\Z-\cat$ in such a way that iv)
is satisfied. In the next subsection we introduce a functor
$\cR:\Z-\cat\to\ring$ which restricts to the identity on $\ring$ and
a natural transformation $p:\cR\to\cA$ of functors
$\inj-\Z-\cat\to\ring$ and discuss conditions on $E$ under which
$E(p)$ is an equivalence.
\end{rem}

\begin{rem}\label{rem:suni}
The examples we are primarily interested in, namely $K$-theory and
Hochschild and cyclic homology, satisfy a stronger version of
property i). Indeed, they not only satisfy excision for rings with
local units, but also for (flat) $s$-unital rings. A ring $A$ is
called \emph{$s$-unital} if for every finite collection
$a_1,\dots,a_n\in A$ there exists an element $e\in A$ such that
$a_ie=ea_i=a_i$. Note that if we add the requirement that $e$ be
idempotent we recover the notion of ring with local units. As is explained
in the Appendix (Example \ref{ex:sunihuni}) every $s$-unital ring is excisive for
both Hochschild and cyclic homology, and every $s$-unital ring which
is flat as an abelian group is $K$-excisive.
\end{rem}

\begin{rem}\label{rem:exci}
If $E$ satisfies excision, then assumptions i) and ii) hold
automatically, and assumptions iii) and v) hold if and only if they
hold for unital rings.
\end{rem}

\begin{lem}\label{lem:enat}
Let $E:\Z-\cat\to\spt$ be a functor satisfying the standing assumptions above. If $F_i:\cC\to\cD$  $i=0,1$ are naturally isomorphic linear functors, then $E(F_0)$
and $E(F_1)$ are homotopic.
\end{lem}
\begin{proof}
Let $\cG[1]=\{0\leftrightarrows 1\}$ be the groupoid with two
objects and exactly one isomorphism between any two given (equal or
distinct) objects. The linear functors $F,G:\cC \to\cD$ are
equivalent if the dotted arrow in the following diagram of
$\Z$-linear functors exists and makes it commute
\[
\xymatrix{&\cC\otimes \Z[\cG[1]]\ar@{.>}[d]\\
\cC\oplus\cD=\cC\otimes\Z[\ob\cG[1]]\ar[ur]^{\iota_0\oplus
\iota_1}\ar[r]_(.6){F\oplus G}&\cD}
\]
Hence it suffices to show that $E(\iota_0)\cong E(\iota_1)$. By
assumption iv), we are reduced to showing that $E(\cA(\iota_0))\cong
E(\cA(\iota_1))$. But one checks that
$\cA(\cC\otimes\Z[\cG[1]])=M_2(\cA(\cC))$ and that the $\iota_i$
induce the two canonical inclusions $x\mapsto x\otimes e_{1,1},\ \
x\otimes e_{2,2}$, hence we are done by assumption iii) (see
\cite[Lemma 2.2.4]{friendly}, e.g.).
\end{proof}

\begin{rem}\label{rem:equiv}
It follows from Lemma \ref{lem:enat} that $E$ sends category equivalences to equivalences of spectra.
\end{rem}

Let $G$ be a group. Assume $E$ satisfies the standing assumptions above. For $A$ an $E$-excisive $G$-ring, consider the $\org$-spectrum
\begin{equation}\label{fun:orgnuni}
G/H\mapsto E(A\rtimes \cG^G(G/H))=\hofi(E(\tilde{A}\rtimes\cG^G(G/H))\to E(\Z[\cG^G(G/H))])
\end{equation}
Applying $(?)_\%$ to \eqref{fun:orgnuni} defines an equivariant
homology theory of $G$-simplicial sets, which we denote
$H^G(-,E(A))$. Moreover, for each fixed $G$-simplicial set $X$,
$H^G(X,E(?))$ is a functor of $E$-excisive rings. Observe that, for
unital $A$, we have two definitions of $E(A\rtimes\cG^G(-))$ and two
definitions of $H^G(-,E(A))$; the next proposition says that the two
definitions are equivalent.

\begin{prop}\label{prop:exciequi1}
Let $E:\Z-\cat\to\spt$ be a functor and $G$ a group. Assume that $E$
satisfies the standing assumptions \ref{stan} above.
\item[a)] If
$R$ is a unital $G$-ring, then the two definitions of
$E(R\rtimes\cG^G(-))$ and the two definitions of $H^G(-,E(R))$ are
equivalent.
\item[b)] If
\[
0\to A'\to A\to A"\to 0
\]
is an exact sequence of $E$-excisive $G$-rings, and $X$ is a $G$-simplicial set, then
\[
E(A'\rtimes\cG^G(-))\to E(A\rtimes\cG^G(-))\to E(A"\rtimes\cG^G(-))
\]
and
\[
 H^G(X,E(A'))\to H^G(X,E(A))\to H^G(X,E(A"))
\]
are homotopy fibrations.
\end{prop}
\begin{proof} If $A$ is $E$-excisive and $H\subset G$ is a subgroup, then conditions ii) and iii) together with Lemma \ref{lem:across=cross} imply that
$\cA(A\rtimes\cG^G(G/H))$ is $E$-excisive. Hence, by condition iv),
the spectrum in \eqref{fun:orgnuni} is equivalent to
$E(\cA(A\rtimes\cG^G(G/H)))$. In particular, by i),
$\cA(R\rtimes\cG^G(G/H))$ is $E$-excisive for $R$ unital, and the
map
\[\hofi(E(\tilde{R}\rtimes\cG^G(G/H))\to E(\Z[\cG^G(G/H)]))\to E(R\rtimes\cG^G(G/H))\]
induced by the projection $\tilde{R}\cong R\times \Z\to R$ is an
equivalence. This proves a). Moreover, because
$\cA(?\rtimes\cG^G(G/H))$ preserves exact sequences, applying
\eqref{fun:orgnuni} to the exact sequence of part b) yields an
object-wise homotopy fibration of $\org$-spectra, which is the first
homotopy fibration of b). Applying $(?)_\%$ we obtain the second
one.
\end{proof}

\begin{rem}\label{rem:assnuni}
Let $E:\Z-\cat\to \spt$  and let $A$ be any, not necessarily
$E$-excisive $G$-ring, equivariantly embedded as an ideal in a
unital $G$-ring $R$. Consider the $\org$-spectrum
\[
E(R\rtimes\cG^G(-):A\rtimes\cG^G(-))=\hofi(E(R\rtimes\cG^G(-))\to
E((R/A)\rtimes\cG^G(-)))
\]
Put
\[
H^G(X,E(R:A))=E(R\rtimes\cG^G(-):A\rtimes\cG^G(-))_\%(X).
\]
A $(G,\cF)$-cofibrant replacement $cX\to X$ gives rise to a map of homotopy fibrations
\[
\xymatrix{H^G(cX,E(R:A))\ar[r]\ar[d]& H^G(cX,E(R))\ar[d]\ar[r] &H^G(cX,E(R/A))\ar[d]\\
H^G(X,E(R:A))\ar[r]& H^G(X,E(R))\ar[r] &H^G(X,E(R/A))}
\]
If $H^G(cX,E(S))\to H^G(X,E(S))$ is an equivalence for all unital $S$, then both the middle and right
hand side vertical maps are equivalences; it follows that the same
is true of the map on the left. We record a particular case of this
in the following corollary.
\end{rem}

\begin{cor}\label{cor:assnuni}
Let $E:\Z-\cat\to\spt$ be a functor; assume $E$ satisfies the
Standing Assumptions \ref{stan}. Further let $G$ be a group, $X\in\fS^G$,
$\cF$ a family of subgroups, $cX\to X$ and $(G,\cF)$-cofibrant replacement.
Assume that the assembly map
$H^G(cX,E(R))\to H^G(X,E(R))$ is an equivalence for every
unital ring $R$.  Then $H^G(cX,E(A))\to H^G(X,E(A))$ is an
equivalence for every $E$-excisive ring $A$.
\end{cor}

\begin{prop}\label{prop:excimoron}
Let $A\triqui R$ be an ideal in a unital $G$-ring, closed under the
action of $G$. Let $E:\ring\to\spt$ be a functor satisfying the
standing assumptions. If $A$ is $E$-excisive then
\[
E(A\rtimes\cG^G(-))\to E(R\rtimes \cG^G(-):A\rtimes \cG^G(-))
\]
is an object-wise weak equivalence of $\org$-spectra.
\end{prop}
\begin{proof} The proof follows from Lemma \ref{lem:across=cross}, using assumptions ii), iii) and
iv).
\end{proof}

\subsection{The ring $\cR(\cC)$}

\smallskip

Let $\cC$ be a $\Z$-linear category. Imitating a construction used
by M. Joachim (\cite{joach}) in the $C^*$-algebra context, we shall
associate to $\cC$ a ring $\cR(\cC)$ which is a quotient of the
tensor algebra of $\cA(\cC)$; first we need some notation. If $M$ is
an abelian group, we write $T(M)=\bigoplus_{n\ge 1} M^{\otimes n}$
for the (unaugmented) tensor algebra. Put
\[
 \cR(\cC)=T(\cA(\cC))/<\{g\otimes f-g\circ f:f\in\hom_\cC(a,b),
 g\in\hom_\cC(b,c),\ \ a,b,c\in\ob\cC\}>
 \]
Note that any $\Z$-linear functor $\cC\to\cD\in\Z-\cat$ defines a
homomorphism $\cR(\cC)\to\cR(\cD)$. Thus we may regard $\cR$ as a
functor
\[
\cR:\Z-\cat\to\rings, \quad \cC\mapsto \cR(\cC)
\]
Observe that the canonical surjection $T(\cA(\cC))\to \cA(\cC)$
factors through a map
\begin{equation}\label{map:bcontoac}
p:\cR(\cC)\onto\cA(\cC)
\end{equation}
whose kernel is the ideal generated by the elements $g\otimes f$ for
non-composable $g$ and $f$. For example if $\cC$ has only one
object, then $p$ is the identity. In particular any functor
$E:\ring\to\spt$ can be extended to $\Z-\cat$ via
$E(\cC)=E(\cR(\cC))$, and $p$ induces a natural transformation
$E(p):E(\cC)\to E(\cA(\cC))$ of functors of $\inj-\Z-\cat$.

\begin{ex}\label{ex:acoprdb}
Let $R,S$ be unital rings, and let $\cC$ be the $\Z$-linear category
with two objects $a$ and $b$ such that
$\hom_\cC(a,b)=\hom_\cC(b,a)=0$, $\hom_\cC(a,a)=R$ and
$\hom_\cC(b,b)=S$. Then  $\cA(\cC)=R\oplus S$ and $\cR(\cC)=R\coprod
S$ is the nonunital coproduct. We shall see in Proposition \ref{prop:kstands} that $K$-theory
satisfies the standing assumptions; however in general $K_*(R\coprod S)\neq K_*(R)\oplus K_*(S)$.
\end{ex}

In Lemma \ref{lem:eac=ebc} we give
conditions on $E$ which guarantee that it sends the map \eqref{map:bcontoac} to a weak equivalence. First we need some
notation.
If $B$ is a ring, we write
$\ev_i:B[t]\to B$ $i=0,1$ for the evaluation maps. If $f,g:A\to B$
are ring homomorphisms, then a (polynomial) \emph{elementary
homotopy} between $f$ and $g$ is a map $H:A\to B[t]$ such that
$\ev_0 H=f$ and $\ev_1 H=g$. A \emph{homotopy} from $f$ to $g$ is a
sequence of homomorphisms $f=h_0,\dots, h_n=g$ and elementary
homotopies $H_i:A\to B[t]$ from $h_i$ to $h_{i+1}$. The functor $E$
is \emph{invariant under polynomial homotopy} if for every ring $A$,
$E$ sends the inclusion $A\subset A[t]$ to a weak equivalence.
Because the composite $\inc\circ\ev_0:A[t]\to A[t]$ is homotopic to
the identity, if $E$ is invariant under polynomial homotopy, and $f$
and $g$ are homotopic ring homomorphisms, then $E(f)$ and $E(g)$
define the same map in $\ho\spt$.

\goodbreak
\begin{lem}\label{lem:eac=ebc}
Let $E:\rings\to\spt$ be a functor. Assume that $E$ satisfies standing assumptions i) and iii). Let $\cC$ be a $\Z$-linear category such that $\cR(\cC)$ is $E$-excisive. Then $E_*$ sends
\eqref{map:bcontoac} to a naturally split surjection. Assume in
addition that $E$ is invariant under polynomial homotopy. Then $E$
sends \eqref{map:bcontoac} to a weak equivalence.
\end{lem}
\begin{proof}
Let $\ob_+\cC=\ob\cC\coprod\{+\}$ be the set of objects of $\cC$
with a base point added. Consider the homomorphism
\[
j:\cA(\cC)\to M_{\ob_+\cC}\cR(\cC),\quad j(f)=f\otimes e_{b,a}\qquad
(f\in\hom_\cC(a,b))
\]
Write $p$ for the map \eqref{map:bcontoac}. Consider the matrices
\begin{gather*}
V=\sum_{a\in\ob\cC}1_a\otimes e_{a,+}\\
W=\sum_{a\in\ob\cC}1_a\otimes e_{+,a}
\end{gather*}
The composite
$q=M_{\ob_+\cC}(p)\circ j$ sends $f\in\cA(\cC)$ to
\[
q(f)=Wf\otimes e_{+,+}V
\]
Observe that left multiplication by $W$ and right multiplication by $V$ leave $M_{\ob_+\cC}\cA(\cC)$ stable, and that
$aVWa'=aa'$ for all $a,a'\in M_{\ob_+\cC}\cA(\cC)$. By the argument of \cite[2.2.6]{friendly}, all this together with matrix invariance imply that
$E_*(q)=E_*(?\otimes e_{+,+})$ is an isomorphism.
 This proves the first assertion of the
Lemma. To prove the second, it suffices to show that $r=j\circ p$\ \
is homotopic to the inclusion $\iota(a)=a\otimes e_{+,+}$. If
$f\in\hom_\cC(a,b)$, write $H(f)\in M_{\ob_+\cC}(\cR(\cC))[t]$ for
%\begin{multline*}
\[H(f)=f\otimes(-t(t^3-2t)e_{+,+}+t(t^2-1)e_{+,a}+(1-t^2)(t^3-2t)e_{b,+}+(1-t^2)^2e_{b,a})\]
%\end{multline*}
Note that $\ev_0H(f)=r(f)$, $\ev_1H(f)=\iota(f)$. Further, one
checks that if $g\in\hom_\cC(b,c)$, then $H(gf)=H(g)H(f)$. Thus $H$
induces a homomorphism $\cR(\cC)\to M_{\ob_+\cC}(\cR(\cC))[t]$ which
is a homotopy from $r$ to $\iota$. This concludes the proof.
\end{proof}

\begin{ex}
If $E:\ring\to\spt$ is excisive and homotopy invariant and satisfies standing assumptions
iii) and v), then its extension $E\circ\cR:\Z-\cat\to\spt$ satisfies all the standing assumptions and
agrees with $E$ on $\ring$. If $F$ is another extension of $E$ which also satisfies the standing assumptions,
then composing $E(\cR(\cC))\to E(\cA(\cC))$ with the map of assumption \ref{stan} iv), we get an equivalence
$E(\cR(\cC))\to F(\cC)$ which is natural with respect to functors which are injective on objects.
\end{ex}

\section{$K$-theory}\label{sec:kth}
\numberwithin{equation}{subsection}
\subsection{The $K$-theory spectrum}

Given a $\Z$-linear category $\cC$, we denote by $\cC_\oplus$ the
$\Z$-linear category whose objects are finite sequences of objects
of $\cC$, and whose morphisms are matrices of morphisms in $\cC$
with the obvious matrix product as composition. Concatenation of
sequences yields a sum $\oplus$ and hence we obtain, functorially,
an additive category; write $\idem\cC_\oplus$ for its idempotent
completion. We shall also need  \emph{Karoubi's cone} $\Gamma(\cC)$
(\cite[pp~270]{kv}). The objects of $\Gamma(\cC)$ are the sequences
$x=(x_1,x_2,\dots)$ of objects of $\cC$ such that the set
\begin{equation}\label{F(x)}
F(x)=\{c\in \cC: (\exists n)\quad x_n=c\}
\end{equation}
is finite. A map $x\to y$ in $\Gamma(\cC)$ is a matrix
$f=(f_{i,j})$ of homomorphisms $f_{i,j}:x_j\to y_i$ such that
\begin{enumerate}
\item There exists an $N$ such that every row and every column of $f$ has at most $N$ nonzero entries.
\item The set $\{f_{i,j}:i,j\in\N\}$ is finite.
\end{enumerate}
Interspersing of sequences defines a symmetric monoidal operation
$\boxplus:\Gamma(\cC)\times \Gamma(\cC)\to\Gamma(\cC)$ and there is
an endofunctor $\tau$ such that $1\boxplus \tau\cong \tau$ (see
\cite[\S III]{karder}). If $\cC$ has finite direct sums, e.g. if
$\cC=\cD_\oplus$ for some $\Z$-linear category $\cD$, then the
interspersing operation is naturally equivalent to the induced sum
$(x\oplus y)_i=x_i\oplus y_i$ (\cite[Lemme ~3.3]{karder}). In
particular, if $\cC$ is additive, then $\Gamma\cC$ is a {\em
flasque} additive category; that is, there is an additive
endofunctor $\tau:\cC\to\cC$ such that $\tau\oplus 1\cong \tau$. A
morphism $f$ in $\Gamma(\cC)$ is \emph{finite} if $f_{ij}=0$ for all
but finitely many $(i,j)$. Finite morphisms form an ideal, and we
write $\Sigma(\cC)$ for the category with the same objects as
$\Gamma(\cC)$, and morphisms taken modulo the ideal of finite
morphisms. The category $\Sigma(\cC)$ is Karoubi's \emph{suspension}
of $\cC$.
 By \cite[Thm. 5.3]{pw2},
if $\cC$ is additive, we have a homotopy fibration sequence
\begin{equation}\label{seq:karfib}
K^Q(\idem\cC)\to K^Q(\Gamma(\idem\cC))\to K^Q(\Sigma(\idem\cC))
\end{equation}
Here each of the categories is regarded as a semisimple exact
category, and $K^Q$ denotes the fibrant simplicial set for its
algebraic $K$-theory. Because $\Gamma(\idem\cC)$ is flasque,
$K^Q(\Gamma(\idem\cC))$ is contractible, whence $K^Q(\idem\cC)\cong
\Omega K^Q(\Sigma(\idem\cC))$. Now let $\cC$ be any small
$\Z$-linear category, possibly without direct sums. Consider the
sequence of categories
\begin{equation}\label{cn}
\cC^{(0)}=\idem(\cC_\oplus),\quad \cC^{(n+1)}=\idem(\Sigma\cC^{(n)})
\end{equation}
Then we have a spectrum $K(\cC)=\{{}_nK(\cC)\}$, with
\begin{equation}\label{eq:kspt}
{}_nK(\cC)\cong K^Q(\cC^{(n)})
\end{equation}

\begin{rem}\label{rem:kr=kr} If $R$ is a unital ring, then by \cite[Prop. 1.6]{kv}, we have category equivalences
\begin{equation}\label{kveq}
    \idem(\Gamma(\proj(R)))\cong\proj(\Gamma(R))\text{ and } \idem(\Sigma(\proj(R)))\cong\proj(\Sigma(R))
\end{equation}
Hence the spectrum $K(R)$ defined above is equivalent to the usual, Gersten-Karoubi-Wagoner spectrum of the
ring $R$.
\end{rem}

\subsection{Comparing $K(\cC)$ with $K(\cA(\cC))$}

\goodbreak

\paragraph{\em The operation $\lozenge$}

Let $X$ be a set and let $\cC$ and $\cD$ be $\Z$-linear categories
with $\ob\cC=\ob\cD=X$. Consider the category $\cC\lozenge\cD$ with
set of objects $\ob(\cC\lozenge\cD)=X$, homomorphisms
\[
\hom_{\cC\lozenge\cD}(x,y)=\hom_{\cC}(x,y)\oplus\hom_{\cD}(x,y)
\]
and coordinate-wise composition. If $\cC$, $\cD$ and $\cE$ are $\Z$-linear categories,
we have
\begin{gather}
(\cC\lozenge\cD)_\oplus=\cC_\oplus\lozenge\cD_\oplus\nonumber\\
\idem((\cC\lozenge\cD)_\oplus)=\idem\cC_\oplus\times\idem\cD_\oplus\label{eq:idemsucrets}\\
(\cC\lozenge\cD)\otimes\cE=(\cC\otimes\cE)\lozenge(\cD\otimes\cE)\label{eq:lozetenso}
\end{gather}
\goodbreak

\paragraph{\em Unitalization}

We have already recalled the definition of the unitalization
$\tilde{A}$ of a not necessarily unital ring $A$. Now we need a
version of unitalization for $\Z$-linear categories; this can be
more generally defined for nonunital $\Z$-categories, but we will
have no occasion for that. Let $\cC\in\Z-\cat$; write $\tilde{\cC}$
for the category with $\ob\tilde{\cC}=\ob\cC$ and with homomorphisms
given by
\[
\hom_{\tilde{\cC}}(x,y)=\hom_{\cC}(x,y)\oplus\delta_{x,y}\Z=\left\{\begin{matrix}\hom_{\cC}(x,y)
&x\ne y\\
\hom_{\cC}(x,x)\oplus\Z & x=y\end{matrix}\right.
\]
Composition between $(f,\delta_{x,y}n)\in\hom_{\tilde{\cC}}(x,y)$
and $(g,\delta_{y,z}m)\in\hom_{\tilde{\cC}}(y,z)$ is defined by the
formula
\[
(g,\delta_{y,z}m)\circ(f,\delta_{x,y}n)=(gf+\delta_{y,z}mf+\delta_{x,y}gn,\delta_{x,y}\delta_{y,z}mn)
\]
Observe that if $R$ is a ring, considered as a $\Z$-linear category
with one object, then

\[\tilde{R}\to R\times\Z=R\lozenge\Z,\qquad (r,n)\mapsto (r+n\cdot
1,n)
\]

\noindent is an isomorphism. This isomorphism generalizes to
$\Z$-categories as follows. Let $\Z\langle\ob\cC\rangle\in\Z-\cat$,
be the $\Z$-linear category with the same objects as $\cC$ and
homomorphisms given by
\[
\hom_{\Z\langle\ob\cC\rangle}(x,y)=\delta_{x,y}\Z
\]
We have an isomorphism of linear categories
\begin{equation}\label{map:sucretilde}
\cC\lozenge\Z\langle\ob\cC\rangle\iso \tilde{\cC}
\end{equation}
 which is the
identity on objects, as well as on
$\hom_{\cC\lozenge\Z\langle\ob\cC\rangle}(x,y)$ for $x\ne y$, and
which sends
\[
\hom_{\cC\lozenge\Z\langle\ob\cC\rangle}(x,x)\owns (f,n)\mapsto
(f-n1_x,n)\in \hom_{\tilde{\cC}}(x,x)
\]
\bigskip
\goodbreak
\paragraph{\em The map $K(\cC)\to K(\cA(\cC))$}

If $\cC$ is a $\Z$-linear category, and $x,y\in\ob\cC$, then by
definition of $\cA(\cC)$,
\begin{equation}\label{inc1}
\hom_{\cC}(x,y)\subset\cA(\cC)
\end{equation}
and the inclusion is compatible with composition. We also have an
inclusion
\begin{equation}\label{inc2}
\hom_{\tilde{\cC}}(x,x)\owns (f,n)\mapsto (f,n)\in \widetilde{\cA(\cC)}
\end{equation}
The inclusions \eqref{inc1} and \eqref{inc2} together with the only
map $\ob\tilde{\cC}\to\ob\widetilde{\cA(\cC)}=\{\bullet\}$ define a
functor
\begin{equation}\label{map:phi}
\phi:\tilde{\cC}\to\widetilde{\cA(\cC)}
\end{equation}
Observe that $\Z\langle\ob\cC\rangle\subset \tilde{\cC}$ and that
$\phi(\Z\langle\ob\cC\rangle)\subset\Z\subset\widetilde{\cA(\cC)}$.
We have a commutative diagram
\[
\xymatrix{
 \tilde{\cC}\ar[r]^\phi\ar[d]^{\pi_1}&\widetilde{\cA(\cC)}\ar[d]^{\pi_2}\\
  \Z\langle\obc\rangle\ar[r]&\Z}
\]
Here the vertical maps are the obvious projections. By \eqref{map:sucretilde} and \eqref{eq:idemsucrets} we have an
equivalence
\[
K(\tilde{\cC})\weq K(\cC)\times K(\Z\langle\obc\rangle)
\]
Under this equivalence the map induced by $\pi_1$ becomes the canonical projection; hence its fiber is
$K(\cC)$. On the other hand, by definition, $K(\cA(\cC))$ is the fiber of $K(\pi_2)$. Hence $\phi$ induces
a map
\begin{equation}\label{map:kkpw}
\varphi:K(\cC)\to K(\cA(\cC))
\end{equation}
\begin{prop}\label{prop:kkpw}
Let $\cC$ be a $\Z$-linear category. Then the map \eqref{map:kkpw} is an equivalence.
\end{prop}
\begin{proof}
Because both the source and the target of \eqref{map:kkpw} commute with filtering colimits, we may assume that $\cC$ has finitely many objects.
Then $\cA(\cC)$ is unital, and thus we have an isomorphism $\widetilde{\cA(\cC)}\cong \cA(\cC)\times\Z$. Recall that the idempotent completion of an additive
category $\fA$ is the category whose objects are the idempotent endomorphisms in $\fA$ and where a map $f:e_1\to e_2$ is an element of $\hom_\fA(\dom e_1,\dom e_2)$
 such that $f=e_2fe_1$. One checks that the composite
\begin{multline*}
\cC_\oplus\to\idem\cC_\oplus\overset{1\times 0}\to \idem\cC_\oplus\times\idem\Z\langle\obc\rangle\cong\\
\idem(\tilde{\cC}_\oplus)\overset\phi\to
\idem(\widetilde{\cA(\cC)}_\oplus)\cong
\idem(\cA(\cC)_\oplus)\times\idem(\Z_\oplus)\to
\idem(\cA(\cC)_\oplus)
\end{multline*}
is the functor $\psi$ which sends an object $(c_1,\dots,c_n)$ to the idempotent $\diag(1_{c_1},\dots,1_{c_n})$ and a map
$f=(f_{i,j}):(c_1,\dots,c_n)\to (d_1,\dots,d_m)$ to the corresponding matrix $(f_{i,j})\in \hom_{\cA(\cC)_\oplus}(\bullet^n,\bullet^m)$. Because $\psi$
is fully faithful and cofinal, it induces an equivalence $K(\cC)\to K(\cA(\cC))$. It follows that \eqref{map:kkpw} is an equivalence.
\end{proof}

\subsection{$K$-theory and the standing assumptions}

\begin{prop}\label{prop:kstands}
The functor $K:\Z-\cat\to\spt$ satisfies the standing assumptions.
\end{prop}
\begin{proof}
Assumption iv) was proved in Proposition \ref{prop:kkpw} above. The
remaining assumptions are either proved in Appendix \ref{app:A} or
follow from results therein. By Example \ref{ex:tfp}, rings with
local units are $K$-excisive; hence $K$-theory satisfies i).
Assumption ii) holds by Proposition \ref{prop:crossbar}. If $A$ is
$K$-excisive and $X$ is a set, then $M_{X}A$ is $K$-excisive, by
Proposition \ref{prop:tenso}. Assumption iii) follows from this and
the fact that $K$-theory is matrix stable on unital rings.
Assumption v) is proved in Proposition \ref{prop:barsum}.
\end{proof}

\section{Homotopy $K$-theory}\label{sec:kh}
\numberwithin{equation}{section}

If $\cC$ is a $\Z$-linear category, then we write $\cC^{\Delta^\bu}$ for the
simplicial $\Z$-linear category
\begin{equation}\label{simp:delta}
\cC^{\Delta^\bu}:[n]\mapsto \cC^{\Delta^n}=\cC\otimes\Z[t_0,\dots,t_n]/<t_0+\dots+t_n-1>
\end{equation}
Applying the functor $K$ dimensionwise we get a simplicial spectrum whose total spectrum is
the \emph{homotopy $K$-theory} spectrum $KH(\cC)$. In particular if $R$ is a unital ring, then
$KH(R)$ was defined by Weibel in \cite{kh}. The following theorem was proved in \cite{kh}; see also
\cite[\S5]{friendly}.

\begin{thm}\label{thm:wei}(Weibel)
The functor $KH:\ring\to\spt$ is excisive, matrix invariant, and invariant under polynomial homotopy.
\end{thm}

\begin{prop}\label{prop:homob}
There is a natural weak equivalence $KH(\cC)\weq KH(\cR(\cC))$.
\end{prop}
\begin{proof}
We begin by observing that the inclusions \eqref{inc1} and \eqref{inc2} lift to inclusions
$\hom_{\cC}(x,y)\subset\cR(\cC)$ and $\hom_{\tilde{\cC}}(x,x)\subset\widetilde{\cR(\cC)}$. Thus we have a functor
\[
\phi':\tilde{\cC}\to \widetilde{\cR(\cC)}
\]
Composing it with
\[
\tilde{p}:\widetilde{\cR(\cC)}\to \widetilde{\cA(\cC)}
\]
we obtain the map
\[
\phi:\tilde{\cC}\to \widetilde{\cA(\cC)}
\]
of \eqref{map:phi} above. Tensoring with $\Z^{\Delta^\bu}$ and applying $K(-)$ we obtain a commutative diagram
\[
\xymatrix{KH(\tilde{\cC})\ar[dr]_\phi\ar[r]^{\phi'}&KH(\widetilde{\cR(\cC)})\ar[d]^p\\
&KH(\widetilde{\cA(\cC)})}
\]
The diagram above maps to the diagram
\[
\xymatrix{
KH(\Z\langle\obc\rangle)\ar[r]^\phi\ar[dr]_\phi&KH(\Z)\ar@{=}[d]\\
&KH(\Z)
}
\]
Taking fibers and using \eqref{eq:idemsucrets}, \eqref{eq:lozetenso} and \eqref{map:sucretilde}, we obtain a homotopy commutative
diagram
\[
\xymatrix{KH(\cC)\ar[r]^{\varphi"}\ar[dr]_{\varphi'}&KH(\cR(\cC))\ar[d]^p\\
&KH(\cA(\cC))}
\]
Here $\varphi'$ comes from a map of simplicial spectra

\begin{multline}
\varphi^\bu:K(\cC\otimes\Z^{\Delta^\bu})\weq
K(\tilde{\cC}\otimes\Z^{\Delta^\bu}:\cC\otimes\Z^{\Delta^\bu})\to\\
K(\widetilde{\cA(\cC)}\otimes\Z^{\Delta^\bu}:\cA(\cC)\otimes\Z^{\Delta^\bu})\overset{\sim}\leftarrow K(\cA(\cC)\otimes\Z^{\Delta^\bu}),
\end{multline}
and $\varphi^0=\varphi$ is the map \eqref{map:kkpw}, which is an
equivalence by Proposition \ref{prop:kkpw}. The same argument of the
proof of Proposition \ref{prop:kkpw} shows that $\varphi^n$ is an
equivalence for every $n$. On the other hand, by Theorem
\ref{thm:wei} and Lemma \ref{lem:eac=ebc}, the map
$p:KH(\cR(\cC))\to KH(\cA(\cC))$ is an equivalence. It follows that
$\varphi"$ is an equivalence too.
\end{proof}

\begin{prop}\label{prop:homostand}
The functor $KH:\Z-\cat\to\spt$ satisfies the standing assumptions.
\end{prop}
\begin{proof} All assumptions except iv) follow from Theorem \ref{thm:wei}. Assumption iv) follows from the proof
of Proposition \ref{prop:homob}, and also from combining the statement of that proposition with Lemma \ref{lem:eac=ebc}.
\end{proof}

\section{Cyclic homology}\label{sec:hc}
\numberwithin{equation}{section}
Let $A$ be a ring, and $M$ an $A$-bimodule. If $a\in A$ and $m\in M$, write $[a,m]=am-ma$ and
\[
[A,M]=\{\sum_i[a_i,m_i]:a_i\in A,m_i\in M\},\quad M_\nat=M/[A,M]
\]

Let $B$ be another ring. We say that $B$ is an algebra over $A$ if $B$ is equipped with an $A$-bimodule structure
such that the multiplication $B\otimes B\to B$ factors through an $A$-bimodule map $B\otimes_AB\to B$. Consider the graded abelian group given in degree $n$ by the $n+1$ tensor power modulo $A$-bimodule commutators:
\[
T(B/A)_n=\left(B^{\otimes_A n+1}\right)_\nat
\]
Note $T(B/A)$ is a quotient of $T(B/\Z)$. If $B$ is unital, then $T(B/\Z)$ carries a canonical cyclic module structure \cite[Section 9.6]{chubu}; if $A$ is unital also, and the $A$-bimodule structure on $B$ comes from a unital homomorphism $A\to B$, then the structure passes down to the quotient; we write $C(B/A)$ for $T(B/A)$ equipped with this cyclic
module structure. The cyclic theory of $B/A$, which includes Hochschild, cyclic, negative cyclic and periodic cyclic homology, is that of $C(B/A)$. If $A$ is unital but $B$ is not, one can unitalize $B$ as an $A$-algebra by $\tilde{B}_A=B\oplus A$,
$(b,a)(b',a')=(bb'+ba'+ab',aa')$; the cyclic theory of $B/A$ is that of the cyclic module $C(\tilde{B}_A:B/A)=\ker(C(\tilde{B}_A/A)\to C(A/A))$. In the unital case, there is a natural quasi-isomorphism $C(B/A)\to C(\tilde{B}_A:B/A)$. In the general case, when neither $A$ nor $B$ is assumed
to be unital, then $B$ has a canonical $\tilde{A}$-algebra structure, and the cyclic theory of $B$ as an $A$-algebra is that of $B$ as an $\tilde{A}$-algebra; we put $M(B/A)=C(\tilde{B}_A:B/\tilde{A})$. Note that if $A$ is unital, then $M(B/A)=C(\tilde{B}_A:B/A)$, whence there is no ambiguity.
We use the following notation for homology; we write $HH(B/A)=(M(B/A),b)$ for the Hochschild complex, $HH(B/A)_n$ for its degree $n$ summand, and
$HH_n(B/A)$ for its $nth$ homology group. We use the same convention with cyclic, negative cyclic and periodic
cyclic homology, which we denote $HC$, $HN$ and $HP$.

\goodbreak

Let $\ell$ be a commutative unital ring and $R$ a unital $\ell$-algebra. Recall that $R$ is called {\em separable} over $\ell$ if $R$ is projective as an $R\otimes_\ell R^{op}$-module.

\begin{lem}\label{lem:sep}
Let $I$ be a filtering poset, $I\to \arr(\ring)$, $i\mapsto \{A_i\to
B_i\}$ a functor to the category of ring homomorphisms, and $A\to
B=\colim_i(A_i\to B_i)$. Assume that $A_i\to B_i$ is unital for all
$i$. Put $C(B/A)=\colim_i C(B_i/A_i)$. Then $C(B/A)\to M(B/A)$ is a
quasi-isomorphism. If furthermore each $A_i$ is separable over $\Z$,
then also $C(B)=C(B/\Z)\to C(B/A)$ is a quasi-isomorphism.
\end{lem}
\begin{proof}
The first assertion follows from the fact that both $C$ and $M$
commute with filtering colimits, and that the map is a
quasi-isomorphism in the unital case \cite[Thm. 1.2.13]{lod}. The
second assertion follows similarly from the unital case.
\end{proof}
\begin{ex}\label{ex:cc}
Let $\cC$ be a small $\Z$-linear category. We have an injective
functor $\Z\langle\obc\rangle\to\cC$, and thus a homomorphism
$\cA(\langle\obc\rangle)=\Z^{(\obc)}\to \cA(\cC)$, which is the
filtering colimit over the finite subsets $X\subset\obc$, of the
functor $X\mapsto (\cA(X)\to \cA(\cC_X))$.  Here $\cC_X\subset\cC$
is the full subcategory whose objects are the elements of $X$. Since
$\cA(X)$ is separable, the natural maps $C(\cA(\cC))\to
C(\cA(\cC)/\Z^{(\obc)})\to M(\cA(\cC)/\Z^{(\obc)})$ are
quasi-isomorphisms, by Lemma \ref{lem:sep}. Put
\[
C(\cC)=C(\cA(\cC)/\Z^{(\obc)})
\]
Note that this cyclic module is functorial on $\Z-\cat$, even though
as we have seen in \eqref{fun:ac}, $\cA(-)$ is only functorial on
$\inj-\Z-\cat$. The cyclic module $C(\cC)$ is often called the
\emph{$\Z$-linear cyclic nerve} of $\cC$ (\cite[\S4.2]{LR1}). The
cyclic theory of a $\Z$-linear category $\cC$ is that of $C(\cC)$.
Note that if $R$ is a unital ring considered as a $\Z$-linear
category with one object, then $C(R)$ is the same cyclic module that
was  defined above.
\end{ex}

\begin{rem}\label{rem:mismofi}
As explained in Example \ref{ex:cc} above, the projection
\[
C(\cA(\cC))\to C(\cA(\cC)/\Z^{(\obc)})=C(\cC)
\]
is a quasi-isomorphism. This map has a left inverse $C(\cC)\to
C(\cA(\cC))$; namely the inclusion
\[
C(\cC)_n=\bigoplus_{(c_0,\dots,c_n)\in\obc^{n+1}}\hom_\cC(c_1,c_0)\otimes\dots\otimes\hom_\cC(c_0,c_n)\subset\cA(\cC)^{\otimes n+1}=C(\cA(\cC))_n
\]
This inclusion is a quasi-isomorphism, and is compatible with the map \eqref{map:phi}; indeed they both form part of a map of distinguished triangles:
\[
\xymatrix{C(\cC)\ar[r]\ar[d]^{inc}& C(\tilde{\cC})\ar[r]\ar[d]^\phi& C(\Z\langle\obc\rangle)\ar[d]\\
C(\cA(\cC))\ar[r]& C(\widetilde{\cA(\cC)})\ar[r]& C(\Z)}
\]
\end{rem}

\begin{prop}\label{prop:hhstands}
Hochschild and cyclic homology satisfy the standing assumptions.
\end{prop}
\begin{proof}
M. Wodzicki showed in \cite{wodex} that the $HH$-excisive rings coincide with the $HC$-excisive ones, and that
they are the $H$-unital rings, whose definition is recalled in Subsection \ref{subsec:huni} of the Appendix.
Rings with local units, and more generally $s$-unital rings are $H$-unital by \cite[Cor. 4.5]{wodex}. 
By Proposition \ref{prop:crossh}, $A\rtimes G$
is $H$-unital for every $H$-unital $G$-ring $A$. It is clear from the definition of $H$-unitality that $H$-unital
rings are closed under filtering colimits. Thus it suffices to verify Standing Assumption iii) for finite $X$,
and this is \cite[Corollary 9.8]{wodex}. Assumption iv) follows from Example \ref{ex:cc}. Finally
assumption v) is proved in Proposition \ref{prop:barsumh}.
\end{proof}

\section{Assembly for Hochschild and cyclic homology}\label{sec:asshh}

Let $G$ be a group, $S$ a $G$-set and $R$ a unital $G$-ring. We have a direct sum decomposition
\begin{equation}\label{eq:decomp}
C(R\rtimes \cG^G(S))=\bigoplus_{(g)\in\con(G)}C_{(g)}(R\rtimes
\cG^G(S))
\end{equation}
here $\con(G)$ is the set of conjugacy classes and $C_{(g)}(R\rtimes
\cG^G(S))_n$ is generated by those elementary tensors $x_0\rtimes
g_0\otimes\dots \otimes x_n\rtimes g_n$ with $g_0\cdots g_n\in(g)$.
If $g\in G$, we write $R_g$ for $R$ considered as a bimodule over
itself with the usual left multiplication and the right
multiplication given by $x\cdot r=xg(r)$. In Proposition
\ref{prop:reilu}, we shall need the absolute Hochschild homology of
$R$ with coefficients in $R_g$. In general if $M$ is any
$R$-bimodule, we write $HH(R,M)$ for the Hochschild complex with
coefficients in $M$ (\cite[\S9.1.1]{chubu}).

Proposition \ref{prop:reilu} below computes the $G$-equivariant
homology of a $G$-simplicial set $X$ with coefficients in $HH(R)$
for an arbitrary unital $G$-ring $R$. The case when $G$ acts
trivially on $R$ was obtained by L\"uck and Reich in \cite{LR1}. The
case when $X$ is a point may be regarded as a transport groupoid
version of Lorenz' computation of $HH(R\rtimes G)$ \cite{lor};
$HC(R\rtimes G)$ was computed by Fe\u\i gin and Tsygan in
\cite{FTA}. Our proof uses ideas from each of the three cited articles.

\begin{lem}\label{lem:ft}
Let $G$ be a group, $S$ a $G$-set, $g\in G$, and $Z_g\subset G$ the
centralizer of $g$. Write $\cE Z_g:=\cE(Z_g,\{1\})$. Then there is a natural weak equivalence of simplicial abelian
groups
\begin{equation}\label{map:ft}
\Z[\cE Z_g]\otimes_{\Z[Z_g]}\left(\Z[S^g]\otimes
HH(R,R_g)\right)\weq HH_{(g)}(R\rtimes\cG^G(S))
\end{equation}
Taking homotopy groups one obtains the relative $\tor$ groups \cite[8.7.5]{chubu}:
\[
\pi_*HH_{(g)}(R\rtimes\cG^G(S))=\tor_*^{[(R\otimes R^{op})\rtimes
Z_g]/\Z}(R,R_g)
\]
\end{lem}
\begin{proof}
Note that
\begin{gather*}
\left[
\Z[\cE Z_g]\otimes_{\Z[Z_g]}\left(\Z[S^g]\otimes
HH(R,R_g)\right)\right]_n=\Z[Z_g]^{\otimes n}\otimes\Z[S^g]\otimes R^{\otimes n+1}\\
\end{gather*}
Define a map
\begin{gather*}
\alpha:\Z[\cE Z_g]\otimes_{\Z[Z_g]}\left(\Z[S^g]\otimes
HH(R,R_g)\right)\to HH_{(g)}(R\rtimes\cG^G(S))\\
\alpha(z_1\otimes\dots\otimes z_n\otimes s\otimes x_0\otimes\dots\otimes x_n)=\\
x_0\rtimes (z_1\cdots z_n)^{-1}g\otimes (z_1\cdots z_n)(x_1)\rtimes z_1\otimes
(z_2\cdots z_n)(x_2)\rtimes z_2\otimes\dots\otimes z_n(x_n)\rtimes z_n\in\\
\hom_{R\rtimes\cG^G(S)}(z_1\cdots z_ns,s)\otimes\dots\otimes\hom_{R\rtimes\cG^G(S)}(s, z_ns)
\end{gather*}
One checks that $\alpha$ is a simplicial homomorphism. Write  $U=(R\otimes R^{op})\rtimes
Z_g$. To prove that $\alpha$ is a weak equivalence, and also that its domain and codomain
both compute $\tor_*^{U/\Z}(R,R_g)$, it suffices to find simplicial resolutions $P\weq R_g$ and $Q\weq R_g$ by relatively projective $U$-modules
and a simplicial module homomorphism $\hat{\alpha}:P\to Q$ covering the identity of $R_g$ and such that $R\otimes_U\hat{\alpha}=\alpha$.
We need some
notation.
Write $\mathrm{E}(Z_g,M)$ for the simplicial $\Z[Z_g]$-module
resolution of a left $Z_g$-module $M$ associated to the cotriple $N\mapsto \Z[Z_g]\otimes
N$ \cite[8.6.11]{chubu}.
Let
$C^{bar}(R,R_g)$ be the bar resolution (\ref{subsec:barcomp}); $Z_g$ acts diagonally on
$\Z[S^g]\otimes C^{bar}(R,R_g)$. Write $P=\mathrm{E}(Z_g,\Z[S^g]\otimes C^{bar}(R,R_g))$ for the diagonal of the bisimplicial
module $([p],[q])\mapsto \mathrm{E}_p(Z_g,\Z[S^g]\otimes C^{bar}(R,R_g)_q)$. By construction, $P\weq R_g$ is a simplicial $U$-module resolution,
and every $U$-module $P_n$ is extended from $\Z$, whence relatively projective. Next, given $k\in G$, consider the simplicial submodule
$V(k)\subset C^{bar}(R\rtimes \cG^G(S))$ generated by the elementary tensors
\begin{multline*}
x_0\rtimes h_0\otimes \dots\otimes x_{n+1}\rtimes h_{n+1}\in\\
\hom_{R\rtimes\cG^G(S)}(h_1\cdots h_{n+1}s,ks)\otimes\dots\otimes
\hom_{R\rtimes\cG^G(S)}(s,h_{n+1}s)
\end{multline*}
with $s\in S$ and $h_0\cdots h_{n+1}=k$ ($n\ge 0$). Put $Q=V(g)$;
note $Q$ is stable under multiplication by elements of the form
$a\rtimes z\otimes b\rtimes z^{-1}\in (R\rtimes G)\otimes (R\rtimes
G)^{op}$ with $z\in Z_g$. We have a ring homomorphism
\begin{gather*}
\iota:U\to (R\rtimes G)\otimes (R\rtimes G)^{op}\\
a\otimes b\rtimes z\mapsto a\rtimes z\otimes b^z\rtimes z^{-1}
\end{gather*}
Thus $Q$ is a simplicial $U$-module. We have an isomorphism of graded $U$-modules
\begin{gather*}
\theta:\bigoplus_{\bar{h}\in G/Z_g}\bigoplus_{k\in G}U\otimes V(k)\to Q\\
\theta((a\otimes b\rtimes z)\otimes v))=\iota(a\otimes b\rtimes z)\cdot (1\rtimes h\otimes v\otimes 1\rtimes (hk)^{-1}g)
\end{gather*}
In particular each $U$-module $Q_n$ is extended from $\Z$. Next observe that the augmentation of $C^{bar}(R\rtimes G)$ restricts to
an augmentation

\begin{equation}\label{map:resolQ}
Q\to R\rtimes g\cong R_g
\end{equation}
and that the canonical contracting chain homotopy $x\mapsto 1\otimes x$ induces
a contracting homotopy for \eqref{map:resolQ}. Thus \eqref{map:resolQ} is a simplicial
resolution by relatively projective $U$-modules. Consider the map
\begin{gather*}
\hat{\alpha}: P\to Q\\
\hat{\alpha}(z_0\otimes\dots\otimes z_n\otimes s\otimes x_0\otimes\dots\otimes x_{n+1})=\\
(z_0\cdots z_n)(x_0)\rtimes z_0\otimes (z_1\cdots z_n)(x_1)\rtimes z_1\otimes\dots\otimes z_n(x_n)\rtimes z_n\otimes x_{n+1}\rtimes (z_0\cdots z_n)^{-1}g\in\\
\hom_{R\rtimes\cG^G(S)}(z_0^{-1}s,s)\otimes\dots\otimes\hom_{R\rtimes\cG^G(S)}(s, (z_0\cdots z_n)^{-1}s)
\end{gather*}
One checks that $\hat{\alpha}$ is a simplicial $U$-module homomorphism covering the identity of $R_g$ and that $R\otimes\hat{\alpha}=\alpha$, concluding
the proof.
\end{proof}

\begin{prop}\label{prop:reilu}(Compare \cite[9.16]{LR1})
Let $G$ be a group, $X\in\fS^G$. For each $\xi\in\con(G)$ choose a
representative $g_\xi$. Then there is an isomorphism
\[
\bigoplus_{\xi\in\con(G)}\bH_*(Z_{g_\xi},\Z[X^{g_\xi}]\otimes
HH(R,R_{g_\xi}))\iso H^G_*(X,HH(R))
\]
natural in $X$ and $R$, which depends on the choice of
representatives $\{g_\xi:\xi\in\con(G)\}$. Here $\bH(Z_g, -)$ is
hyperhomology of complexes of $Z_g$-modules, and the tensor product
is equipped with the diagonal action.
\end{prop}
\begin{proof}
By \eqref{eq:decomp} we have
\[
H^G_*(X,HH(R))=\bigoplus_{\xi\in\con(G)}H^G_*(X,HH_{(\xi)}(R))
\]
By Lemma \ref{lem:ft} and the definition of equivariant homology, if
$g\in\xi$, then

\begin{gather*}
H^G(X,HH_{\xi}(R))=\int^{\org}HH_{\xi}(R\rtimes\cG^G(G/H))\otimes \Z[\map(G/H,X)]\\
\lweq \int^{\org}\left(\Z[\cE Z_{g}]\otimes_{\Z[Z_{g}]}\left[\Z[\map (G/\langle g\rangle,G/H)]\otimes HH(R,R_g)\right]\right)\otimes\Z[\map(G/H,X)]=\\
\Z[\cE Z_{g}]\otimes_{\Z[Z_{g}]}\left(\Z[X^g]\otimes HH(R,R_g)\right)=\\
\bH(Z_{g},\Z[X^{g}]\otimes HH(R,R_g))
\end{gather*}
\end{proof}

\begin{prop}\label{prop:asshh} Let $G$ be a group, $\cF$ a family of subgroups of $G$ and $R$ a unital $G$-ring. Assume that $\cF$
contains all cyclic subgroups of $G$. Then $H^G(-,HH(R))$ preserves
$(G,\cF)$-weak equivalences. In particular, the assembly map
\[
H^G_*(\cE(G,\cF), HH(R))\to HH_*(R\rtimes G)
\]
is an isomorphism. The analogue statements for cyclic homology also
hold.
\end{prop}
\begin{proof} The first statement about Hochschild homology follows from \ref{prop:reilu}, and
the fact that if $K$ is a group, then $\bH(K,-)$ preserves quasi-isomorphisms. The second
follows from the first and the fact that $\cE(G,\cF)\to *$ is an
equivalence. Next, given a cyclic module $M$, consider the
subcomplex
\[
\xc^{n}(M)=\ker(S^n:HC(M)\to HC(M)[-2n])
\]
Note that
\[
0=\xc^0(M)\subset
HH(M)=\xc^1(M)\subset\xc^2(M)\subset\dots\subset\bigcup_n\xc^n(M)=HC(M)
\]
is an exhaustive filtration. Hence, because  $H^G(X,-)$ preserves
filtering colimits ($X\in\fS^G$), to prove the statement of the
lemma for cyclic homology, it is sufficient to show that for each
$n$, $H^G(-,\xc^n(R))$ preserves $(G,\cF)$-equivalences of
$G$-simplicial sets. Observe that if $M$ is a cyclic module, then we
have an exact sequence
\[
0\to \xc^n(M)\to\xc^{n+1}(M)\to HH(M)[-2n]\to 0
\]
Using the sequence above and what we have already proved, one shows
by induction that $H^G(-,\xc^n(R))$ preserves
$(G,\cF)$-equivalences. This finishes the proof.
\end{proof}
\section{The Chern character and infinitesimal $K$-theory}\label{sec:ch}
\numberwithin{equation}{subsection}
\subsection{Nonconnective Chern character}
Let $\cC$ be a $\Z$-linear category. By results of Randy McCarthy \cite[\S3.3 and \S4.4]{machc} we have a Chern character
\begin{equation}\label{map:machern}
K^Q(\idem\cC_\oplus)\to |\tau_{\ge 0}HN(\idem\cC_\oplus)|
\end{equation}
going from the $K$-theory simplicial set to the simplicial set
obtained via the Dold-Kan correspondence from the good truncation of
the negative cyclic complex without negative terms. In this section we use
this to obtain a map
\[
K(\cC)\to |HN(\cC)|
\]
going from the nonconnective $K$-theory spectrum of Section \ref{sec:kth} to the spectrum obtained from
the negative cyclic complex via Dold-Kan correspondence. We shall need the following result of McCarthy.

\begin{prop}\label{prop:mac}\cite[Thm. 2.3.4]{machc}
Let $\cD$ be a $\Z$-linear category and $\cC\subset\cD$ a full subcategory. Assume that for every object $d\in\cD$ there
exists an $n=n(d)$, a finite sequence $c_1,\dots,c_n$ of objects of $\cC$, and morphisms $\phi_i:c_i\to d$ and $\psi_i:d\to c_i$
such that $\sum_i\phi_i\psi_i=1_d$. Then the inclusion functor $\cC\to\cD$ induces a quasi-isomorphism $C(\cC)\to C(\cD)$.
\end{prop}

\begin{lem}\label{lem:gamatenso}
Let $\cC$ be an additive category, and let $\bu$ be the only object of $\Gamma(\Z)$. Consider the functor
\begin{gather*}
\mu:\Gamma\Z\otimes\cC\to \Gamma(\cC)\\
\mu(\bu,c)=(c,c,\dots), \qquad \mu(f\otimes\alpha)_{ij}=f_{ij}\alpha
\end{gather*}
Then
\item[i)] The functor $\mu$ is fully faithful.
\item[ii)] Let $F(-)$ be as in \eqref{F(x)}. For every object $x\in\Gamma(\cC)$ there
exist morphisms $\phi_c:\mu(\bu, c)\to x$ and $\psi_c:x\to \mu(\bu, c)$, $c\in F(x)$
such that $\sum_{c\in F(x)}\phi_c\psi_c=1_x$.
\item[iii)] The functor $\mu$ induces a fully faithful functor $\bar{\mu}:\Sigma\otimes\cC\to \Sigma(\cC)$.
\end{lem}
\begin{proof}
Part i) is proved in \cite[Lemma 4.7.1]{CT} for the case when $\cC$ has only one object; the same
argument applies in general. To prove ii), let $x\in\Gamma(\cC)$ be an object.  If $c\in F(x)$, write $I(c)=\{n\in\N:x_n=c\}$, and let $\chi_{I(c)}$ be the characteristic function. Put
\[
\phi_c:\mu(\bu, c)\to x,\quad \psi_c:x\to\mu(\bu, c),\quad
(\phi_c)_{i,j}=(\psi_c)_{i,j}=\delta_{i,j}\chi_{I(c)}(j)1_{c}
\]
One checks that
\[
\sum_{c\in F(x)}\phi_c\psi_c=1_x
\]
This proves ii). Next, consider the exact sequence
\[
0\to M_\infty\Z\to \Gamma\Z\overset\pi\to \Sigma\Z\to 0
\]
As is explained in \cite[pp 92]{CT}, it follows from results of N\"obeling \cite{nob} that the sequence above is split as a sequence of abelian groups. Hence if $c,d\in\cC$, then
\[
\ker(\pi\otimes 1:\hom_{\Gamma\Z\otimes\cC}((\bu,c),(\bu,d))\to\hom_{\Sigma\Z\otimes\cC}((\bu,c),(\bu,d)))=M_\infty\Z\otimes\hom_{\cC}(c,d)
\]
Next observe that if $\alpha\in\hom_{\cC}(c,d)$ and $f\in M_\infty\Z$, then $\mu(f\otimes\alpha)$ is a finite morphism. Hence $\mu$ passes to the quotient, inducing a functor $\bar{\mu}:\Sigma\Z\otimes\cC\to\Sigma(\cC)$. If $c,d\in\ob\cC$ and we put $x=\mu(\bu, c)$, $y=\mu(\bu, d)$ then we have a map of exact sequences
\[
\xymatrix{0\to M_\infty\Z\otimes\hom_{\cC}(c,d)\ar[d]\ar[r]&\Gamma\Z\otimes\hom_{\cC}(c,d)\ar[r]\ar[d] &\Sigma\Z\otimes\hom_{\cC}(c,d)\ar[d]\to 0\\
0\to\hom_{\mathrm{Fin}(\cC)}(x,y)\ar[r] &\hom_{\Gamma(\cC)}(x,y)\ar[r] &\hom_{\Sigma(\cC)}(x,y)\to 0}
\]
Here $\mathrm{Fin}(\cC)\subset\Gamma(\cC)$ is the subcategory of
finite morphisms. The second vertical map is an isomorphism by part
i). In particular the first map is injective; furthermore, one
checks that it is onto. It follows that the third vertical map is an
isomorphism; this proves iii).
\end{proof}
\begin{prop}\label{prop:hckarseq} Let $\cC$ be a $\Z$-linear category. Then:
\item[i)] $C(\cC)\to C(\cC_\oplus)$ is a quasi-isomorphism.
\item[ii)] If $\cC$ is additive, then $C(\cC)\to C(\idem\cC)$ is a quasi-isomorphism.
\item[iii)] The maps $C(\Gamma(\Z)\otimes\cC)\to C(\Gamma(\cC))$, $C(\Gamma(\cC))\to 0$ and $C(\Sigma(\Z)\otimes\cC)\to C(\Sigma(\cC))$ are quasi-isomorphisms.
\item[iv)] The sequence
\[
\idem\cC_\oplus\to \Gamma\cC_\oplus\to \Sigma\cC_\oplus
\]
induces a distinguished triangle of Hochschild, cyclic, negative cyclic and periodic cyclic complexes.
\end{prop}
\begin{proof} The first two assertions are straightforward applications of Proposition \ref{prop:mac}. That
$C(\Gamma(\Z)\otimes\cC)\to C(\Gamma(\cC))$ and
$C(\Sigma(\Z)\otimes\cC)\to C(\Sigma(\cC))$ are quasi-isomorphisms
follows from Proposition \ref{prop:mac} and Lemma
\ref{lem:gamatenso}. In particular we have quasi-isomorphisms
\[
\xymatrix{C(\Gamma(\cC))&
C(\cA(\Gamma\Z\otimes\cC)/\cA(\ob(\Gamma\Z\otimes\cC))\ar[l]_(.65)\sim&
C(\cA(\Gamma\Z\otimes\cC))\ar[l]_(.35)\sim\ar@{=}[d]&\\
&&C(\Gamma\cA(\cC))\ar[r]^(.5)\sim &HH(\Gamma\cA(\cC))}
\]
But because $\cA(\cC)$ is $H$-unital, $HH(\Gamma\cA(\cC))$ is
acyclic by \cite[Thm. 10.1]{wodex}. To prove iv), consider the
commutative diagram
\begin{equation}\label{seq:mapcone}
\xymatrix{\cC\ar[d]\ar[r]&\Gamma\cC\ar[r]\ar[d] &\Sigma\cC\ar[d]\\
\idem\cC_\oplus\ar[r]& \Gamma\cC_\oplus\ar[r]& \Sigma\cC_\oplus}
\end{equation}
By i) and ii), the first vertical map induces quasi-isomorphisms of
cyclic modules. If $R$ is a unital ring flat as a $\Z$-module, then
the quasi-isomorphism $C(\cC)\to C(\cC_\oplus)$ of i) induces a
quasi-isomorphism $C(R\otimes\cC)=C(R)\otimes C(\cC)\to
C(R\otimes\cC_\oplus)$. In particular this applies when
$R=\Gamma\Z,\Sigma\Z$. Hence the second and third vertical maps in
\eqref{seq:mapcone} are quasi-isomorphisms as well, by iii). By
Lemma \ref{lem:sep} the cyclic modules of the top row are
quasi-isomorphic to the cyclic modules of their associated rings;
thus iv) reduces to the fact, proved in \cite[\S10]{wodex}, that the
sequence
\[
\xymatrix{C(\cA(\cC))\ar[r]&C(\Gamma\cA(\cC))\ar[r] &C(\Sigma\cA(\cC))}
\]
induces distinguished triangles for $HH$, $HC$, $HN$ and $HP$.
\end{proof}

Let $\cC^{(n)}$ be as in \eqref{cn}. Observe that by Proposition
\ref{prop:hckarseq}, we have an equivalence $|\tau_{\ge
0}HN(\cC^{(n)})|\weq |\tau_{\ge 0}HN(\cC)[+n]|$. Composing with the
map $K^Q(\cC^{(n)})\to |\tau_{\ge 0}HN(\cC^{(n)})|$ we obtain a
sequence $K^Q(\cC^{(n)})\to \tau_{\ge 0}HN(\cC)[+n]$ which induces a
map of nonconnective spectra
\begin{equation}\label{map:ch}
ch:K(\cC)\to |HN(\cC)|
\end{equation}
\begin{rem}\label{rem:hc=hc}
If $\cC$ has only one object, then the Chern character \eqref{map:ch} agrees
with the usual one. This follows from \eqref{kveq} and the ring
analogue of Proposition \ref{prop:hckarseq}, part iv), proved in
\cite[\S10]{wodex}. Furthermore, for any $\Z$-linear category $\cC$,
the character \eqref{map:ch} agrees with that of $\cA(\cC)$. Indeed,
$K(\cA(\cC))\weq K(\cC)$ by Proposition \ref{prop:kkpw}, and the proof of
Proposition \ref{prop:hckarseq} makes clear that the homology
sequences of iv) are equivalent to the corresponding sequences for
$\cA(\cC)$.
\end{rem}

\subsection{$K^{\nil}$ and the relative Chern character}

Let $E:\Z-\cat\to\spt$ be a functor and $\cC\in\Z-\cat$. Consider the homotopy fiber
\[
E^{\nil}(\cC)=\hofi(E(\cC)\to E(\cC\otimes\Z^{\Delta^\bu}))
\]
Write
\[
ch^{\Delta}:KH(\cC)=K(\cC\otimes\Z^{\Delta^\bu})\to HN(\cC\otimes\Z^{\Delta^\bu})
\]
for the result of applying the map $K\to HN$ dimensionwise.
We have a map of spectra $ch^{\nil}:K^{\nil}(\cC)\to HN^{\nil}(\cC)$ which fits into a map of homotopy
fibrations
\[
\xymatrix{K^{\nil}(\cC)\ar[d]^{ch^{\nil}}\ar[r]& K(\cC)\ar[d]^{ch}\ar[r]& KH(\cC)\ar[d]^{ch^{\Delta}}\\
|HN^{\nil}(\cC)|\ar[r]&|HN(\cC)|\ar[r]&|HN(\cC\otimes\Z^{\Delta^\bu})|}
\]

\begin{lem}
Let $\cC$ be a $\Q$-linear category. Then there is a homotopy commutative diagram with vertical weak equivalences
\[
\xymatrix{|HN^{\nil}(\cC)|\ar[d]_\iota\ar[r]&|HN(\cC)|\ar[r]\ar@{=}[d]&|HN(\cC\otimes\Z^{\Delta^\bu})|\ar[d]^\wr\\
          \Omega^{-1}|HC(\cC)|\ar[r]& |HN(\cC)|\ar[r]&HP(\cC)}
\]
\end{lem}
\begin{proof} By Example \ref{ex:cc}, this is a statement about the $\Q$-algebra $\cA(\cC)$. The latter is proved in
\cite[Theorem 4.1]{gw}.
\end{proof}

By \cite[Prop. 1.6]{kh}, if $A$ is a $\Q$-algebra the groups $K^{\nil}_*(A)$ are $\Q$-vectorspaces. Hence for every ring $A$ we have
a map
\begin{equation}\label{map:q}
q:K^{\nil}(A)\otimes\Q\to K^{\nil}(A\otimes\Q)
\end{equation}
which is an equivalence if $A$ is a $\Q$-algebra.
We write
\begin{gather}
\nu=\iota ch^{\nil}(-\otimes\Q)q:K^{\nil}(\cC)\otimes\Q\to \Omega^{-1}|HC(\cC\otimes\Q)|\lweq \Omega^{-1}|HC(\cC)|\otimes\Q\label{map:nu}\\
K^{\ninf}(\cC)=\hofi(\nu)\nonumber
\end{gather}
We remark that $\nu$ is a variant of the relative character introduced by Weibel in \cite{wenil}.

\begin{prop}\label{prop:ninfexci}
$K^{\ninf}:\Z-\cat\to\spt$ satisfies the standing assumptions. In addition, it is excisive and $K_*^{\ninf}$ commutes with filtering colimits.
\end{prop}
\begin{proof}
It is proved in \cite{kabi} that
\begin{equation}\label{def:kinf}
K^{\inf,\Q}:=\hofi(ch^\Q:K(-)\otimes\Q\to HN(-\otimes\Q))
\end{equation}
is excisive; it follows that $K^{\ninf}$ is excisive too. Next observe that $K^\ninf$ satisfies iii) and v) of the standing assumptions \ref{stan} for unital rings, and iv) for all $\cC\in\Z-\cat$, since both $K^\nil$ and $HC$ do. Because $K^\ninf$ is excisive, this implies that it satisfies all standing assumptions, by Remark \ref{rem:exci}. Finally $K_*^\ninf$ commutes with filtering colimits because both $K_*^\nil$ and $HC_*$ do.
\end{proof}

\section{Rings of polynomial functions on a simplicial set}\label{sec:zx}

\subsection{Finiteness}
An object $K$ in a category $\fC$ is \emph{small} if $hom_\fC(K,-)$ preserves colimits.
If $\fC=\fS$, then $X$ is small if and only if it has only a finite number of nondegenerate simplices, or, equivalently, if there
exists a finite set of nonnegative integers $n_1,\dots,n_r$ and a surjection
\[
\coprod_{i=1}^r\Delta^{n_i}\onto X
\]
Small simplicial sets are called \emph{finite}. Similarly, a $G$-simplicial set is small if there are $n_1,\dots,n_r\ge 0$ and
a $G$-equivariant surjection
\[
\coprod_{i=1}^r\Delta^{n_i}\times G\onto X
\]
Let $\cF$ be a family of subgroups of $G$. A \emph{finite} $(G,\cF)$-complex is $G$-simplicial set
obtained by attaching finitely many cells of the form $\Delta^n\times G/H$ with $H\in\cF$. A $G$-finite
simplicial set is a finite $(G,\fall)$-complex. The concept of $G$-finiteness is the simplicial set version of the
concept of $G$-compactness. Indeed one checks that a $G$-simplicial set $X$ is $G$-finite if and only if $X/G$ is finite as a simplicial set.

\subsection{Locally finite simplicial sets}
If $X$ is a simplicial set and $\sigma\in X$ is a simplex, we write
$<\sigma>\subset X$ for the simplicial subset generated by $\sigma$. We have
\[
<\sigma>_n=\{\alpha^*(\sigma):\alpha\in\hom([n],[\dim\sigma])\}
\]
The {\em star} of $\sigma$ is the following set of simplices of $X$:
\[
\St(\sigma)=\St_X(\sigma)=\{\tau\in X:<\tau>\cap <\sigma>\neq\emptyset\}
\]
The {\em closed star} is the simplicial subset
$\cSt(\sigma)=<\St(\sigma)>$ generated by $\St(\sigma)$. If $M$ is a
set of simplices of $X$ we put $\St_X(M)=\cup_{\sigma\in
M}\St_X(\sigma)$, $\cSt_X(M)=<\St_X(M)>$. We also define the {\em
link} of $M$ as $\li(M)=\cSt_X(M)\backslash\St_X(M)$.

\begin{lem}\label{lem:locfin}
Let $X$ be a simplicial set; write $NX$ for the set of nondegenerate simplices. The following are equivalent.
\item[i)]
$(\forall\sigma\in X)\{\tau\in NX:<\tau>\supset<\sigma>\}$ is a finite set.
\item[ii)] For every $\sigma\in X$, $\cSt_X(\sigma)$ is a finite simplicial set.
\end{lem}
\begin{proof} If $\sigma\in X$, then $<\sigma>$ has finitely many nondegenerate simplices, and thus the set $\{<\tau>\cap<\sigma>:\tau\in X\}$
is finite. Hence if i) holds, there are finitely many $\tau\in NX$ such that $<\tau>\cap <\sigma>\ne\emptyset$; in other words, $NX\cap \St_X(\sigma)$ is a finite set, and therefore $\cSt_X(\sigma)$ is a finite simplicial set. Thus i)$\Rightarrow$ii). Next note that $<\tau>\supset<\sigma>$ implies
$\tau\in\St_X(\sigma)$, whence ii)$\Rightarrow$i).
\end{proof}

We say that $X$ is {\em locally finite} if it satisfies the equivalent conditions of the lemma above.

\subsection{Rings of polynomial functions on a simplicial set}\label{subsec:rings}
 If $X$ is a simplicial set and $A$ is a ring, we put
 \[
 A^{X}=\hom_\fS(X,A^{\Delta^\bu})
 \]
 The simplicial ring $A^{\Delta^\bu}=A\otimes\Z^{\Delta^\bu}$ is defined as in \eqref{simp:delta}.
 Note $X\mapsto A^{X}$, $f\mapsto f^*$ gives a functor $\fS^{op}\to\ring$. By its very definition, the functor $A^{-}$ sends colimits to limits; if $I$ is a small category
 and $X:I\to\fS$ is a functor, then
\[
 A^{\colim_i X_i}=\lim_i A^{X_i}
 \]
\begin{ex} Any simplicial set $X$ is the union of the subobjects generated by each of its nondegenerate simplices; in symbols
\[
X=\colim_{\sigma\in NX}<\sigma>
\]
Thus we obtain
\begin{equation}\label{limax}
A^X=\lim_{\sigma\in NX}A^{<\sigma>}=
\{\phi\in\prod_{\sigma\in NX}A^{<\sigma>}:{\phi(\sigma)}_{|<\sigma>\cap<\tau>}={\phi(\tau)}_{|<\sigma>\cap<\tau>},\sigma,\tau\in NX\}
\end{equation}
\end{ex}

If $\phi\in A^X$, then its {\em support} is
\[
\supp(\phi)=<\{\sigma\in X:\phi(\sigma)\ne 0\}>
\]
Note that if $\phi,\psi\in A^X$ and $f:X\to Y$ is a simplicial map,
then
\begin{equation}\label{supprop}
\supp(\phi\cdot\psi)\subset\supp(\phi)\cap\supp(\psi)\qquad \supp(f^*(\phi))\subset f^{-1}(\supp(\phi))
\end{equation}
We say that $\phi$ is \emph{finitely supported} if $\supp(\phi)$ is a finite simplicial set. Note $\phi$ is finitely supported if and only if there is
only a finite number of nondegenerate simplices $\sigma$ such that $\phi(\sigma)\ne 0$. Put
\[
A^{(X)}=\{f\in A^X:\supp(f) \text{ is finite.}\}
\]
If $X$ is finite, then clearly $A^X=A^{(X)}$. In general, $A^{(X)}\subset A^X$ is a two-sided ideal, by \eqref{supprop}. We remark that
if $f: X\to Y$ is an arbitrary map of simplicial sets, then the associated ring homomorphism $f^*:A^Y\to A^X$ does not necessarily send
$A^{(Y)}$ into $A^{(X)}$. However, if $f$ happens to be {\em proper},
i.e. if $f^{-1}(K)$ is finite for every finite $K\subset Y$, then $f^*(A^{(Y)})\subset A^{(X)}$, by \eqref{supprop}. Hence $A^{(-)}$ is a functor on the category of simplicial sets and proper maps. Next we consider the behaviour of this functor with respect to colimits. First of all, if
 $\{X_i\}$ is a family of simplicial sets, then we have
\begin{equation}\label{coprodfinsupp}
A^{(\coprod X_i)}=\bigoplus_iA^{(X_i)}
\end{equation}
Here $\bigoplus$ indicates the direct sum of abelian groups, equipped with coordinatewise multiplication. Second, $A^{(-)}$ maps coequalizers of proper maps
to equalizers; if $\{f_j:X\to Y\}$ is a family of proper maps, then
\begin{equation}\label{coeqfinsupp}
A^{(\coequ_j \{f_j:X\to Y\})}=\equ_j\{f_j^*:A^{(Y)}\to A^{(X)}\}
\end{equation}
Next recall that if $I$ is a small category and $X:I\to \fS$ is a functor, then the colimit of $X$ can be computed as a coequalizer:
\[
\colim_iX_i=\coequ(\coprod_{\alpha\in\Ar(I)}X_{s(\alpha)}\overset{\partial_0}{\underset{\partial_1}{\rightrightarrows}}\coprod_{i\in\Ob(I)} X_i)
\]
Here $\Ob(I)$ and $\Ar(I)$ are respectively the sets of objects and of arrows of $I$, and if $\alpha\in\Ar(I)$ then $s(\alpha)\in\Ob(I)$ is its source; we
also write $r(\alpha)$ for the range of $\alpha$. The maps $\partial_0$ and $\partial_1$ are defined as follows. The restriction of $\partial_i$ to the copy of $X_{s(\alpha)}$ indexed by $\alpha$ is the inclusion $X_{s(\alpha)}\subset\coprod_jX_j$ if $i=0$ and the composite of $X(\alpha)$ followed by the inclusion $X_{r(\alpha)}\subset\coprod_jX_j$ if $i=1$. The conditions that $\partial_0$ and $\partial_1$ be proper are equivalent to the following
\begin{itemize}
\item[$\partial_0$)] Each object of $I$ is the source of finitely many arrows.
\item[$\partial_1$)] Each object of $I$ is the range of finitely many arrows, and $X$ sends each map of $I$ to a proper map.
\end{itemize}
\begin{ex}\label{ex:locfinex}
For example the functor $\sigma\mapsto <\sigma>$ from the set of nondegenerate simplices of $X$, ordered by $\sigma\le\tau$ if $<\sigma>\subset<\tau>$,  always satisfies $\partial_1$; condition
$\partial_0$ is precisely condition i) of Lemma \ref{lem:locfin}. Hence $\partial_0$ is satisfied if and only if
$X$ is locally finite, and in that case we have
\[
A^{(X)}=\equ(\bigoplus_{\sigma\in
NX}A^{<\sigma>}\overset{\partial_0^*}{\underset{\partial_1^*}{\rightrightarrows}}
\underset{\sigma,\tau\in
NX}{\bigoplus_{<\tau>\subset<\sigma>,}}A^{<\tau>})
\]
\end{ex}

\begin{lem}\label{lem:zxfree}
If $X$ is a locally finite simplicial set, then $\Z^{(X)}$ is a free abelian group.
\end{lem}
\begin{proof}
By \cite[3.1.3]{CT} the lemma is true when $X$ is finite. Hence if $X$ is any simplicial set, and
$\sigma\in X$ is a simplex, then
$\Z^{<\sigma>}$ is free. If $X$ locally finite, then by Example \ref{ex:locfinex}, $\Z^{(X)}$
is a subgroup of a free group, and therefore it is free.
\end{proof}

\subsection{Extending polynomial functions}

\begin{thm}\label{thm:supp} Let $X$ be a simplicial set, $Y\subset X$ a simplicial subset and $A$ a ring. Let $\phi\in A^{Y}$ and $K=\supp\phi$. Then there
exists $\psi\in A^X$ with $\supp\psi\subset\cSt_XK$ such that
$\psi_{|\li_X(K)}=0$ and $\psi_{|Y}=\phi$.
\end{thm}
\begin{proof}  We have $K\subset \St_YK\subset \cSt_YK$, whence $\phi_{|\li_Y(K)}=0$. Note $\St_XK\cap Y=\St_YK$; thus $\phi$ vanishes on
$\li_X(K)\cap Y$. Hence we may extend $\phi$ to a map
$\phi':Y'=Y\cup\li_X(K)\to A^{\Delta^\bu}$ by $\phi'_{|\li_X(K)}=0$.
Put $Y"=Y\cup\cSt_XK$. Because $Y'\subset Y"$ is a cofibration and
$A^{\Delta^\bu}\onto 0$ is a trivial fibration, we may further
extend $\phi'$ to a map $\phi":Y"\to A^{\Delta^\bu}$. By construction,
$\{\sigma\in X:\phi"(\sigma)\ne 0\}\subset\St_XK$, and $\phi"$
vanishes on $\li_XK$. Hence we may finally extend $\phi"$ to a
map $\psi:X\to A^{\Delta^\bu}$, by letting $\psi(\sigma)=0$ if
$\sigma\notin\cSt_XK$. This concludes the proof.
\end{proof}
\begin{cor}\label{cor:sursupp} If $X$ is locally finite and $Y\subset X$ is a simplicial subset, then the restriction map $A^{(X)}\to A^{(Y)}$ is surjective.
\end{cor}
\begin{proof} It follows from Theorem \ref{thm:supp}, using
\ref{lem:locfin}.
\end{proof}

\begin{prop}\label{prop:axseparate}(Compare \cite[Lemma 2.5]{cunc})
Let $A$ be a nonzero ring. The following are equivalent for a simplicial set $X$.
\item[i)] For every simplex $\sigma\in X$ there exists $\phi\in A^{(X)}$ such that $\phi(\sigma)\ne 0$.
\item[ii)] $X$ is locally finite.
\end{prop}
\begin{proof}
Observe that if $\sigma,\tau\in X$ are simplices with
$<\tau>\supset<\sigma>$ and $\phi\in A^X$ satisfies $\phi(\sigma)\ne
0$, then $\phi(\tau)\ne 0$. If $X$ is not locally finite, then by
Lemma \ref{lem:locfin}, there exists a simplex $\sigma\in X$ which
is contained in infinitely many nondegenerate simplices. By the
previous observation, $\phi(\sigma)=0$ for every $\phi\in A^{(X)}$.
We have proved that i)$\Rightarrow$ii). Assume conversely that $X$
is locally finite, and let $\sigma$ be a simplex of $X$. We want to
show that there exists $\phi\in A^{(X)}$ such that $\phi(\sigma)\ne
0$. We may assume that $\sigma$ is nondegenerate. Let
$Y=<\sigma>\subset X$ be the sub-simplicial set generated by
$\sigma$; by Corollary \ref{cor:sursupp}, it suffices to show that
$A^Y\ne 0$. Now $Y$ is an $n$-dimensional quotient of $\Delta^n$,
whence $S^n=\Delta^n/\partial\Delta^n$ is a quotient of $Y$. So we
may further reduce to showing $A^{S^n}$ is nonzero. Now
\[
A^{S^n}=Z_nA^{\Delta^\bu}=\bigcap_{i=0}^n\ker (d_i:A^{\Delta^n}\to A^{\Delta^{n-1}})
\]
But if $0\ne a\in A$, then $at_0\dots t_n$ is a nonzero element of $Z_nA^{\Delta^\bu}$.
\end{proof}
\subsection{Excision properties}

\begin{prop}\label{prop:tfp}
If $X$ is a locally finite simplicial set, then $\Z^{(X)}$ is $s$-unital.
\end{prop}
\begin{proof}
Let $\phi_1,\dots,\phi_n\in \Z^{(X)}$, and let
$K=\bigcup_i\supp(\phi_i)$. By Theorem \ref{thm:supp} there is $\mu\in\Z^{(X)}$
 such that $\mu_{|K}=1$ is the constant map. Thus
\begin{equation}\label{triplefi}
\phi_i=\phi_i\mu\qquad (\forall i).
\end{equation}
\end{proof}

\begin{prop}\label{prop:exiax}
If $A$ is $K$-excisive and $X$ is locally finite, then
$\Z^{(X)}\otimes A$ is $K$-excisive.
\end{prop}
\begin{proof} Follows from Lemma \ref{lem:zxfree} and Propositions 
\ref{prop:tfp} and \ref{prop:tenso}.
\end{proof}

\begin{rem}
If $A$ is a ring and $X$ a locally finite simplicial set, then there
is a natural map
\[
\Z^{(X)}\otimes A\to A^{(X)}
\]
It was proved in \cite[3.1.3]{CT} that this map is an isomorphism if
$X$ is finite. 
\end{rem}

\section{Proper $G$-rings}\label{sec:propring}
\subsection{Proper rings over a $G$-simplicial set}

Fix a group $G$ and consider rings equipped with an
action of $G$ by ring automorphisms. We write $G-\ring$ for the
category of such rings and equivariant ring homomorphisms. If $C\in
G-\ring$ is commutative but not necessarily unital and $A\in
G-\ring$, then by a {\em compatible $(G,C)$-algebra structure} on
$A$ we understand a $C$-bimodule structure on $A$ such
that the following identities hold for $a,b\in A$, $c\in C$, and $g\in
G$:
\begin{gather}\label{condipropertriv}
c\cdot a=a\cdot c\nonumber\\
c\cdot(ab)=(c\cdot a)b=a(c\cdot b)\nonumber\\
g(c\cdot a)=g(c)\cdot g(a)\nonumber\\
\end{gather}
If $X$ is a $G$-simplicial set and $A\in G-\ring$, then we say that
$A$ is {\it proper} over $X$ if it carries a compatible
$(G,\Z^{(X)})$ algebra structure such that
\begin{equation}\label{condiproper}
\Z^{(X)}\cdot A=A
\end{equation}
If $\cF$ is a family of subgroups
of $G$, we say that $A$ is {\em $(G,\cF)$-proper} if
it is proper over some $(G,\cF)$ complex $X$.

\begin{ex}\label{ex:proper}
Fix a group $G$, a family of subgroups $\cF$ and a $(G,\cF)$-complex
$X$. By Proposition \ref{prop:tfp}, we have $\Z^{(X)}\cdot
\Z^{(X)}=\Z^{(X)}$; thus $\Z^{(X)}$ is proper over $X$. Hence if $A$ is a $G$-ring with a compatible $(G,\Z^{(X)})$-action, then $\Z^{(X)}\cdot A$ is proper over $X$. If $A$ is proper over
$X$, and $B$ is any ring, then $A\otimes B$ is proper over $X$. In
particular, $\Z^{(X)}\otimes B$ is proper. If $T\in\Top$ is the geometric
realization of $X$, and $\F$ is either $\R$ or $\C$, then the $\F$-algebra
$P=C_{\rm comp}(T)$ of compactly supported continuous functions $T\to \F$ is
proper over $X$. To check that $\Z^{(X)}\cdot P=P$, observe that if $f\in P$
then its support meets finitely many maximal simplices; write $K\subset X$ for their
union. By Corollary \ref{cor:sursupp}, there exists $\phi\in\Z^{(X)}$ which is constantly equal
to $1$ on $K$; thus $f=\phi\cdot f\in \Z^{(X)}\cdot P$. 
\end{ex}
Let $X$ be a locally finite simplicial set, and $Y\subset X$ a subobject.
Put
\[
I(Y)=\{\phi:\supp\phi\subset Y\}\triqui\Z^{(X)}
\]
Note that if $\psi\in\Z^{(Y)}$ and $\hat{\psi}\in\Z^{(X)}$ restricts to $\psi$, then
the product
\[
\psi\cdot\phi:=\hat{\psi}\phi
\]
depends only on $\psi$. This defines a compatible action of $\Z^{(Y)}$ on $I(Y)$ which makes
the latter ring proper over $Y$. More generally,
if $A\in\ring$ has a compatible $(G,\Z^{(X)})$-structure, we put
\begin{equation}\label{eq:aye}
A(Y)=I(Y)\cdot A\triqui A
\end{equation}
Observe that $A(Y)$ is an ideal of $A$, proper over $Y$. In particular if
$X$ is a $(G,\cF)$-complex, then $A(Y)$ is $(G,\cF)$-proper for all $Y\subset X$.

\begin{lem}\label{lem:properfil}
Let $A$ be a $G$-ring. Assume that $A$ is $(G,\cF)$-proper. Then $A$ has an exhaustive filtration
$\{A(K)\}$ by ideals such that each $A(K)$  proper over a finite $(G,\cF)$-complex $K$.
\end{lem}
\begin{proof}
By hypothesis, there exists a $(G,\cF)$-complex $X$ such that $A$ is proper over $X$.
Consider the filtration $\{A(K)\}$ where $A(K)$ is defined in \eqref{eq:aye} and $K$ runs among the $G$-finite simplicial subsets of $X$. By the discussion above, $A(K)\subset A$ is an ideal, proper over $K$. It is clear that $\{I(K)\}$ and $\{A(K)\}$ are filtering systems and that $\cup_K I(K)=\Z^{(X)}$. We claim
furthermore that $A=\cup_KA(K)$. By definition of
$\Z^{(X)}$-algebra,  $A=\Z^{(X)}\cdot A$. Hence if $a\in A$, then
there exist $\phi_1,\dots,\phi_n\in \Z^{(X)}$ and $a_1,\dots,a_n\in
A$ such that $a=\sum_i\phi_i a_i$. Hence $a\in A(K)$ for
$K=\cup_iG\cdot\supp(\phi_i)$.
\end{proof}

\begin{lem}\label{lem:festar}(cf. \cite[pp. 51]{ght})
Let $A\in G-\ring$ be proper over a locally finite $G$-simplicial
set $X$, and let $f:X\to Y$ be an equivariant map with $Y$ locally
finite. Then the map $f^*:\Z^Y\to \Z^X$ induces a compatible
$(G,\Z^{(Y)})$-algebra structure on $A$ which makes it proper over
$Y$.
\end{lem}
\begin{proof} We begin by showing that the compatible $(G,\Z^{(X)})$-algebra structure on $A$ extends to a compatible $(G,\Z^X)$-module structure.
By the lemma above, if $a\in A$ then there exists a finite
simplicial subset $K\subset X$ such that $a\in A(K)=I(K)\cdot A$. By
Theorem \ref{thm:supp} there exists $\mu_K\in Z^X$, with $\supp(\mu_K)\subset \cSt(K)$ such that
\begin{equation}\label{condiid}
\mu_Ka=a \qquad\forall a\in A(K).
\end{equation}
Because $X$ is locally finite, $\cSt(K)$ is finite and $\mu_K\in\Z^{(X)}$. Thus we have a
map $A(K)\to I(\cSt(K)))\otimes A(K)$, $a\mapsto \mu_K\otimes a$.
Now $I(\cSt(K))$ is an ideal in $\Z^{X}$ by \eqref{supprop}; using
the multiplication of $\Z^{X}$ we obtain a map
\begin{equation}\label{map:act}
\Z^X\otimes A(K)\to A(\cSt(K)),\quad \phi\otimes a\mapsto
(\phi\cdot\mu_K)a.
\end{equation}
If $L\supset K$, and we choose an element $\mu_L$ as above, then for
$a\in A(K)$ and $\phi\in \Z^X$ we have:
\[
(\phi\cdot\mu_L)\cdot a=(\phi\cdot\mu_L)\cdot (\mu_K\cdot a)=(\phi\cdot\mu_K)a
\]
This shows that \eqref{map:act} is independent of the choice of the
element $\mu_K$ of \eqref{condiid}, and that we have a well-defined
action $\Z^X\otimes A\to A$. Compatibility with the $G$-action
follows from the fact that $g\cdot\mu_K$ is the identity on $g\cdot
K$. The remaining compatibility conditions are immediate. Now $A$
becomes an $\Z^{(Y)}$-module through $f^*$. If $K\subset X$ is a
finite simplicial subset, then $L=f(K)\subset Y$ is finite, and
since $Y$ is locally finite, there is a $\mu_L\in \Z^{(Y)}$ which is
the identity on $L$, and thus $f^*(\mu_L)$ is the identity on $K$.
It follows that the action of $\Z^{(Y)}$ on $A$ satisfies
\eqref{condiproper}. The remaining $(G,\Z^{(Y)})$-compatibility
conditions of \eqref{condipropertriv} are straightforward.
\end{proof}

\subsection{Induction}\label{subsec:ind}
Let $G$ be a group, $H\subset G$ a subgroup and $A$ an $H$-ring.
Consider
\[
\bind_H^G(A)=\{f:G\to A:f(gh)=h^{-1}f(g)\}
\]
Note that $\bind_H^G(A)$ is a $G$-ring with operations defined pointwise,
and where $G$ acts by left multiplication. If $f\in \bind_H^G(A)$
and $x=sH\in G/H$, then the value of $f$ at any $g\in x$ determines
$f$ on the whole $x$; in particular,
\[
\supp(f)\cap sH\ne\emptyset\Rightarrow sH\subset\supp(f)\quad (sH\in G/H)
\]
Hence
\[
\supp(f)=\coprod_{sH\cap\supp(f)\ne\emptyset} sH
\]
Consider the projection $\pi:G\to G/H$. Put
\[
\ind_H^G(A)=\{f\in\bind_H^G(A):\#\pi(\supp(f))<\infty\}
\]
One checks that $\ind_H^G(A)\subset \bind_H^G(A)$ is a subring; we
shall presently introduce some of its typical elements. If $s\in G$, we
write $\chi_s:G\to \Z$ for the characteristic function. If $a\in A$
and $s\in G$, then
\[
\xi_H(s,a)=\sum_{h\in H}h^{-1}(a)\chi_{sh}\in\ind_H^G(A)
\]
Let $r:G/H\to G$ be a pointed section and $\cR=r(G/H)$. Every
element $\phi\in \bind_H^G(A)$ can be written as a formal sum
\begin{equation}\label{uniquexp}
\phi=\sum_{s\in \cR}\xi_H(s,\phi(s))
\end{equation}
Note that $\phi\in\ind_H^G(A)$ if and only if the sum above is finite. In particular
\[
\ind_H^G(A)=\sum_{s\in G,a\in A}\Z\xi_H(s,a)\subset \bind_H^G(A)
\]
Next observe that, for each fixed $s\in G$, the map
\[
\xi_H(s,-):A\to \bind_H^G(A)
\]
is a ring homomorphism. Moreover, we have the following relations
\begin{gather}
g\xi_H(s,a)=\xi_H(gs,a)\label{actixi}\\
 \xi_H(sh,a)=\xi_H(s,ha)\label{hxi}\\
\xi_H(s,a)\xi_H(t,b)=\left\{\begin{matrix}0 &\text{ if } sH\ne tH\\
\xi_H(s,ab)&\text{ if } s=t\end{matrix}\right.\label{prodxi}
\end{gather}
It follows that $(s,a)\mapsto \xi_H(s,a)$ gives a $G$-equivariant map
\[
G\times_H A\to  \ind_H^G(A).
\]
Here $G\times_H A=G\times A/\sim$, where $(g_1,a_1)\sim (g_2,a_2)\iff h=g_1^{-1}g_2\in H$ and $a_1=ha_2$. Extending by linearity we obtain an isomorphism
of left $G$-modules
\[
\Z[G]\otimes_{\Z[H]}A\to \ind_H^G(A)
\]
Thus we may think of $\ind_H^G(A)$ as the $G$-module induced from
the $H$-module $A$ equipped with a ring structure compatible with
that of $A$. In fact \eqref{prodxi} implies that if $r:G/H\to G$ is
a section as above, then
\begin{equation}\label{map:indinonequi}
\Z^{(G/H)}\otimes A\to \ind_H^G(A),\quad \chi_{x}\otimes a\mapsto
\xi_H(r(x),a)
\end{equation}
is a (nonequivariant) ring isomorphism.

\begin{lem}\label{lem:indx}
Let $X$ be an $H$-simplicial set; put
\[
\ind_H^G(X)=G\times_HX
\]
There is a natural, $G$-equivariant isomorphism
$\Z^{(\ind_H^G(X))}\cong \ind_H^G(\Z^{(X)})$.
\end{lem}
\begin{proof} Let $\pi:G\times X\to\ind_H^G(X)$ be the projection. We have a $G$-ring isomomorphism
\[
\theta:\bind_H^G(\Z^{X})\to\Z^{\ind_H^G(X)},\qquad
\theta(f)(\pi(g,x))=f(g)(x)
\]
For $s\in G$ and $\phi\in \Z^{X}$,
\[
\theta(\xi_H(s,\phi))\pi(g,x)=\left\{\begin{matrix}\phi(s^{-1}gx)& \text{ if } g\in sH\\ 0&\text{ else.}\end{matrix}\right.
\]
In particular, for $\theta(\xi_H(s,\phi))$ not to vanish on
$\pi(g,x)$, we must have $g=sh$ and $x\in h^{-1}\{\phi\ne 0\}$ for
some $h\in H$. Hence $\supp(\theta(\xi_H(s,\phi)))\subset
\pi(\{s\}\times \supp(\phi))$ which is a finite simplicial set if
$\phi\in\Z^{(X)}$. Therefore $\theta$ maps $\ind_H^G(\Z^{(X)})$
inside $\Z^{(\ind_H^G(X))}$. It remains to show that
$\theta^{-1}(\Z^{(\ind_H^G(X))})\subset \ind_H^G(\Z^{(X)})$. Let
$\{g_i\}\subset G$ be a full set of representatives of $G/H$. Every
element of $G\times_HX$ can be written uniquely as $\pi(g_i,x)$ for
some $i$ and some $x\in X$. Hence as a simplicial set, $\ind_H^G(X)$
is the disjoint union of the $Y_i=\pi(\{g_i\}\times X)$. In
particular if $\phi\in\Z^{(\ind_H^G(X))}$, then its support meets
finitely many of the $Y_i$, and $\supp(\phi)\cap Y_i$ is a finite
simplicial set. Thus there is a finite number of $i$ such that
$\psi=\theta^{-1}(\phi)$ is nonzero on $g_iH$, and its restriction
to each of these subsets takes values in $\Z^{(X)}$. By
\eqref{uniquexp}, this implies that $\psi\in \ind_H^G(\Z^{(X)})$, as
we had to prove.
\end{proof}

If $C,A\in H-\ring$ with $C$ commutative and we have a compatible
$(H,C)$-algebra structure on $A$, then $\ind_H^G(A)$ carries a
compatible $(G,\ind_H^G(C))$-algebra structure, given by
\[
\xi_H(s,c)\cdot\xi_H(t,a)=\left\{\begin{matrix}\xi_H(s,c\cdot a)& s=t\\ 0& sH\ne tH\end{matrix}\right.
\]
If moreover $C\cdot A=A$, then $\ind_H^G(C)\cdot \ind_H^G(A)=\ind_H^G(A)$. We record a particular case of this in the following
\begin{lem}\label{lem:indproper} If $A\in H-\ring$ is proper over an $H$-simplicial set $X$, then the $G$-ring $\ind_H^G(A)$ is proper over $\ind_H^G(X)$.
\end{lem}
\begin{proof} It follows from Lemma \ref{lem:indx} and the discussion above.
\end{proof}
\subsection{Compression}

Let $A\in G-\ring$, and $H\subset G$ a subgroup. Assume that $A$ is
proper over $G/H$. Let $\chi_H\in \Z^{(G/H)}$ be the characteristic
function of $H$. The {\it compression} of $A$ over $H$ is the
subring
\[
\comp_H^G(A)=\chi_H\cdot A
\]
Note the action of $G$ on $A$ restricts to an action of $H$ on $\comp_H^G(A)$, which makes it into an object of $H-\ring$.
\begin{prop}\label{prop:indcomp}(Compare \cite[Lemma 12.3, and paragraph after 12.4]{ght})
\item[i)] If $B\in H-\ring$, then $\ind_H^G(B)$ is proper over $G/H$, and
\[
B\to \comp_H^G\ind_H^GB,\quad b\mapsto\xi_H(1,b)
\]
is an $H$-equivariant isomorphism.
\item[ii)] If $A\in G-\ring$ is proper over $G/H$, then
\[
\ind_H^G\comp_H^G(A)\to A ,\quad \xi_H(s,\chi_Ha)\mapsto \chi_{sH}s(a)
\]
is a $G$-equivariant isomorphism.
\end{prop}
\begin{proof} Any $B\in H-\ring$ is proper over the $1$-point space $*$. Hence $\ind_H^G(B)$ is proper over $\ind_H^G(*)=G/H$, by Lemma \ref{lem:indproper}. The proof that the maps of i) and ii) are isomorphisms is straightforward; to show equivariance, one uses \eqref{actixi} and
\eqref{hxi}.
\end{proof}
\subsection{A discrete variant of Green's imprimitivity theorem}
Let $G$ be a group, $H\subset G$ a subgroup and $A$ an $H$-ring.. Observe that, by
definition, the $G$-ring $\ind_H^G(A)$ is a $G$-subring of the ring $\map(G,A^{\Delta^\bu})=\map(G,A)=A^{G}$
(note that this is not the same as the subring of $G$-invariants of $A$).
Since $A^{(G)}\triqui A^{G}$ is a $G$-ideal, we may regard $A^{(G)}$
as a left $\ind_H^G(A)$-module via left multiplication in $A^G$, and
moreover, this action is compatible with that of $G$, in the sense
that the two together define a left $\ind_H^G(A)\rtimes G$-module
structure on $A^{(G)}$. We may also regard $A^{(G)}$ as a right
module over $A\rtimes H$, via
\[
[\phi\cdot(a\rtimes h)](g)=h^{-1}(\phi(gh^{-1})a)
\]
One checks that these two actions satisfy
\[
(f\rtimes g)\cdot[\phi\cdot(a\rtimes h)]=[(f\rtimes g)\cdot\phi]\cdot(a\rtimes h)
\]
Hence they make $A^{(G)}$ into an $(\ind_H^G(A)\rtimes G, A\rtimes
H)$-bimodule. In particular left multiplication by elements of
$\ind_H^G(A)\rtimes G$ induces a ring homomorphism
\begin{equation}\label{map:repcross}
\ind_H^G(A)\rtimes G\to \End_{A\rtimes H}(A^{(G)})
\end{equation}
Observe that the decomposition $G=\coprod_{x\in G/H}x$ induces
\begin{equation}\label{decompoag}
A^{(G)}=\bigoplus_{x\in G/H}A^{(x)}
\end{equation}
 and that $A^{(x)}\cdot (A\rtimes H)\subset A^{(x)}$. Hence \eqref{decompoag} is a direct sum of right $A\rtimes H$-modules. Thus we may think of an element $T\in\End_{A\rtimes H}(A^{(G)})$ as a matrix $T=[T_{x,y}]_{x,y\in G/H}$, where $T_{x,y}:A^{(y)}\to A^{(x)}$  is a homomorphism of right $A\rtimes H$-modules, and is such that for each $v\in A^{(y)}$, $T_{x,y}(v)=0$ for all but a finite number of $x$. Moreover
 \[
 A\rtimes H\to A^{(gH)},\quad a\rtimes h\mapsto \chi_g\cdot (a\rtimes h)=\chi_{gh}h^{-1}(a)
 \]
 is an isomorphism of right $A\rtimes H$-modules. Fix a full set of representatives
 $\cR$ of $G/H$, with $1\in\cR$, write $M_\cR\in\Z-\ring$ for the ring of $\cR\times \cR$-matrices with finitely many nonzero coefficients in $\Z$, and put $M_\cR(A\rtimes H)=M_\cR\otimes (A\rtimes H)$.
 We have a ring homomorphism
 \begin{gather*}
 M_\cR(A\rtimes H)\to \End_{A\rtimes H}(A^{(G)})\\
  M\mapsto (\sum_{y\in \cR}\chi_y\cdot \alpha_y\mapsto \sum_{x\in \cR}\chi_x\sum_{y\in\cR} m_{x,y}\alpha_y)
 \end{gather*}
 Furthermore, we have a map $G\to \cR$, which sends each $s\in G$ to the representative $\hat{s}\in \cR$ of $sH$.
 Using this map we obtain an isomorphism $M_{G/H}\cong M_\cR$ which sends the matrix unit $E_{sH,tH}$ to $E_{\hat{s},\hat{t}}$. By composition, we obtain a ring homomorphism
 \begin{equation}\label{map:MGtoEnd}
 M_{G/H}(A\rtimes H)\to \End_{A\rtimes H}(A^{(G)})\cong \End_{A\rtimes H}((A\rtimes H)^{(G/H)})
 \end{equation}

\begin{rem}\label{rem:unital}
If $A$ happens to be unital, then both \eqref{map:repcross} and \eqref{map:MGtoEnd} are injective.
\end{rem}

 \begin{thm}\label{thm:git}
 Let $G$ be a group, $H\subset G$ a subgroup, and $A\in H-\ring$. Then there is an isomorphism
 $\ind_H^G(A)\rtimes G\cong M_{G/H}(A\rtimes H)$ such that the following diagrams commute
 \begin{gather*}
 \xymatrix{\ind_H^G(A)\rtimes G\ar[rr]^{\eqref{map:repcross}}\ar[dr]_{\cong}&&\End_{A\rtimes H}(A^{(G)})\\
                              &M_{G/H}(A\rtimes H)\ar[ur]_{\eqref{map:MGtoEnd}}&}\\
 \xymatrix{\ind_H^G(A)\rtimes G\ar[rr]^{\cong}& &M_{G/H}(A\rtimes H)\\
 &A\rtimes H\ar[ul]^{\xi_H(1,-)\rtimes id}\ar[ur]_{e_{H,H}\otimes -}&}
 \end{gather*}
 \end{thm}
 \begin{proof} We use the notation introduced in the paragraph preceding the theorem. If $s\in G$, put $\phi(s)=\hat{s}^{-1}s\in H$. Note that $\phi(sh)=\phi(s)h$ $(s\in G, h\in H)$.
 One checks that the following map is a well-defined, bijective ring homomorphism with the required properties
 \begin{gather*}
 \alpha:\ind_H^G(A)\rtimes G\to M_{G/H}(A\rtimes H),\\
 \alpha(\xi_H(s,a)\rtimes g)=e_{sH,g^{-1}sH}\otimes \phi(s)(a)\rtimes \phi(s)\phi(g^{-1}s)^{-1}
 \end{gather*}
  \end{proof}
 \begin{rem}\label{rem:noncan}
 The isomorphism of the theorem above is natural in $A$, but not in the pair $(G,H)$, as it depends
 on a choice of a full set of representatives $\cR$ of $G/H$, or what is the same, of a choice of pointed
 section $G/H\to G$ of the canonical projection.
 \end{rem}

\subsection{Restriction}\label{subsec:restri}
Let $B$ be a $G$-ring, $H\subset G$ a subgroup. Write $\res^H_GB$
for the $H$-ring obtained by restriction to $H$ of the action of $G$
on $B$.

\begin{lem}\label{lem:indtriv}
If $B$ is a $G$-ring, then $\ind_H^G\res_G^H B\to \Z^{(G/H)}\otimes
B$, $\xi_H(s,b)\mapsto \chi_{sH}\otimes s(b)$ is a $G$-ring
isomorphism.
\end{lem}
\begin{proof} Straightforward.
\end{proof}

Now suppose $K\subset G$ is another subgroup. Let $x\in H\backslash G/K$. Put

\begin{equation}\label{resindx}
\res_G^H\ind_K^G(A)[x]=\{f\in \ind_K^G(A):\supp(f)\subset x\}\in H-\ring
\end{equation}
We have
\begin{equation}\label{decompx}
\res_G^H\ind_K^G(A)=\bigoplus_{x\in H\backslash G/K}\res_G^H\ind_K^G(A)[x]
\end{equation}
Write $x=H\theta K$ for some $\theta\in G$. Consider the subgroup
\[
H\supset H_\theta=H\cap \theta K\theta^{-1}
\]
We shall see presently that the $H$-ring \eqref{resindx} is proper
over  $H/H_\theta$. Consider the subgroup
\[
K\supset K_{\theta^{-1}}=\theta^{-1}H\theta\cap K
\]
Conjugation by $\theta^{-1}$ defines an isomorphism
\[
c_{\theta^{-1}}:H_\theta\to K_{\theta^{-1}},\quad c_{\theta^{-1}}(h)=\theta^{-1}h\theta
\]
Hence we may view $\res_K^{K_{\theta^{-1}}}A$ as an $H_\theta$-ring
via $c_{\theta^{-1}}$; we write
$c_{\theta^{-1}}^*(\res_K^{K_{\theta^{-1}}}A)$ for the resulting
$H_\theta$-ring.
\bigskip
\begin{lem}\label{lem:indxtheta}
The map
\begin{gather*}
\alpha:\res^H_G\ind_K^G(A)[H\theta K]\to \ind_{H_\theta}^H(c_{\theta^{-1}}^*(\res_K^{K_{\theta^{-1}}}(A))\\
\alpha(f)(h)=f(h\theta)
\end{gather*}
is an isomorphism of $H$-rings.
\end{lem}
\begin{proof} One checks that if $m\in H_\theta$, then $\alpha(f)(hm)=m^{-1}\alpha(f)(h)$. It is clear that
$\alpha$ is $H$-equivariant. A calculation shows that $\alpha(\xi_K(h\theta,a))=\xi_{H_\theta}(h,a)$. It follows that $\alpha$ is an isomorphism.
\end{proof}
\section{Induction and equivariant homology}\label{sec:ind}

\numberwithin{equation}{section}

\begin{lem}\label{lem:indexci}
Let $G$ be a group, $K\subset G$ a subgroup, $A$ a $K$-ring, and
$E:\Z-\cat\to\spt$ a functor satisfying the standing assumptions.
Then $A$ is $E$-excisive if and only if $\ind_K^G(A)$ is
$E$-excisive.
\end{lem}
\begin{proof}
The map \eqref{map:indinonequi} gives a nonequivariant isomorphism
\[
\ind_K^G(A)\cong \Z^{(G/K)}\otimes A=\bigoplus_{x\in G/K} A
\]
The equivalence of the lemma follows from Standing Assumption v).
\end{proof}

Let $G$, $K$ and $A$ be as in Lemma \ref{lem:indexci}, and let $X$
be a $G$-simplicial set. If $A$ is unital, then for each subgroup
$S\subset K$ we have a functor
\begin{gather*}
A\rtimes \cG^K(K/S)\to \ind_K^G(A)\rtimes \cG^G(G/S)\\
kS\mapsto kS,\\
a\rtimes k\mapsto \xi_K(1,a)\rtimes k
\end{gather*}
If $A$ is any $E$-excisive ring, the map above is defined for the unitalization $\tilde{A}$; applying $E$,
 taking fibers relative to the augmentation $\tilde{A}\to \Z$, and using the standing assumptions, we get a map
 $E(A\rtimes \cG^K(K/S))\to E(\ind_K^G(A)\rtimes \cG^G(G/S))$.
The maps
\[
X^S_+\land E(A\rtimes \cG^K(K/S))\to X^S_+\land E(\ind_K^G(A)\rtimes\cG^G(G/S))\to H^G(X,E(\ind_K^G A))
\]
assemble to
\begin{equation}\label{map:ind}
\ind: H^K(X,E(A))\to H^G(X,E(\ind_K^G(A)))
\end{equation}

\begin{prop}\label{prop:equindiso}(Compare \cite[Proposition 12.9]{ght})
Let $A$ be an $E$-excisive $G$-ring. Then the map \eqref{map:ind} is
an equivalence.
\end{prop}
\begin{proof}
As a functor of $G$-simplicial sets, equivariant homology satisfies excision and commutes with filtering colimits (see \cite{dl}). Because of this, and because $X$ is obtained by gluing together cells of the form
$\ind_H^G(\Delta^n)$, $H\in\fall$,  it suffices to prove the
proposition for $X=\ind_H^G(T)$ where $H$ acts trivially on $T$. Let
$\cR$ be a full set of representatives of $K\backslash G/H$. We have
\begin{align*}
\ind_H^G(T)=&T\times G/H\\ =&\coprod_{\theta\in \cR}T\times K\theta H\\
\cong&\coprod_{\theta\in \cR}T\times K/K_\theta
\end{align*}
Here as in Subsection \ref{subsec:restri}, $K_\theta=c_\theta(H)\cap K$.
Thus
\[
H^K(\ind_H^G(T),E(A))=T_+\land\bigvee_{\theta\in\cR}E(A\rtimes \cG^K(K/K_\theta))
\]
On the other hand,
\begin{align*}
H^G(\ind_H^G(T), E(\ind_K^G(A))=&T_+\land E(\ind_K^G(A)\rtimes\cG^G(G/H))
\end{align*}
We have to show that
\[
\bigvee_{\theta\in\cR}E(A\rtimes \cG^K(K/K_\theta))\to E(\ind_K^G(A)\rtimes\cG^G(G/H))
\]
is an equivalence. By standing assumptions iv) and v) we may replace
the map above by that induced by the corresponding ring homomorphism
\begin{equation}\label{map:sumind}
\bigoplus_{\theta\in\cR}\cA(A\rtimes \cG^K(K/K_\theta))\to
\cA(\ind_K^G(A)\rtimes\cG^G(G/H))
\end{equation}
Here $\cA(A\rtimes \cG^K(K/K_\theta))\to
\cA(\ind_K^G(A)\rtimes\cG^G(G/H))$ is induced by $\xi_K(1,-):A\to \ind_K^G(A)$ and by the inclusions $K\subset G$ and $K/K_\theta\to G/H$, $kK_\theta\mapsto k\theta H$.
One checks that the following diagram commutes

\[\xymatrix{&\cA(\ind_K^G(A)\rtimes\cG^G(G/H))\ar[dr]^{\ref{lem:across=cross}}_\sim
&\\
\cA(A\rtimes \cG^K(K/K_\theta))\ar[ur]^{\xi_K(1,-)\rtimes \inc}& & M_{G/H}(\ind_K^G(A)\rtimes H)\\
A\rtimes K_\theta\ar[u]\ar[d]_{\wr}^(.4){1\rtimes
c_{\theta^{-1}}}\ar[rr]^{\xi_K(\theta^{-1},-)\rtimes
c_{\theta^{-1}}}&&\ind_K^G(A)[H\theta^{-1} K]\rtimes
H\ar[u]_{e_{\theta H,\theta H}}&\\
c_{\theta}^*(A)\rtimes
H_{\theta^{-1}}\ar[r]^(.4){e_{H_{\theta^{-1}},H_{\theta^{-1}}}} &
M_{H/H_{\theta^{-1}}}\left(c_{\theta}^*(A)\rtimes
H_{\theta^{-1}}\right)\ar[r]_{\ref{thm:git}}^\sim &
\ind_{H_{\theta^{-1}}}^H (c_{\theta}^*(A))\rtimes
H\ar[u]^\sim_{\ref{lem:indxtheta}}}
\]

Because the lower rectangle commutes,  $E(A\rtimes K_\theta\to
\ind_K^G(A)[H\theta^{-1} K]\rtimes H)$ is an equivalence, by matrix
stability. Again by matrix stability and by Lemma
\ref{lem:across=cross}, applying $E$ to the top left vertical arrow
is an equivalence. Hence to prove that $E$ applied to
\eqref{map:sumind} is an equivalence, it suffices to show that $E$
applied to
\begin{equation}\label{map:sumind2}
\xymatrix{\ind_K^G(A)\rtimes
H=\bigoplus_{\theta\in\cR}\ind_K^G(A)[H\theta K]\rtimes H
\ar[rr]^(.6){\sum_\theta e_{\theta H,\theta H}}
&&M_{G/H}(\ind_K^G(A)\rtimes H)}
\end{equation}
is one. But another application of matrix stability (using
\cite[Prop. 2.2.6]{friendly}) shows that $E$ applied to \eqref{map:sumind2} gives
the same map in $\ho\spt$ as $E$ applied to the inclusion
\[e_{H,H}:\ind_K^G(A)\rtimes H\to M_{G/H}(\ind_K^G(A)\rtimes H).\]
This concludes the proof.
\end{proof}

\begin{thm}\label{thm:proper0}
Let $E:\Z-\cat\to\spt$ be a functor satisfying the standing
assumptions \ref{stan}. Also let $G$ be a group, $\cF$ a family of
subgroups of $G$ and $B$ an $E$-excisive ring, proper over a
$0$-dimensional $(G,\cF)$-complex $X$. Then $H^G(-,E(B))$ maps
$(G,\cF)$-equivalences to equivalences. In particular, the assembly
map
\[
H^G(\cE(G,\cF),E(B))\to E(B\rtimes G)
\]
is an equivalence.
\end{thm}
\begin{proof}
We have $X=\coprod_i G/K_i$ for some $K_i\in\cF$, and
$\Z^{(X)}=\bigoplus_i\Z^{(G/K_i)}$. The ring $B_i=\Z^{(G/K_i)}\cdot
B$ is proper over $G/K_i$, and is excisive by Standing assumption v).
Again by Standing assumption v), it suffices
to prove the assertion of the theorem individually for each $B_i$;
in other words, we may assume $X=G/K$ for some $K\in\cF$. Hence for
$A=\comp_G^KB$ we have $B=\ind_K^GA$, by Proposition
\ref{prop:indcomp}. Moreover, by Lemma \ref{lem:indexci}, $A$ is
$E$-excisive. Let $Y\to Z$ be a $(G,\cF)$-equivalence. We have a
commutative diagram
\[
\xymatrix{H^G(Y,E(B))\ar[rr]&& H^G(Z,E(B))\\
           H^K(Y, E(A))\ar[u]^{\ind}\ar[rr] && H^K(Z, E(A))\ar[u]^\ind}
\]
The bottom horizontal arrow is an equivalence because $K\in\cF$. The two
vertical arrows are equivalences by Proposition \ref{prop:equindiso}.
It follows that the top horizontal arrow is an equivalence too.
\end{proof}

\section{Assembly as a connecting map}\label{sec:dirac}

Throughout this section, we consider a fixed functor
$E:\Z-\cat\to\spt$, and --except when otherwise stated-- we assume
that, in addition to the standing assumptions, it satisfies the
following:
\begin{sect}\label{secstan}
\item[vi)]$E_*$ commutes with filtering colimits.
\item[vii)] If $A$ is $E$-excisive and $L$ has local units and is flat as a
$\Z$-module, then $L\otimes A$ is $E$-excisive.
\end{sect}
\numberwithin{equation}{subsection}
\subsection{Preliminaries}\label{subsec:prelidirac}

\goodbreak

\smallskip

\paragraph{\em Mapping cones}

Let $f:A\to B$ be a ring homomorphism; the {\em mapping cone} of $f$ is defined as the pullback
\[
\xymatrix{\Gamma_f\ar[d]\ar[r]& \Gamma B\ar[d]\\ \Sigma A\ar[r]_{\Sigma f}&\Sigma B }
\]

\begin{lem}\label{lem:mapcone}
Let $E:\Z-\cat\to\spt$ be a functor satisfying both the standing and
the sectional assumptions, and $f:A\to B$ a homomorphism of
$E$-excisive rings. Then
\item[i)] $E(\Gamma B)$ is weakly contractible.
\item[ii)] $E(\Sigma B)\weq \Sigma E(B)$.
\item[iii)] The following is a distinguished triangle in $\ho\spt$
\[
E(B)\to E(\Gamma_f)\to \Sigma E(A)\overset{\Sigma E(f)}\longrightarrow\Sigma E(B)
\]
\end{lem}
\begin{proof}
By Lemma \ref{lem:gamatenso}, $\Gamma B=\Gamma\Z\otimes B$, whence
it is $E$-excisive, by sectional assumption \ref{secstan} vii). Part i) follows from matrix stability and the
fact that $\Gamma\Z$ is a ring with infinite sums (see e.g.
\cite[Prop. 2.3.1]{friendly}). Parts ii) and iii) follow from i) and
excision.
\end{proof}

\paragraph{\em Matrix rings and group actions}
\begin{lem}\label{lem:mxg}
Let $G$ be a group, $A$ a $G$-ring and $X$ a $G$-set. Write
$M_{\ul{X}}$ for the ring $M_{X}$ equipped with the $G$-action
\[
g(e_{x,y})=e_{gx,gy}
\]
The map
\[
(M_{\ul{X}} A)\rtimes G\to M_{X}(A\rtimes G), (e_{x,y}\otimes
a)\rtimes g\mapsto e_{x,g^{-1}y}\otimes(a\rtimes g)
\]
is a $G$-equivariant isomorphism of rings.
\end{lem}

\subsection{Dirac extensions}

Let $G$ be a group, $\cF$ a family of subgroups, $E:\Z-\cat\to\spt$ a functor
satisfying the standing assumptions, and $A$ an $E$-excisive ring. A {\em Dirac extension} for $(G,\cF,A,E)$ consists of an extension of $E$-excisive $G$-rings
\begin{equation}\label{seq:dirac}
0\to B\to Q\to P\to 0
\end{equation}
together with a zig-zag
\[
\xymatrix{A=Z_0\ar[r]^{f_0}& Z_1& Z_2\ar[l]_{f_2}\ar[r]^{f_3}&\dots&
Z_n=B}
\]
such that
\begin{itemize}
\item[a)] $E(f_i\rtimes H)$ is an equivalence for every subgroup $H\subset G$.
\item[b)] $E_*(Q\rtimes H)=0$ for every $H\in\cF$.
\item[c)] $H^G(-,E(P))$ sends $(G,\cF)$-equivalences to equivalences.
\end{itemize}

\begin{rem}\label{rem:adir}
Condition a) together with standing assumptions iii) and iv) and Lemma \ref{lem:across=cross} imply that the zig-zag $f=\{f_i\}$ induces
an equivalence
$H^G(X,E(A))\weq H^G(X,E(B))$
for every $G$-space $X$. Similarly, it follows from condition b) that $H_*^G(Y,E(Q))=0$ for every $(G,\cF)$-complex
$Y$.
\end{rem}
\begin{prop}\label{prop:dirac}
Let $E:\Z-\cat\to\spt$ be a functor satisfying the standing
assumptions, $G$ a group, $\cF$ a family of subgroups of $G$, and
$A$ a $G$-ring. Let \eqref{seq:dirac} be a Dirac extension for
$(G,\cF,A,E)$. Then there are an exact sequence
\[
E_{*+1}(A\rtimes G)\to E_{*+1}(Q\rtimes G)\to E_{*+1}(P\rtimes G)\overset\partial\to E_*(A\rtimes G)
\]
an isomorphism $H^G_*(\cE(G,\cF),E(A))\cong E_{*+1}(P\rtimes G)$,
and a commutative diagram
\[
\xymatrix{H^G_*(\cE(G,\cF),E(A))\ar[dr]_\cong\ar[rr]^\assem && E_*(A\rtimes G)\\
&E_{*+1}(P\rtimes G)\ar[ur]^\partial&}
\]
\end{prop}
\begin{proof}
By Proposition \ref{prop:exciequi1} and Remark \ref{rem:adir} we have a distinguished triangle
\begin{equation}\label{seq:excidirac}
\xymatrix{H^G(X,E(A))\ar[r]& H^G(X,E(Q))\ar[r]& H^G(X,E(P))\ar[r]^{\partial^X}&\Sigma H^G(X,E(A))}
\end{equation}
for every $G$-simplicial set $X$. The proposition follows by comparison of the long exact sequence of homotopy associated to the triangles for $X=\cE(G,\cF)$, and $X=*$, using that, again by Remark \ref{rem:adir}, we have $H_*^G(\cE(G,\cF),E(Q))=0$.
\end{proof}

\begin{rem}\label{rem:guarda}
If $X\in\fS^G$ and $cX\fibeq X$ is a $(G,\cF)$-cofibrant replacement, then the same argument as that of the proof
of Proposition \ref{prop:dirac} shows that the map $H^G(cX,E(A))\to H^G(X,E(A))$ is an equivalence if and only if
the boundary map $\partial^X$ in the sequence \eqref{seq:excidirac} is an equivalence.
\end{rem}

\subsection{A canonical Dirac extension}

Let $G$ be a group and $\cF$ a family of subgroups. Consider the discrete $G$-simplicial sets
\[
X=X_\cF=\coprod_{H\in\cF}G/H, \quad Y=G/G\coprod X
\]
The group $G$ acts on $Y$ and thus on the ring $M_Y$ of $Y\times
Y$-matrices with finitely many nonzero integral coefficients. The
point $y_0$ corresponding to the unique orbit of $G/G$ is fixed by
$G$, whence the map $\iota:\Z\to M_{\ul{Y}}$, $\lambda\to \lambda
E_{y_0,y_0}$ is $G$-equivariant. In particular we have a directed system
of $G$-rings $\{id\otimes\iota:(M_\infty M_{\ul{Y}})^{\otimes n}\to
(M_\infty M_{\ul{Y}})^{\otimes n+1}\}_n$. Put
\[
\fF^0=\colim_n (M_\infty M_{\ul{Y}})^{\otimes n}
\]
Since $X$ is discrete, the ring of finitely supported functions breaks up into a sum
\[
\Z^{(X)}=\bigoplus_{x\in X}k\chi_x
\]
Multiplication by an element of $M_{\ul{Y}}$ gives an $\Z$-linear
endomorphism of $\Z^{(Y)}$. This defines an equivariant monomorphism
\[
M_{\ul{Y}}\to \End_\Z(\Z^{(Y)})
\]
whose image consists of those linear transformations $T$ such that
the matrix of $T$ with respect to the basis $\{\chi_{y}:y\in Y\}$ has
finitely many nonzero entries. Note that multiplication by $\chi_x$
in $\Z^{(X)}\subset \Z^{(Y)}$ is in this image. Thus we have an
equivariant injective ring homomorphism
\[
\rho:\Z^{(X)}\to M_{\ul{Y}}
\]
For each $n\ge 1$, consider the $G$-ring
\[
\fF^n=\left(\bigotimes_{i=1}^n\Gamma_\rho\right)\otimes\fF^0
\]
The inclusion $M_\infty M_{\ul{Y}}\to \Gamma_\rho$ induces an
inclusion $\fF^n\subset\fF^{n+1}$ for each $n\ge 0$. Put
\[
\fF^\infty=\bigcup_{n\ge 0}\fF^n
\]
If $A\in\ring$, we also write  $\fF^nA=\fF^n\otimes A$ $(n\ge 0)$. We have

\begin{lem}\label{lem:preassbound}
\item[i)]  $\fF^n\subset \fF^\infty $ is an ideal ($n<\infty$).
\item[ii)] For each $n\ge 0$, $\fF^n$ and $\fF^{n+1}/\fF^n\cong\Sigma \Z^{(X)}\otimes\fF^n$ have local units, are $(G,\cF)$-proper rings and are
flat as abelian groups.
\item[iii)] If $H\in\cF$, $\chi_H\in\Z^{(G/H)}\subset\Z^{(X)}$ is the characteristic function, and $A$ is a $G$-ring, we have a commutative diagram
\[
\xymatrix{(\Z^{(G/H)}\otimes \fF^nA)\rtimes H\subset
(\Z^{(X)}\otimes \fF^nA)\rtimes
H\ar[rr]^(.65){(\rho\otimes 1)\rtimes id}&&
(M_{\ul{Y}}\fF^nA)\rtimes H\ar[d]_\wr^{\eqref{lem:mxg}}\\
\fF^nA\rtimes H\ar[u]^{\chi_H\otimes 1}\ar[rr]_{e_{H,H}\otimes
-}&&M_{Y}(\fF^nA\rtimes H)}
\]
\end{lem}
\begin{proof} Part i) is clear. Because $M_{\ul{Y}}$ is proper over $Y$, $\fF^n$ is proper over $Y$ for all $n$, by \ref{ex:proper}. Similarly,
\begin{equation}\label{eq:slice}
\fF^{n+1}/\fF^n=\Sigma\Z^{(X)}\otimes \fF^n
\end{equation}
is proper. That $\fF^n$ is flat is clear for $n=0$; the general case
follows by induction, using \eqref{eq:slice}. The ring $\fF^0$ has
local units because $M_{\ul{Y}}$ and $M_\infty$ do. To prove that
$\fF^n$ has local units for $n\ge 1$, it suffices to show that
$\Gamma_\rho$ does. We may and do identify $\Gamma_\rho$ with the
inverse image of $\Sigma(\rho(\Z^{(X)}))$ under the projection
$\pi:\Gamma M_{\ul{Y}}\to \Sigma M_{\ul{Y}}$; thus
\[
\Gamma_\rho=\Gamma\rho(\Z^{(X)})+M_\infty M_{\ul{Y}}\subset \Gamma
M_{\ul{Y}}
\]
One checks that if $\phi_1,\dots,\phi_r\in \Gamma_\rho$, then there
are finite subsets $F_1\subset X$ and $F_2\subset\N$ such that for
$y_0=G/G\in Y$, the element
\[
e=1\otimes \sum_{x\in F_1} e_{x,x}+\sum_{p\in F_2}e_{p,p}\otimes
e_{y_0,y_0}\in \Gamma_\rho
\]
satisfies $e^2=e$ and $e\phi_i=\phi_ie=\phi_i$ for all $i=1,\dots,r$. This proves part ii);
part iii) is straightforward.
\end{proof}

\begin{thm}\label{thm:assbound}(Compare \cite[Theorem 5.18]{cmr})
Let $E:\Z-\cat\to \spt$ be a functor satisfying both the standing
and the sectional assumptions. Let $G$ a group, $\cF$ a family of
subgroups, and $A$ an $E$-excisive $G$-ring. Then
\[
\fF^0A\to \fF^\infty A\to \fF^\infty A/\fF^0A
\]
is a Dirac extension for $(G,\cF,E,A)$.
\end{thm}
\begin{proof}
The three rings in the extension of the theorem are $E$-excisive, by
Lemma \ref{lem:preassbound} ii) and sectional assumption \ref{secstan} vii). The map
$E(A\rtimes H)\to E(\fF^0A\rtimes H)$ is an equivalence for all subgroups $H\subset G$ by Lemma \ref{lem:mxg},
standing assumptions ii) and iii)
and sectional assumption vi). Next
we prove that if $cX\to X$ is a cofibrant replacement, then $H^G(cX,E(\fF^\infty A/\fF^0
A))\to H^G(X,E(\fF^\infty A/\fF^0
A))$ is an equivalence. By
excision and sectional assumption vi), it suffices to show that
\begin{equation}\label{seq:hdirac}
H^G(cX,E(\fF^n A/\fF^0 A))\to H^G(X,E(\fF^n A/\fF^0 A))\quad (n\ge 1)
\end{equation}
 is an equivalence. Consider the extension
\[
0\to\fF^nA/\fF^0A\to \fF^{n+1}A/\fF^0A\to \fF^{n+1}A/\fF^nA\to 0
\]
By Proposition \ref{prop:exciequi1}, $cX\to X$ gives a map of homotopy fibration sequences
\[
\xymatrix{H^G(cX,E(\fF^nA/\fF^0A))\ar[r]\ar[d]&H^G(X,E(\fF^nA/\fF^0A))\ar[d]\\
H^G(cX,E(\fF^{n+1}A/\fF^0A))\ar[r]\ar[d]&H^G(X,E(\fF^{n+1}A/\fF^0A))\ar[d]\\
H^G(cX,E(\fF^{n+1}A/\fF^nA))\ar[r]&H^G(X,E(\fF^{n+1}A/\fF^nA))}
\]
By Lemma \ref{lem:preassbound} and Theorem \ref{thm:proper0}, the
bottom horizontal map is an equivalence. Hence \eqref{seq:hdirac} is
an equivalence for each $n$, by induction. It remains to show that
$E_*(\fF^\infty A\rtimes H)=0$ for each $H\in\cF$. Because $E_*$
preserves filtering colimits by assumption, we may further restrict
ourselves to proving that the map $j_n:E_*(\fF^n A\rtimes H)\to
E_*(\fF^{n+1}A\rtimes H)$ induced by inclusion is zero for all $n$.
By Lemma \ref{lem:mapcone} we have a long exact sequence $(q\in\Z)$
\[
\xymatrix{E_q(\fF^nA\rtimes H)\ar[r]^{j_n}& E_q(\fF^{n+1}A\rtimes H)
\ar[r]& E_{q-1}(\Z^{(X)}\otimes \fF^nA\rtimes H)\ar[d]^{\partial}\\
&& E_{q-1}(\fF^nA\rtimes H)}
\]
where $\partial=E_{q-1}(\rho\otimes 1\rtimes 1)$. By Lemma
\ref{lem:preassbound}, part iii), $\partial$ is a split surjection.
It follows that $j_n=0$; this concludes the proof.
\end{proof}

\begin{ex}
The hypothesis of Theorem \ref{thm:assbound} are satisfied, for
example, by the functorial spectra $K$, $K^{\ninf}$ and $KH$.
\end{ex}

\section{Isomorphism conjectures with proper coefficients}\label{sec:main}

\subsection{The excisive case}
\begin{thm}\label{thm:propern}
Let $E:\Z-\cat\to \spt$ be a functor. Assume that $E$ satisfies the standing assumptions \ref{stan}, that it is excisive and that $E_*$ commutes with filtering colimits.
Let $A$ be a $(G,\cF)$-proper $G$-ring. Then the functor $H^G(-,E(A))$ sends
$(G,\cF)$-equivalences to equivalences. In particular the assembly map
\[
H^G(\cE(G,\cF),E(A))\to E(A\rtimes G)
\]
is an equivalence.
\end{thm}
\begin{proof}
By definition of properness, there is a locally finite $(G,\cF)$-complex $X$ such that $A$ is proper over $X$. We consider first the case when $X$ is finite dimensional.
If $\dim X=0$, the theorem follows from Theorem \ref{thm:proper0}. Let
$n>0$ and assume the theorem true in dimensions $<n$. If $\dim X=n$,
and $Y\subset X$ is the $n-1$-skeleton, we have a pushout diagram
\[
\xymatrix{\coprod_i\ind_{H_i}^G(\Delta^n)\ar[r]& X\\
\coprod_i\ind_{H_i}^G(\partial\Delta^n)\ar[r]\ar[u]& Y\ar[u]}
\]
Here $H_i\in\cF$ and the horizontal arrows are proper, since $X$ is
assumed locally finite. Hence we obtain a pullback diagram
\begin{equation}\label{diag:main}
\xymatrix{\bigoplus_i\Z^{(\Delta^n)}\otimes
\Z^{(G/H_i)}\ar[d]&\Z^{(X)}\ar[l]\ar[d]\\
\bigoplus_i\Z^{(\partial\Delta^n)}\otimes \Z^{(G/H_i)}&
\Z^{(Y)}\ar[l]}
\end{equation}
Let $I=\ker(\Z^{(X)}\to\Z^{(Y)})$ be the kernel of the restriction map; because the diagram above is cartesian, $I\cong\bigoplus_i\ker(\Z^{(\Delta^n)}\otimes\Z^{(G/H_i)}\to \bigoplus_i\Z^{(\partial\Delta^n)}\otimes\Z^{(G/H_i)})$. The quotient
$A/I\cdot A$ is proper over $Y$, and $I\cdot A$ is proper over $\coprod_i\ind^G_{H_i}(\Delta^n)$, whence also over the zero-dimensional  $\coprod_i G/H_i$, by
Lemma \ref{lem:festar}. Thus the theorem is true for both $A/I\cdot A$ and $I\cdot A$; because $E$ is excisive
by hypothesis, this implies that the theorem is also true for $A$. This proves the theorem for $X$ finite dimensional.
The general case follows from this using Lemma \ref{lem:properfil} and the hypothesis that $E_*$ commutes with filtering colimits.
\end{proof}

\begin{ex}\label{ex:khkninfkinf}
Both $KH$ and $K^\ninf$ satisfy the hypothesis of Theorem \ref{thm:propern}.
\end{ex}

\begin{rem}\label{rem:kinf}
The proof of Theorem \ref{thm:propern} makes clear that if the hypothesis that $E_*$ commutes with filtering colimits is dropped, then the theorem remains true for $A$ proper over a finite dimensional $(G,\cF)$-complex. On the other hand, the hypothesis that $E$ be excisive is key, since the standing assumptions alone do not guarantee that the excision arguments of the proof go through, not even
for $A=\Z$. The argument uses that the common kernel of the vertical maps of \eqref{diag:main} be $E$-excisive; by standing assumption \ref{stan} v) this is equivalent to saying that $I_n=\ker(\Z^{\Delta^n}\to \Z^{\partial\Delta^n})$
is $E$-excisive. However $I_n$ is not $K$-excisive, because $\tor_1^{\tilde{I}_n}(\Z,I_n)=I_n/I_n^2\neq 0$ (see Subsection \ref{subsec:appsusli}). 
\end{rem}
\subsection{The $K$-theory isomorphism conjecture with proper coefficients}

\begin{thm}\label{thm:main}
Let $G$ be a group, $\cF$ a family of subgroups of $G$, and $A$ a
$G$-ring. Assume that $\cF$ contains all the cyclic
subgroups, and that $A$ is proper over a locally finite
$(G,\cF)$-complex. Also assume that $A\otimes\Q$ is $K$-excisive. Then $H^G(-,K(A))$ sends
$(G,\cF)$-equivalences to rational equivalences. If moreover $A$ is a $\Q$-algebra, then
$H^G(-,K(A))$ sends $(G,\cF)$-equivalences to integral equivalences. In particular the assembly map
\[
H_*^G(\cE(G,\cF),K(A))\to K_*(A\rtimes G)
\]
is a rational isomorphism if $A$ is a $(G,\cF)$-proper ring, and an integral isomorphism if in addition $A$ is a $\Q$-algebra.
\end{thm}
\begin{proof}
By Theorem \ref{thm:propern}, $H^G(-,KH(A))$ maps $(G,\cF)$-equivalences to equivalences. Hence using the fibration
\[
K^{\nil}\to K\to KH
\]
we see that it suffices
to show that the statement of the theorem is true with $K^{\nil}$ substituted for $K$. Because the map
\eqref{map:q} is an equivalence for $\Q$-algebras, it suffices to prove that if $A$ is a $(G,\cF)$-proper ring, then $H^G(-,K^\nil(A))$ sends $(G,\cF)$-equivalences to rational equivalences. Consider the fibration
\[
K^{\ninf}\to K^{\nil}\otimes\Q\to \Omega^{-1}|HC(-\otimes\Q)|
\]
Because $\cF$ contains all cyclic subgroups and $A\otimes\Q$ is $H$-unital, $H^G(-,HC(A\otimes\Q))$ sends $(G,\cF)$-equivalences to equivalences, by Proposition \ref{prop:asshh} and Corollary \ref{cor:assnuni}. Similarly,
$H^G(-,K^{\ninf}A))$ sends $(G,\cF)$-equivalences to equivalences, by Theorem \ref{thm:propern} and Proposition \ref{prop:ninfexci}. It follows that the same is true of  $H^G(-,K^{\nil}(A)\otimes\Q)$. This completes the proof.
\end{proof}

\begin{ex}\label{ex:main}
If $X$ is a $(G,\cF)$-complex locally finite as a simplicial set and $B$ is $K$-excisive, then $\Z^{(X)}\otimes B$ is $(G,\cF)$-proper by Example \ref{ex:proper} and is $K$-excisive by Proposition \ref{prop:exiax}. If $T$ is the geometric realization of $X$ and $\F=\R,\C$, then the ring $\cC_{\rm comp}(T)$ of $\F$-valued compactly supported continuous functions is proper over $X$, again by Example \ref{ex:proper}, and therefore it $(G,\cF)$-proper. In fact the argument given in \ref{ex:proper} to show that $\Z^{(X)}\cdot\cC_{\rm comp}(T)=\cC_{\rm comp}(T)$ shows that $\cC_{\rm comp}(T)$ is $s$-unital and therefore $K$-excisive, by Example \ref{ex:sunihuni}. Hence $\cC_{\rm comp}(T)\otimes B$ is $K$-excisive if $B$ is, by Proposition \ref{prop:tenso}.
\end{ex}
\addtocounter{section}{-13}
\renewcommand{\thesection}{\Alph{section}}
\section{Appendix: $K$-excisive and $H$-unital rings}\label{app:A}

\subsection{The groups $\tor^{\tilde{A}}_*(-,A)$}\label{subsec:appsusli}

\bigskip

Let $M=\Z,\Z/n\Z, \Q$. Theorems of Suslin \cite{sus} (for $M=\Z,\Z/n\Z$) and Suslin-Wodzicki \cite{qs} (for $M=\Q$)
establish that a ring $A$ is excisive for $K$-theory with coefficients in $M$ if and only if
\[
\tor^{\tilde{A}}_*(M,A)=0
\]

\begin{ex}\label{ex:tfp}
A ring $A$ is said to have the (right) \emph{triple factorization property}
if for every finite family $a_1,\dots,a_n\in A$ there exist
$b_1,\dots,b_n,c,d\in A$ such that
\[
a_i=b_icd \text{ and } \{a\in A:ad=0\}=\{a\in A:acd=0\}
\]
It was proved in \cite[Theorem C]{qs} that rings having the triple
factorization property are $K$-excisive. In particular, rings with local units
are $K$-excisive.
\end{ex}

Let $M$ be an abelian group; regard $M$ as an $\tilde{A}$-module through the augmentation $\tilde{A}\to\Z$. We shall introduce a functorial abelian group $\bar{Q}(A,M)$ which computes $\tor_*^{\tilde{A}}(M,A)$. Consider
the functor $\perp:\tilde{A}-mod\to\tilde{A}-mod$,
\[
\perp N=\bigoplus_{x\in N}\tilde{A}.
\]
The functor $\perp$ is the free $\tilde{A}$-module cotriple \cite[8.6.6]{chubu}.
Let $Q(A)\to A$ be the canonical simplicial resolution by free
$\tilde{A}$-modules associated
to $\perp$ \cite[8.7.2]{chubu};
by definition, its $n$-th term is $Q_n(A)=\perp^{n+1} A$. Put
\[
\bar{Q}(A,M)=M\otimes_{\tilde{A}}Q(A).
\]
We have
\[
\pi_*(\bar{Q}(A,M))=\tor_*^{\tilde{A}}(M,A)
\]
We abbreviate $\bar{Q}(A)=\bar{Q}(A,\Z)$. Note that
\[
\bar{Q}(A,M)=M\otimes\bar{Q}(A)
\]
We have
\[
\bar{Q}_0(A)=\Z[A],\quad \bar{Q}_{n+1}(A)=\Z[Q_n(A)].
\]
\begin{lem}\label{lem:flatres}
Let $F\fibeq A$ be a simplicial resolution in $\ring$ and $M$ an abelian group. 
Let $\diag\bar{Q}(F)$ be the diagonal of the bisimplicial abelian group $\bar{Q}(F)$. Then
\[
\tor_*^{\tilde{A}}(M,A)=\pi_*(M\otimes\diag\bar{Q}(F))
\]
\end{lem}
\begin{proof} Because $F\to A$ is a simplicial resolution in $\ring$, $\bar{Q}_0(F)=\Z[F]\to \Z[A]=\bar{Q}_0(A)$ is a free simplicial resolution in $\ab$
of the free abelian group $\Z[A]$. Observe that if $G\to N$ is a free resolution of a free abelian group $N$, then $\tilde{A}\otimes G\to \tilde{A}\otimes N$
is a free simplicial $\tilde{A}$-module resolution, and $\Z[\tilde{A}\otimes G]\to\Z[\tilde{A}\otimes N]$ is a free simplicial $\Z$-module resolution.
Thus for each $n$, $\bar{Q}_n(F)\to\bar{Q}_n(A)$ is a free resolution of the free abelian group $\bar{Q}_n(A)$, and thus it remains a resolution after
tensoring by $M$. It follows that $M\otimes \diag\bar{Q}(F)$ computes $\tor_*^{\tilde{A}}(M,A)$.
\end{proof}

\begin{prop}\label{prop:specseq}
Let $F\fibeq A$ be a simplicial resolution and $M$ an abelian group.
Then there is a first quadrant spectral sequence
\[
E^2_{p,q}=\pi_q(\tor_p^{\tilde{F}}(M,F))\Rightarrow \tor_{p+q}^{\tilde{A}}(M,A)
\]
\end{prop}
\begin{proof}
This is just the spectral sequence of the bisimplicial
abelian group $([p],[q])\mapsto\bar{Q}_p(F_q,M)$.
\end{proof}
\begin{cor}\label{cor:freeres}
Let $F\fibeq A$ be free simplicial resolution in $\ring$. Then
\[
\pi_*(M\otimes (F/F^2))=\tor_*^{\tilde{A}}(M,A)
\]
\end{cor}
\begin{proof}
In view of the previous proposition, and of the fact that $\tor_0^{\tilde{B}}(M,B)=M\otimes B/B^2$ for every ring $B$, it suffices
to show that if $V$ is a free abelian group, and $TV$ the tensor algebra, then $\tor_n^{\tilde{TV}}(M,TV)=0$ for
$n\ge 1$. But this is clear, since $TV$ is free as a $\tilde{TV}$-module; indeed, the multiplication map $\tilde{TV}\otimes V\to TV$ is an isomorphism.
\end{proof}

\subsection{Bar complex}\label{subsec:barcomp}

\bigskip

Let $A$ be a ring. Consider the complex $P(A)$ given by $P_n(A)=\tilde{A}\otimes A^{\otimes n+1}$ $(n\ge 0)$, with boundary map
\[
b"(a_{-1}\otimes a_0\otimes a_1\otimes\dots\otimes a_n)=\sum_{i=-1}^{n-1}(-1)^ia_{-1}\otimes\dots\otimes a_ia_{i+1}\otimes\dots\otimes a_n
\]
The multiplication map $\mu:P_0(A)=\tilde{A}\otimes A\to A$ gives a surjective
quasi-isomorphism $\mu:P(A)\onto A$ \cite[8.6.12]{chubu}. A canonical $\Z$-linear section of $\mu$ is $j=1\otimes -:A\to \tilde{A}\otimes A$.
Let $\epsilon:\tilde{A}\to A$, $\epsilon(a,n)=a$. A $\Z$-linear homotopy $j\mu\to 1$ is defined by
\[
s:P_n(A)\to P_{n+1}(A),\quad s(a_{-1}\otimes\dots\otimes a_n)=1\otimes\epsilon(a_{-1})\otimes a_0\otimes\dots\otimes a_n
\]
Thus $P(A)$ is a resolution of $A$ by $\tilde{A}$-modules, and moreover these
$\tilde{A}$-modules are scalar extensions of $\Z$-modules. Put
\[
C^{bar}(A)=\Z\otimes_{\tilde{A}}P(A),\quad b'=\Z\otimes_{\tilde{A}}b"
\]
If
$A$ is flat as $\Z$-module, then $C^{bar}(A)$
 computes
$\tor^{\tilde{A}}_*(\Z,A)$ and $C^{bar}(A,M)=M\otimes C^{bar}(A)$ computes
$\tor^{\tilde{A}}_*(M,A)$. In general, the homology of $C^{bar}(A)$
can be interpreted as the $\tor$ groups relative to the extension
$\Z\to\tilde{A}$. For an arbitrary ring $A$, one can use the natural homotopy $s$ to give a natural
map
\[
Q(A)\to P(A)
\]
The induced map $M\otimes\bar{Q}(A)\to M\otimes C^{bar}(A)$ is a
quasi-homomorphism if $A$ is flat as a $\Z$-module.  In particular,
we have the following.

\begin{lem}\label{lem:barflatres}
Let $F\fibeq A$ be a simplicial resolution by flat rings, and $M$ an abelian group. Then
\[
\tor_*^{\tilde{A}}(M,A)=H_*(\tot (M\otimes C^{bar}(F)))
\]
\end{lem}
\subsection{$H$-unital rings}\label{subsec:huni}
A ring $A$ is called \emph{$H$-unital} if for every abelian group $V$, the complex $C^{bar}(A)\otimes V$
is acyclic.

\begin{rem}\label{rem:hunikexci}
Note that for $A$ flat as a $\Z$-module, $H$-unitality is equivalent to the acyclicity of $C^{bar}(A)$,
that is, to the vanishing of the groups $\tor_*^{\tilde{A}}(\Z,A)$. Thus for a flat ring $H$-unitality
equals $K$-excisiveness.
\end{rem}

\paragraph{\em Pure exact sequences}
Let
\begin{equation}\label{seq:pure}
0\to A\to B\to C\to 0
\end{equation}
be an exact sequence of rings. We say that \eqref{seq:pure} is \emph{pure} if for every abelian group $V$,
the sequence of abelian groups
\[
0\to A\otimes V\to B\otimes V\to C\otimes V\to 0
\]
is exact. Pure injective and pure surjective maps, and pure acyclic complexes are defined in the obvious way.
If $X(-)$ is a functorial chain complex, then we say that $A$ is \emph{pure $X$-excisive} if
for every pure exact sequence \eqref{seq:pure},
\[
X(A)\to X(B)\to X(C)
\]
is a distinguished triangle. The following theorem was proved by M. Wodzicki in \cite{wodex}.

\begin{thm}\label{thm:pure}(Wodzicki)
The following conditions are equivalent for a ring $A$.
\item[i)] $A$ is $H$-unital.
\item[ii)] $A$ is pure $C^{bar}$-excisive.
\item[iii)] $A$ is pure $HH$-excisive.
\item[iv)] $A$ is pure $HC$-excisive.
\end{thm}

\begin{ex}\label{ex:pure}
Any linearly split sequence \eqref{seq:pure} is pure. In particular, any sequence \eqref{seq:pure} with
$A$ a $\Q$-algebra is pure, since any $\Q$-vectorspace is injective as an abelian group. Thus for
a $\Q$-algebra $A$, Wodzicki's theorem remains valid if we omit the word ``pure" everywhere. Furthermore,
by the Suslin-Wodzicki theorem cited above, for $A$ a $\Q$-algebra the conditions of Theorem \ref{thm:pure} are also equivalent
to $A$ being $K^\Q$-excisive. In fact it is well-known that for a $\Q$-algebra $A$, being $K^\Q$-excisive is equivalent
to being $K$-excisive; as explained in \cite[Lemma 4.1]{kabi} this well-known fact follows from the main
result of \cite{chup}. See \cite[Lemma 1.9]{qs} for a different proof.
\end{ex}

\begin{ex}\label{ex:sunihuni}
Each $s$-unital ring is $H$-unital, by \cite[Cor. 4.5]{wodex}. Thus any
$s$-unital ring which is flat as a $\Z$-module is $K$-excisive, by Remark \ref{rem:hunikexci}.
\end{ex}

\subsection{Colimits}

\bigskip
The bar complex manifestly commutes with filtering colimits, and thus $H$-unital rings are closed under them.
The next proposition establishes the analogue of this property for $K$-excisive rings.

\begin{prop}\label{prop:filtcoli}
Let $\{A_i\}$ be a filtering system of rings, and let  $M$ be an
abelian group. Write $A=\colim A_i$. Then
\[
\tor_*^{\tilde{A}}(M,A)=\colim_i \tor_*^{\tilde{A_i}}(M,A_i)
\]
\end{prop}
\begin{proof}
Write $\perp:\rings\to \rings$, $\perp B=T(\Z[B])$ for the cotriple
associated with  the forgetful functor $\rings\to\sets$ and its
adjoint. Write $F(A)\fibeq A$ for the cotriple resolution
$F(A)_n=\perp^{n+1}A$ (\cite[\S 8/6]{chubu}). We have
$F(A)=\colim_iF(A_i)$. Thus $\tot (M\otimes C^{bar}F(A))=\colim_i
M\otimes C^{bar}F(A_i)$. Hence we are done by Lemma
\ref{lem:barflatres}.
\end{proof}

\begin{cor}\label{cor:filtcoli}
$K$-excisive rings are closed under filtering colimits.
\end{cor}

Let $M^0$ and $M^1$ be chain complexes of abelian groups, and let $f\in [1]^n$. Put
\[
T^f(M^0,M^1)=M^{f(1)}\otimes\dots\otimes M^{f(n)}
\]
Let
\[
M^0\star M^1=\bigoplus_{n\ge 0}\bigoplus_{f\in\map([n],[1])}T^f(M^0,M^1)
\]
\begin{lem}\label{lem:barsum}
Let $A$ and $B$ be rings. Then
\[
C^{bar}(A\oplus B)=(C^{bar}(A)[-1]\star C^{bar}(B)[-1])[+1]
\]
\end{lem}
\begin{proof} If $D$ is a ring then $C^{bar}(D)=T(D[-1])[+1]$ as graded abelian groups.
Hence for $\coprod$ the coproduct of rings, we have
\begin{align*}
C^{bar}(A\oplus B)=&T(A[-1]\oplus B[-1])[+1]\\
=&(T(A[-1])\coprod T(B[-1]))[+1]\\
=&(C^{bar}(A)[-1]\star C^{bar}(B)[-1])[+1]
\end{align*}
It is is straightforward to check that the identifications above are compatible with boundary maps.
\end{proof}

\begin{prop}\label{prop:barsum}
Let $\{A_i\}$ be a family of rings and $A=\bigoplus_i A_i$. Then $A$ is $K$-excisive if and only if each $A_i$ is,
and in that case $\bigoplus_i K(A_i)\to K(A)$ is an equivalence.
\end{prop}
\begin{proof} Let $B$ and $C$ be rings, and let $F\to B$ and $G\to C$ be free simplicial resolutions in $\ring$. Then $F\oplus G\to B\oplus C$ is a flat
simplicial resolution. Fix $q\ge 0$, and put $C^0=C^{bar}(F_q)$, $C^1=C^{bar}(G_q)$. Let $p\ge 1$, and $f\in[1]^p$.  Then by the K\"unneth formula
\begin{multline*}
H_n(T^f(C^0[-1],C^1[-1])[+1])=\\
T^f(H_*(C^0),H_*(C^1))_{n+1}=
\left\{\begin{matrix} T^f(F_q/F_q^2,G_q/G_q^2)& p=n+1\\
0& p\neq n+1\end{matrix}\right.
\end{multline*}
Hence the second page of the spectral sequence for the double complex of Lemma \ref{lem:barflatres} is
\[
E^2_{p,q}=\bigoplus_{f\in[1]^{p+1}}\pi_q(T^f(F/F^2,G/G^2))
\]
If $B$ and $C$ are $K$-excisive, we have $E^2=0$, by the Eilenberg-Zilber theorem and the K\"unneth formula, and thus $B\oplus C$ is again $K$-excisive. It follows from this
and from Proposition \ref{prop:filtcoli} that if $\{A_i\}$ is a family of $K$-excisive rings as in the proposition, then $A$ is $K$-excisive. If
$B$ and $C$ are arbitrary, then
\begin{gather*}
E^2_{0,q}=\tor_q^{\tilde{B}}(\Z,B)\oplus \tor_q^{\tilde{C}}(\Z,C)\\
E^2_{p,0}=\bigoplus_{f\in[1]^{p+1}}T^f(B/B^2,C/C^2)
\end{gather*}
Hence if $B\oplus C$ is excisive, $E^2_{*,0}=0$. It follows that $E^2_{0,1}=0$, and therefore $\pi_1(T^f(F/F^2,G/G^2))$ involves direct summands of tensor products of the form $E^2_{p,0}\otimes E^2_{0,1}$ and its symmetric, and both of these are zero. Thus $E^2_{*,1}=0$. A recursive argument shows that
$E^2=0$, whence both $B$ and $C$ are $K$-excisive. If now $A$ and $\{A_i\}$ are as in the proposition, $A$ is excisive, and $j\in I$, then setting
$B=A_j$ and $C=\bigoplus_{i\ne j}A_i$ above, we obtain that $A_j$ is $K$-excisive. The last assertion of the proposition
is well-known if each $A_i$ is unital. More generally, assume all $A_i$ are $K$-
excisive, and consider the exact sequence
\begin{equation}\label{seq:ai}
0\to A\to \bigoplus_i\tilde{A}_i\to \bigoplus_i\Z\to 0
\end{equation}
We have a commutative diagram with homotopy fibration rows
\[
\xymatrix{\bigoplus_iK(A_i)\ar[r]\ar[d] &\bigoplus_iK(\tilde{A}_i)\ar[d]\ar[r]&\bigoplus_iK(\Z)\ar[d]\\
K(A)\ar[r]&K(\bigoplus_i\tilde{A}_i)\ar[r]&K(\bigoplus_i\Z)}
\]
Because the middle and right vertical arrows are equivalences, it follows that the left one is an equivalence
too.
\end{proof}

\begin{prop}\label{prop:barsumh}
Let $\{A_i\}$ be a family of rings and $A=\bigoplus_i A_i$. Then $A$ is $H$-unital if and only if each $A_i$ is,
and in that case $\bigoplus_i HH(A_i)\to HH(A)$ and $\bigoplus_i HC(A_i)\to HC(A)$ are quasi-isomorphisms.
\end{prop}
\begin{proof}
The last assertion is proved by the same argument as its $K$-theoretic counterpart. By Theorem \ref{thm:pure} and Lemma \ref{lem:barsum},
if $B$ and $C$ are rings and $B$ is $H$-unital, then $C^{bar}(B\oplus C)\otimes V\to C^{bar}(C)\otimes V$
is a quasi-isomorphism for every abelian group $V$. Thus if also $C$ is $H$-unital, then so is $B\oplus C$.
Using this and the fact that $H$-unitality is preserved under filtering colimits, it follows that if
$\{A_i\}$ is a family of $H$-unital rings, then $A=\bigoplus_iA_i$ is $H$-unital. Suppose conversely that $A$
is $H$-unital, and consider the pure extension \eqref{seq:ai}. A similar argument as that of the proof of
Proposition \ref{prop:barsum} shows that $\bigoplus_iHH(A_i)\to HH(A)$ is a quasi-isomorphism. Next fix an
index $j$ and let
\[
0\to A_j\to B\to C\to 0
\]
be a pure extension. Then
\[
0\to A\to \bigoplus_{i\ne j}A_i\oplus\tilde{B}\to \bigoplus_{i\ne j}A_i\oplus\tilde{C}\to 0
\]
is a pure extension. Applying $HH$ yields a distinguished triangle quasi-isomorphic to
\[
\bigoplus_iHH(A_i)\to \bigoplus_{i\ne j}HH(A_i)\oplus HH(B)\oplus HH(\Z)\to \bigoplus_{i\ne j}HH(A_i)\oplus HH(C)\oplus HH(\Z)
\]
Removing summands, we obtain a triangle
\[
HH(A_j)\to HH(B)\to HH(C)
\]
We have shown that $A_j$ satisfies excision for pure extensions in Hochschild homology;
by Theorem \ref{thm:pure}, this implies that $A_j$ is $H$-unital.
\end{proof}

\subsection{Tensor products}

\bigskip
It was proved by Suslin and Wodzicki \cite[Theorem 7.10]{qs} that the tensor product of $H$-unital
rings is $H$-unital. Here we establish a weak analogue of this property for $K$-excisive rings.

\bigskip
Let $A$ be a ring. Put
\[
L_{-1}A=A,\qquad L_{n+1}A=\ker(A\otimes L_n(A)\overset{\mu}\to L_n(A)) \qquad (n\ge -1)
\]
Here $\mu$ is the multiplication map. 

\begin{lem}\label{lem:freetenso}
Let $A$ be a $K$-excisive ring, and $V$ an abelian group. Assume
both $A$ and $V$ are flat over $\Z$. Then $L_{n-1}A$ is flat as an abelian group and
\[
\tor^{\widetilde{A\otimes TV}}_n(\Z,A\otimes TV)=L_{n-1}A\otimes V^{\otimes n+1}\qquad (n\ge 0).
\]
\end{lem}
\begin{proof}
If $M$ is a left $A$-module such that
\begin{equation}\label{am=m}
A\cdot M=M,
\end{equation}
 and $L(M)=\ker(A\otimes M\to M)$ is the kernel of the multiplication map, then we have a short exact sequence
\[
0\to L(M)\otimes T^{\ge n+1}V\to \widetilde{A\otimes TV}\otimes M\otimes V^{\otimes n}\to M\otimes T^{\ge n}V\to 0
\]
By definition, $L_nA=L^{n+1}A$. By \cite[Theorem 7.8 and Lemma 7.6]{qs}, $M=L_nA$ satisfies \eqref{am=m} for all $n$,
and moreover, it is a flat abelian group, by induction.  Thus for $n\ge 1$, the sequence
\[
0\to L_{n-1}(M)\otimes T^{\ge n+1}V\to \widetilde{A\otimes TV}\otimes L_{n-2}M\otimes V^{\otimes n}\to L_{n-2}M\otimes T^{\ge n}V\to 0
\]
is exact. Hence
\begin{align*}
\tor_i^{\widetilde{A\otimes TV}}(\Z,A\otimes TV)=& \tor_i^{\widetilde{A\otimes TV}}(\Z,L_{-1}A\otimes T^{\ge 1}V)\\
=& \tor_0^{\widetilde{A\otimes TV}}(\Z,L_{i-1}A\otimes T^{\ge i+1}V)\\
=& L_{i-1}A\otimes V^{\otimes i+1}
\end{align*}
\end{proof}
\begin{prop}\label{prop:tenso}
Let $A$ and $B$ be $K$-excisive rings, at least one of them flat as a $\Z$-module. Then $A\otimes B$
is $K$-excisive.
\end{prop}
\begin{proof}
Assume $A$ is flat. Let $F\fibeq B$ be a simplicial resolution by free rings. Then $A\otimes F\fibeq A\otimes B$ is a resolution by flat rings. By Lemma \ref{lem:freetenso}, the second page of the spectral sequence of Proposition \ref{prop:specseq} is
\[
E^2_{p,q}=\pi_q(L_{p-1}A\otimes (F/F^2)^{\otimes p+1})=L_{p-1}A\otimes\pi_q((F/F^2)^{\otimes p+1})
\]
which equals zero by Corollary \ref{cor:freeres} and the K\"unneth formula, since $B$ is $K$-excisive by assumption,
and $L_{p-1}A$ is flat by Lemma \ref{lem:freetenso}.
\end{proof}

\subsection{Crossed products}

\bigskip

Let $G$ be a group and $\pi:\Z[G]\to \Z$ the augmentation $g\mapsto
1$. Put
\[
JG=\ker\pi
\]
\begin{lem}\label{lem:freecross} Let $V$ be a $\Z[G]$-module,
free as an abelian group. Then
\[
\tor^{\widetilde{TV\rtimes G}}_n(\Z,TV\rtimes G)=V^{\otimes
n+1}\otimes JG^{\otimes n}\otimes\Z[G] \qquad n\ge 0
\]
\end{lem}
\begin{proof}
Note that the subset
\[
V^{\otimes n}\oplus TV^{\ge n+1}\rtimes G\subset TV\rtimes G
\]
is a left ideal, and that the map
\begin{equation}\label{map:isog}
\begin{array}{c}
\widetilde{TV\rtimes G}\otimes V^{\otimes n}\to V^{\otimes n}\oplus TV^{\ge n+1}\rtimes G\\
1\otimes y\mapsto y\qquad\\
x\rtimes g\otimes y\mapsto xg(y)\rtimes g\\
\end{array}
\end{equation}
is a $\widetilde{TV\rtimes G}$-module isomorphism. Let $M$ be a $\Z[G]$-module. Consider the map
\[
V^{\otimes n}\otimes M\oplus (TV^{\ge n+1}\rtimes G)\otimes M\to TV^{\ge n}V\otimes M,\quad
(x,(y\rtimes g)\otimes m)\mapsto x+y\otimes gm
\]
Tensoring the isomorphism \eqref{map:isog} with $M$ and composing, we obtain a $\Z$-split surjective homomorphism of
$\widetilde{TV\rtimes G}$-modules
\[
\widetilde{TV\rtimes G}\otimes V^{\otimes n}\otimes M\onto TV^{\ge n}\otimes M
\]
This map fits in an exact sequence
\[
0\to T^{\ge n+1}V\otimes JG\otimes M\to \widetilde{TV\rtimes G}\otimes V^{\otimes n}\otimes M\to T^{\ge n}V\otimes M\to 0
\]
If $M$ is flat as an abelian group, then the middle term in the exact sequence above is a flat
$\widetilde{TV\rtimes G}$-module. Applying this successively, starting with $M=\Z[G]$, we obtain
\begin{align*}
\tor_n^{\widetilde{TV\rtimes G}}(\Z,TV\rtimes G)=&
\tor_0^{\widetilde{TV\rtimes G}}(\Z,TV^{\ge n+1}\otimes JG^{\otimes n}\otimes \Z[G])\\
=&V^{\otimes n+1}\otimes JG^{\otimes n}\otimes \Z[G]
\end{align*}
\end{proof}
\begin{prop}\label{prop:crossbar}
Let $G$ be a group and $A\in G-\ring$. Assume $A$ is $K$-excisive. Then $A\rtimes G$ is $K$-excisive.
\end{prop}
\begin{proof}
Note that the forgetful functor from $G-\ring$ to sets has a left adjoint; namely
$X\mapsto T(\Z[G\times X])$. Hence $A$ admits a free resolution $F\fibeq A$ such that each $F_n$ is a $G$-ring; for example we may take the cotriple resolution associated to the adjoint pair just described. Since $F$ is a simplicial $G$-ring, we can take its crossed product with $G$, to obtain a $\Z$-flat resolution $F\rtimes G\fibeq A\rtimes G$. Now proceed as in the proof of Proposition \ref{prop:tenso}, using Lemma \ref{lem:freecross}.
\end{proof}

\begin{prop}\label{prop:crossh}
Let $G$ be a group and $A\in G-\ring$. Assume $A$ is $H$-unital. Then $A\rtimes G$ is $H$-unital.
\end{prop}
\begin{proof} The bar resolution $E(G,M)$ (\cite[\S6.5]{chubu}) is functorial on the $G$-module $M$. Applying it
dimensionwise to $C^{bar}(A)$, we obtain a simplicial chain
complex\\
$E(G,C^{bar}(A))$. We may view the latter as a double chain complex
with $A^{\otimes q+1}\otimes\Z[G^{p+1}]$ in the $(p,q)$ spot.
Removing the first row and the first column yields a double complex
whose total chain complex we shall call $M[-1]$. Note $M$ is a chain
complex of $A\rtimes G$-modules and homomorphisms. We have $M_0\cong
(A\rtimes G)^{\otimes 2}$, and the multiplication map $(A\rtimes
G)^{\otimes 2}\to A\rtimes G$ induces a surjection onto the kernel
$L$ of the augmentation $A\rtimes G\to A$, $a\rtimes g\to a$. Note
that the hypothesis that $A$ is $H$-unital implies that the
augmented complex
\begin{equation}\label{seq:pseudo}
\dots\to M_1\to M_0\to L
\end{equation}
 is pure acyclic. Since each $M_n$ is extended, \eqref{seq:pseudo} is a pure pseudo-free resolution
 in the terminology of \cite[7.7]{qs}. On the other hand, because $A$ is $H$-unital, the multiplication map
 $\mu:A^{\otimes 2}\to A$ is pure surjective; thus $\mu\circ (id\otimes g)$ is pure surjective for each $g\in G$.
 It follows from this that the multiplication map $(A\rtimes G)^{\otimes 2}\to A\rtimes G$ is pure surjective. We
 have shown that $A\rtimes G$ satisfies condition d) of \cite[Theorem 7.8]{qs}, which by \emph{loc. cit.} implies
 that $A\rtimes G$ is $H$-unital.
\end{proof}

\bibliographystyle{plain}

\begin{thebibliography}{10}
\bibitem{kabi}
G.~Corti\~nas.
\newblock The obstruction to excision in {$K$}-theory and in cyclic homology.
\newblock{\em Invent. Math.} 164:143--173, 2006.

\bibitem{friendly}
G. ~Corti\~nas.
\newblock{Algebraic v. topological K-theory: a friendly match.}
\newblock{\emph{Topics in algebraic and topological K-theory}, 103--165}
\newblock{Lecture Notes in Mathematics 2008, Springer, Berlin, 2011.}

\bibitem{CT}
G.~Corti\~nas and A.~Thom.
\newblock {Bivariant algebraic $K$-theory}.
\newblock {\em Journal f{\"u}r die Reine und Angewandte Mathematik (Crelle's
  Journal)}, 610:71--123, 2007.

\bibitem{cunc}
J.~Cuntz.
\newblock{Noncommutative simplicial complexes and the {B}aum-{C}onnes
conjecture.}
\newblock{\em Geom. Funct. Anal.},
\newblock{12:307--329, 2002.}

\bibitem{cmr}
J.~Cuntz, R.~Meyer, J.~Rosenberg.
\newblock{\em Topological and Bivariant K-Theory.}
Birkh\"auser, Oberwolfach Seminars, vol. 36, 2007, 262 pages.

\bibitem{dl}
J.~Davis and W.~L{\"u}ck.
\newblock {Spaces over a category and assembly maps in isomorphism conjectures
  in $K$- and $L$-theory}.
\newblock {\em $K$-theory}, 15:241--291, 1998.

\bibitem{FTA}
B.L. Fe\u\i gin and B.L. Tsygan.
\newblock{Additive $K$-theory}
\newblock{In {\em $K$-theory, arithmetic and geometry (Moscow,
1984--1986)}, 67--209,}
\newblock{Lecture Notes in Math., 1289, Springer, Berlin, 1987.}

\bibitem{gw}
S.~Geller, C.~Weibel.
\newblock{Hochschild and cyclic homology are far from being homotopy
 functors.}
\newblock{\emph{Proc. Amer. Math. Soc.}  106(1):49--57, 1989.}

\bibitem{ght}
E.~Guentner, N.~Higson, J~Trout.
\newblock{Equivariant $E$-theory for $C\sp *$-algebras.}
\newblock{Mem. Amer. Math. Soc. 148 (2000), no. 703, viii+86 pp.}

\bibitem{hk}
N.~Higson, G.~Kasparov.
\newblock{$E$-theory and $KK$-theory for groups which act properly and
 isometrically on Hilbert space.}
\newblock{\emph{Invent. Math.}  144(1):23--74, 2001.}

\bibitem{hirsch}
P.~Hirschhorn.
\newblock{ Model categories and their localizations}.
\newblock{Math. Surveys and Monographs ~99},
\newblock{Amer. Math. Soc., Providence, RI, ~2003}.

\bibitem{hov}
M. ~Hovey.
\newblock{Model categories.}
\newblock{Mathematical Surveys and Monographs, ~63}.
\newblock{American Mathematical Society, Providence, RI, ~1999.}

\bibitem{joach}
M. Joachim.
\newblock{$K$-homology of $C^*$-categories and symmetric spectra representing
$K$-homology.}
\newblock{\em Math. Ann.}
\newblock{327 (2003), 641--670.}

\bibitem{karder}
M.Karoubi.
\newblock{Foncteurs deriv\'ees et K-th\'eorie.}
\newblock{In {\em S\'eminaire Heidelberg-Saarbrucken-Strasbourg (1967/68)}}.
\newblock{Lecture Notes in Math. 136, 107--186 (1970).}

\bibitem{kv}
M. Karoubi, O.E. Villamayor.
\newblock{$K$-th\'eorie alg\'ebrique et $K$-th\'eorie topologique. I.}
\newblock{\em Math. Scand.}
\newblock{ 28 (1971), 265--307 (1972).}

\bibitem{lod}
J. L. Loday.
\newblock{\em Cyclic homology.}
\newblock{Grundlehren der Mathematischen Wissenschaften, 301.}
\newblock{Springer-Verlag, Berlin, 1992.}



\bibitem{lor}
M. Lorenz.
\newblock{On the homology of graded algebras.}
\newblock{\em Comm. Algebra.}
\newblock{ 20, (1992) 489--507.}

\bibitem{LR1}
W.~L{\"u}ck, H.~Reich. \newblock{ Detecting $K$-theory by cyclic
homology.} \newblock{\em Proc. London Math. Soc.}   93(3):593--634,
2006.

\bibitem{ralf} R. Meyer
\newblock{Universal coefficient theorems and assembly maps in KK-theory.}
\newblock{\em Topics in algebraic and topological K-theory,}
\newblock{45--102, Lecture Notes in Math., 2008, Springer, Berlin,  2011. }
\bibitem{mn1} R. Meyer, R. Nest.
\newblock{The Baum-Connes conjecture via localization of categories.}
 \newblock{\em Lett. Math. Phys.} 69:237--263, 2004.

\bibitem{mn2}  R. Meyer, R. Nest.
\newblock{The Baum-Connes conjecture via localisation of categories.}
\newblock{\em Topology}  45(2):209--259, 2006.

\bibitem{machc} R. McCarthy.
\newblock{The cyclic homology of an exact category.}
\newblock{\em J. Pure Appl. Algebra}
\newblock{ 93 (1994), no. 3, 251--296.}

\bibitem{nob}
 G. N\"obeling.
\newblock{Verallgemeinerung eines Satzes von Herrn E. Specker}.
\newblock {\em Inventiones Mathematicae}
\newblock{ 6 (1968), 41--55.}

\bibitem{pw2}
E. Pedersen,  C. Weibel.
\newblock{$K$-theory homology of spaces.}
\newblock{Algebraic topology (Arcata, CA, 1986),
 346--361},\newblock{{\em Lecture Notes in Math.}, 1370, Springer, Berlin,  1989}.

\bibitem{sus}
A. Suslin.
\newblock Excision in the integral algebraic $K$-theory.
\newblock {\em Proceedings of the Steklov Institute of Mathematics} 208:255--279, 1995.

\bibitem{qs}
A.~Suslin, M.~Wodzicki.
\newblock Excision in algebraic $K$-theory.
\newblock {\em Ann. of Math.} 136:51--122, 1992.

\bibitem{chup} C. Weibel.
\newblock {Mayer-Vietoris sequences and mod $p$ $K$-theory.}
\newblock{\em Springer Lect. Notes Math.}966:390--406, 1982.

\bibitem{wenil} C. Weibel,
\newblock Nil $K$-theory maps to cyclic homology.
\newblock \emph{Trans. Amer. Math. Soc.}  303  (1987),  no. 2, 541--558.


\bibitem{chubu}
C. Weibel.
\newblock {\em An introduction to homological algebra},
\newblock Cambridge Univ.\ Press, 1994.

\bibitem{kh}
C. Weibel. {\it Homotopy Algebraic $K$-theory}. Contemporary Math. {\bf 83} (1989)
461--488.

\bibitem{wodex}
M. Wodzicki.
\newblock Excision in cyclic homology and in rational algebraic $K$-Theory.
\newblock {\em Ann. of Math.} 129:591--639, 1989.

\end{thebibliography}

\end{document}